\documentclass[a4paper,12pt,reqno]{amsart}
\usepackage[utf8]{inputenc}
\usepackage[T1]{fontenc}
\usepackage{lmodern}
\usepackage[english]{babel}
\usepackage{microtype}
\usepackage{comment}
\usepackage{todonotes}
  
\usepackage{amsmath,amssymb,amsfonts,amsthm,esint}
\usepackage{mathtools,accents}
\usepackage{mathrsfs}
\usepackage{aliascnt}
\usepackage{braket}
\usepackage{bm}
\usepackage{stmaryrd}
\usepackage{tikz-cd}
\newdimen\vintkern\vintkern11pt
\def\vint{-\kern-\vintkern\int}

\usepackage[a4paper,margin=3cm]{geometry}
\usepackage[citecolor=blue,colorlinks]{hyperref}

\usepackage{enumerate}
\usepackage{xcolor}
\usepackage{graphicx}
\usepackage{mathbbol}
\usepackage{soul}

\mathtoolsset{showonlyrefs}

\makeatletter
\g@addto@macro\@floatboxreset\centering
\makeatother

\makeatletter
\def\newaliasedtheorem#1[#2]#3{
  \newaliascnt{#1@alt}{#2}
  \newtheorem{#1}[#1@alt]{#3}
  \expandafter\newcommand\csname #1@altname\endcsname{#3}
}
\makeatother

\numberwithin{equation}{section}

\newtheoremstyle{slanted}{\topsep}{\topsep}{\slshape}{}{\bfseries}{.}{.5em}{}

\theoremstyle{plain}
\newtheorem{theorem}{Theorem}[section]
\newaliasedtheorem{proposition}[theorem]{Proposition}
\newaliasedtheorem{question}[theorem]{Question}
\newaliasedtheorem{lemma}[theorem]{Lemma}
\newaliasedtheorem{corollary}[theorem]{Corollary}
\newaliasedtheorem{counterexample}[theorem]{Counterexample}

\theoremstyle{definition}
\newaliasedtheorem{definition}[theorem]{Definition}
\newaliasedtheorem{example}[theorem]{Example}
\newaliasedtheorem{conjecture}[theorem]{Conjecture}

\theoremstyle{remark}
\newaliasedtheorem{remark}[theorem]{Remark}
\newcommand{\setN}{\mathbb{N}}

\newcommand{\setR}{\mathbb{R}}

\newcommand{\br}[1]{\llbracket #1 \rrbracket}

\let\phi\varphi

\newcommand{\norm}[1]{\left\lVert#1\right\rVert}

\DeclareMathOperator{\Hess}{Hess}

\newcommand{\loc}{{\rm loc}}

\DeclareMathOperator{\Lip}{Lip}

\DeclareMathOperator{\lip}{lip} % %for the slope

\newcommand{\haus}{\mathscr{H}}
\newcommand{\Leb}{\mathscr{L}}

\newcommand{\dist}{\mathsf{d}}

\newcommand{\meas}{\mathfrak{m}}

\DeclareMathOperator{\RCD}{RCD}

\newfont{\tmpf}{cmsy10 scaled 2500}

\newcommand{\GCBA}{{\rm GCBA}}
\newcommand{\CBA}{{\rm CBA}}
\newcommand{\CAT}{{\rm CAT}}

\newcommand{\rstr}{\:\mbox{\rule{0.1ex}{1.2ex}\rule{1.1ex}{0.1ex}}\:}
\def\set{{\rm set}}

\newcommand{\spt}{{\rm spt}}
\newcommand{\sgn}{{\rm sgn}}

\newcommand{\bM}{\operatorname{\mathbf M}}
\newcommand{\bN}{\operatorname{\mathbf N}}
\newcommand{\bI}{\operatorname{\mathbf I}}

\def\Xint#1{\mathchoice
{\XXint\displaystyle\textstyle{#1}}%
{\XXint\textstyle\scriptstyle{#1}}%
{\XXint\scriptstyle\scriptscriptstyle{#1}}%
{\XXint\scriptscriptstyle\scriptscriptstyle{#1}}%
\!\int}
\def\XXint#1#2#3{{\setbox0=\hbox{$#1{#2#3}{\int}$ }
\vcenter{\hbox{$#2#3$ }}\kern-.6\wd0}}

\def\dashint{\Xint-}

\begin{document}

\author[N. Cavallucci]{Nicola Cavallucci}
\address[Nicola Cavallucci]{Département de mathématiques,
Université de Fribourg,
Ch. du musée 23
CH-1700 Fribourg,
Switzerland}\email{n.cavallucci23@gmail.com}

\author[D. Marti]{Denis Marti}
\thanks{D. Marti was funded by the Swiss National Science Foundation grant 212867}
\address[Denis Marti]{Département de mathématiques,
Université de Fribourg,
Ch. du musée 23
CH-1700 Fribourg,
Switzerland}\email{denis.marti@unifr.ch}

\author[A. Mondino]{Andrea Mondino}
\thanks{A. Mondino was funded by the European Research  Council
(ERC), under the European Union Horizon 2020 research and innovation programme, via the ERC Starting Grant “CURVATURE”, grant agreement No. 802689.}
\address[A. Mondino]{Mathematical Institute, University of Oxford, United Kingdom}
\email{andrea.mondino@maths.ox.ac.uk}

\author[R. Perales]{Raquel Perales}
\thanks{R. Perales was partially funded by 
the Austrian Science Fund (FWF) [Grant DOI: 10.55776/EFP6]}
\address[R. Perales]{Centro de Investigaci\'on en Matem\'aticas, 
De Jalisco s/n, Valenciana, Guanajuato, Gto. Mexico. 36023}
\email{raquel.perales@cimat.mx}

\title[On the orientability of GCBA spaces]{Orientability of metric manifolds and metric spaces with curvature bounds}

\begin{abstract}
We develop a theory of orientability for metric manifolds and singular
metric spaces using locally integral currents. For topological
$n$-manifolds with locally finite Hausdorff measure, we introduce local
degree and measure-theoretic conditions under which topological
orientations correspond canonically to boundaryless, locally integer
rectifiable $n$-currents of multiplicity one on the rectifiable part.
We extend these currents to a class of singular spaces whose manifold
part is dense and of full measure, and establish constancy and
top-dimensional homology results under local uniqueness of orientation.
For orientable non-collapsed $\RCD(\kappa,n)$ spaces, we prove a
current-theoretic Stokes theorem: the boundary of the orientation
current is the induced orientation current of the geometric boundary,
with multiplicity one. In particular, its mass measure is precisely
the $(n-1)$-dimensional Hausdorff measure on that boundary. For purely
$n$-dimensional, locally geodesically complete spaces with local upper
curvature bounds, we show that local uniqueness of orientation excludes
codimension-one branching and implies a local $1$-Poincar\'e inequality.
The latter result applies beyond homology manifolds and extends the
connection between quantitative topology and Poincar\'e inequalities
to this class of singular spaces.
\end{abstract}

\maketitle
{
\hypersetup{linkcolor=blue}
\setcounter{tocdepth}{2}
\tableofcontents
}

\section{Introduction}

An orientation on a smooth manifold determines an integral current, and
Stokes' theorem identifies the boundary of this current with integration over
the oriented boundary. These familiar statements link topology, measure and
first-order analysis on smooth manifolds. For a topological manifold equipped with a general
compatible metric, however, this link is no longer automatic: topological
coordinate maps need not be Lipschitz, and local finiteness of the Hausdorff
measure does not imply rectifiability. On a singular space there is a further
question. Even when its manifold part admits an integral  current, the
extension of this current across the singular set may acquire a boundary,
and that boundary need not have locally finite mass.

In this paper we develop a theory of orientation by metric currents that
addresses these questions in a common framework. We first characterize
topological orientability for a class of metric manifolds defined by local
degree and measure-theoretic conditions. We then extend the construction to
spaces with a dense manifold part of full measure and study the resulting
top-dimensional homology. The principal geometric applications concern
non-collapsed spaces with Ricci curvature bounded below and geodesically
complete spaces with curvature bounded above. In the former setting, we
identify the current boundary of an orientation with the orientation of the
geometric boundary, with multiplicity one. In the latter, local uniqueness
of orientation excludes codimension-one branching and yields a local
$1$-Poincar\'e inequality, including examples that are not homology manifolds.

\subsection{Orientations of metric manifolds}

The theories of metric currents developed by Ambrosio--Kirchheim
\cite{AmbrosioKirchheim00} and Lang \cite{Lang} provide a notion of integration
and boundary that does not require a differentiable structure. A central
problem is to determine when this analytic theory detects the topological
fundamental class. Basso--Marti--Wenger \cite{BassoMartiWenger25} constructed
metric fundamental classes for closed, oriented metric manifolds with finite
Hausdorff measure under linear local contractibility, and used them to prove
relative isoperimetric and Poincar\'e inequalities; see also \cite{Heinonen-Keith,heinonen-rickman02,heinonen-sullivan02} for earlier work in this direction. Their framework initially
uses a smooth reference manifold and also admits a Lipschitz reference
manifold. Marti--Soultanis \cite{marti2024characterization} subsequently
related the existence of a metric fundamental class to local index bounds
and approximation by bi-Lipschitz manifolds. Marti \cite{Marti2026} extended
the fundamental class theory to manifolds with boundary and to non-orientable
manifolds, with coefficients modulo two in the latter case.

Our starting point is a local construction that does not assume a global
smooth or Lipschitz manifold structure. Throughout, a metric $n$-manifold is
a second-countable topological $n$-manifold without boundary, equipped with
a compatible metric for which $\haus^n$ is locally finite. We call it
\emph{degree-regular} if porous subsets are $\haus^n$-negligible and the
following local-degree condition holds: whenever $K\subseteq X$ is compact
and $\rho:K\to\setR^n$ is bi-Lipschitz, a Lipschitz extension
$\widetilde\rho:X\to\setR^n$ has local degree $\pm1$ at
$\haus^n$-almost every point of $K$; see
Definition~\ref{def: good-metric-mfld}. Under the porosity assumption this
condition is independent of the extension. It holds for linearly locally
contractible metric manifolds and, more generally, for the Lipschitz
homeomorphic charts with an almost-everywhere bi-Lipschitz decomposition
that we call ETR, cf.\ \cite{MP26}, for ``essentially topologically regular''.

Let $E$ be the rectifiable part of $X$, defined up to a
$\haus^n$-negligible set, and let $\lambda_E$ denote its tangential area
factor in the mass representation of integer rectifiable currents. Thus
$\lambda_E=1$ wherever the approximate tangent norm is Euclidean. We use
$\mathcal I_{n,\loc}(X)$ for locally integer rectifiable currents and
$\mathbf I_{n,\loc}(X)$ for locally integral currents.

\begin{theorem}[Orientability of metric manifolds]
\label{thm:intro_metric_manifolds}
Let $(X,\dist)$ be a degree-regular metric $n$-manifold. Then $X$ is
 topologically orientable if and only if there exists
$T\in\mathcal I_{n,\loc}(X)$ such that
\[
 \partial T=0,
 \qquad
 \|T\|=\lambda_E\haus^n\rstr E.
\]
Each topological orientation determines a unique such current. If $X$ is
connected and oriented, every locally integral $n$-cycle is an integer
multiple of this current.
\end{theorem}

This combines Theorem~\ref{thme: characterization-orientability-metric-manifold},
its more precise form in
Theorem~\ref{theo: characterization-orientability-metric-manifold-with-details},
and Corollary~\ref{cor: any-bdry-less-current-implies-orient}.
No completeness, compactness or a priori rectifiability of $X$ is assumed.
The mass is prescribed on the rectifiable part, while the support of the
current is all of $X$. In particular, the assertion distinguishes full
support from full-measure rectifiability.

The construction is explicit on rectifiable charts. If
$\rho_i:K_i\to\setR^n$ are bi-Lipschitz charts with pairwise disjoint domains covering the rectifiable part
of an oriented regular disk $D$, then
\[
 T_D^{\mathfrak o}
 =\sum_i(\rho_i^{-1})_\#
   \br{\deg(\widetilde\rho_i,\rho_i^{-1}(\cdot))},
\]
where the multiplicities are extended by zero outside $\rho_i(K_i)$.
The local degree records the topological orientation, and the porosity
condition ensures compatibility on overlapping rectifiable charts.
Bate's perturbation theorem \cite{bate-perturb,bate24} allows us to control
the purely unrectifiable part when testing these currents. Doubling a
regular disk and using additivity of degree then shows that the current
has no boundary in its interior. These local currents agree on overlaps
and therefore define the global orientation.

\subsection{Extension across singular sets}

We call a separable, locally compact metric space $X$ \emph{$n$-regular} if
$\haus^n$ is locally finite, its topological manifold part $\mathcal M(X)$
is a degree-regular metric $n$-manifold, and
\[
 \overline{\mathcal M(X)}=X,
 \qquad
 \haus^n(\mathcal S(X))=0,
 \qquad
 \mathcal S(X):=X\setminus\mathcal M(X).
\]
Here $n$-regularity is a property of the whole space; it should not be
confused with the notion of an $n$-regular point in Cheeger-Colding's theory of Ricci-limit spaces \cite{CC:97, CC:II, CC:III}.

Every orientation of $\mathcal M(X)$ extends uniquely to a locally integer
rectifiable current on $X$, with the same mass measure and with boundary
supported in $\mathcal S(X)$; see Corollary~\ref{cor:extension_of_T_X}.
We call this extension an \emph{orientation} of $X$. The extension is unique for
a prescribed orientation of the manifold part, but different components
of that part can carry independent choices of sign. This distinction is
essential at branching singularities.

We say that $X$ is \emph{uniquely orientable} if it has exactly two
orientations, and \emph{locally uniquely orientable} (LUO) if all sufficiently
small balls about each point are uniquely orientable. Unique orientability
is equivalent to connectedness and orientability of $\mathcal M(X)$.
For connected LUO spaces the manifold part is connected, and, if
$\haus^{n-1}(\mathcal S(X))=0$, we recover the familiar dichotomy
\[
 H_n^{\mathrm{IC}}(X)\cong
 \begin{cases}
 \mathbb Z,& X\text{ is orientable},\\
 \{0\},& X\text{ is not orientable}.
 \end{cases}
\]
Here and throughout, $H_\bullet^{\mathrm{IC}}$ is formed from
\emph{locally integral currents}, with no compact-support condition.
Its topological analogue is locally finite, rather than ordinary,
singular homology. We also establish constancy results on connected
open subsets and construct a branched orientation double cover when
$(\mathcal M(X),\dist)$ is a non-orientable length space.

Two boundary criteria underlie these applications. A locally finite-mass
$n$-current has zero boundary if the support of its boundary is
$\haus^{n-1}$-negligible. When the boundary support is not negligible,
suitable upper $n$-density bounds for the mass, together with local
finiteness and positive lower $(n-1)$-density of its Hausdorff measure,
imply local finiteness of the boundary mass; see
Theorems~\ref{theo:boundary_zero_mass} and
\ref{theo:boundary_finite_mass_hard}. These criteria let us pass from
the topology of the manifold part to integral current statements on
the entire space.

\subsection{Non-collapsed spaces with Ricci curvature bounded below}

The structure theory of non-collapsed Ricci limit spaces, initiated by
Cheeger--Colding \cite{CC:97,CC:II,CC:III}, and its development for
non-collapsed $\RCD$ spaces \cite{MondinoNaber19, DPG, KapMonGT} provide a
natural setting for this construction.

Orientability in Ricci-limit geometry was studied by
Matveev--Portegies \cite{MatveevPortegies17} through flat
limits of the integral currents associated with closed
orientable manifolds, and by Honda \cite{Honda17} through
differential forms with Sobolev regularity.
For non-collapsed $\RCD$ spaces without boundary,
Brena--Bru\`e--Pigati \cite{BrenaBruePigati2024} established
the equivalence between topological orientability of the
manifold part, Honda's formulation via differential forms,
and a characterization in terms of metric currents.
As a geometric application, they proved that four-manifolds
with nonnegative Ricci curvature and Euclidean volume
growth are orientable.
The current-theoretic approach of Mondino--Perales
\cite{MP26} further connects orientability with local flat
convergence for possibly noncompact non-collapsed
$\RCD(\kappa,n)$ spaces and Kato limit spaces, both without boundary.
For closed, orientable, Alexandrov
spaces with a lower curvature bound, integral current structures were
constructed by Jaramillo--Perales--Rajan--Searle--Siffert \cite{JPRSS21}.

Our results include non-collapsed $\RCD(\kappa,n)$ spaces \emph{with nonempty
boundary}. We use the notion of boundary proposed in \cite{DPG}, namely
\[
 \partial X
 :=\overline{\mathcal S^{n-1}(X)\setminus\mathcal S^{n-2}(X)},
\]
where $\mathcal S^k(X)$ denotes the usual metric singular stratification.
Building on the boundary regularity theory of Bru\`e--Naber--Semola
\cite{BrueNaberSemola20}, we show that $X$ is $n$-regular and that
$(\partial X,\dist)$, with the restricted ambient distance, is
$(n-1)$-regular. The approximate tangent norms are Euclidean almost
everywhere in both cases. A point that requires particular care is the
boundary chart: its almost-everywhere bi-Lipschitz decomposition must
hold for the same Lipschitz homeomorphism used to describe the topology.
We establish this in Lemma~\ref{lem:ETR-boundary-RCD}.

\begin{theorem}[Orientation and boundary of non-collapsed $\RCD$ spaces]
\label{thm:intro_RCD_Stokes}
Let $(X,\dist,\haus^n)$ be a non-collapsed $\RCD(\kappa,n)$ space,
with $n\geq2$. If $\mathcal M(X)$ is orientable, then its orientation
extends to a current $T\in\mathbf I_{n,\loc}(X)$, unique up to sign,
which satisfies
\[
 \|T\|=\haus^n,
 \qquad
 \spt(\partial T)=\partial X,
 \qquad
 \|\partial T\|=\haus^{n-1}\rstr\partial X.
\]
Moreover, $\partial X$ is orientable and
$(\partial T)|_{\partial X}$ is its induced orientation current.
\end{theorem}

See Proposition~\ref{prop:RCD_orientable_implies_uniquely_orientable} and
Theorem~\ref{theo:RCD-boundary-mass}. The boundary identity gives a
current-theoretic Stokes theorem: it identifies both the support and
the multiplicity of the boundary, rather than only its concentration
on the geometric boundary. The proof first establishes local finiteness
of the boundary mass and then uses degree-one half-space charts to
identify the boundary multiplicity.

The resulting top-dimensional absolute and relative homology groups are
those expected for an oriented manifold with boundary. In particular,
for orientable $X$,
\[
 H_n^{\mathrm{IC}}(X,\partial X)\cong\mathbb Z,
 \qquad
 H_n^{\mathrm{IC}}(X)\cong
 \begin{cases}
 \mathbb Z,&\partial X=\varnothing,\\
 \{0\},&\partial X\neq\varnothing.
 \end{cases}
\]
Both groups vanish when $X$ is non-orientable. As a further consequence,
there is no proper Lipschitz retraction of an orientable $X$ onto its
nonempty boundary. For compact $X$, properness is automatic. In the
noncompact case it is essential: a Euclidean half-space retracts onto
its boundary by a Lipschitz projection, which is not proper.

\subsection{Upper curvature bounds and Poincar\'e inequalities}

Our second application concerns purely $n$-dimensional $\GCBA$ spaces,
that is, separable, locally compact, locally geodesically complete spaces
with local upper curvature bounds. The structure theory of
Lytchak--Nagano \cite{LytchakNagano2019_GCBA,LN:JEMS22} shows that these
spaces have a large manifold part and controlled singular strata.
We prove that they are $n$-regular and that every orientation has
locally finite boundary mass. In contrast with the $\RCD$ case,
the boundary of an orientation can exhibit branching between several
manifold sheets. It is concentrated on the set of \emph{wall points},
where the tangent splits an $\setR^{n-1}$ factor but is not $\setR^n$.
Using Nagano's analysis of wall singularities \cite{Nagano2026}, we
show that local unique orientability excludes wall points.

This yields an analytic consequence beyond the manifold setting.

\begin{theorem}[Local Poincar\'e inequality]
\label{thm:intro_GCBA_PI}
Let $(X,\dist)$ be a purely $n$-dimensional, locally uniquely orientable
$\GCBA$ space. Then $(X,\dist,\haus^n)$ is locally doubling and
supports a local $1$-Poincar\'e inequality.
\end{theorem}

This is Theorem~\ref{thm:local-PI-new}. It applies in particular
to separable $\CBA$ homology manifolds, but also to spaces that are not
homology manifolds, such as the spherical suspension of a flat
two-dimensional torus described in
Example~\ref{ex:GCBA_loc_uniq_orient_not_mfd}.

Semmes' theorem \cite{Semmes1996} establishes Poincar\'e inequalities
from quantitative topology and Ahlfors regularity in the metric-manifold
setting. Our result extends this connection to the above singular class
under an upper curvature bound. The role of manifold topology is played
by local uniqueness of the orientation current and the resulting
constancy theorem; filling estimates in small $\CAT(\kappa)$ neighbourhoods
provide the geometric input. This places the result within the
isoperimetric approach to Poincar\'e inequalities, as in
\cite{BassoMartiWenger25,Wenger2005,KorteLahti2014}. Such inequalities
are central to first-order analysis and quasiconformal geometry on
metric spaces; see \cite{HeinonenKoskela1998,Cheeger1999,HKST15}.

The study of local Poincaré inequalities on general $\GCBA$ spaces remains an open problem. Currently, the only general result available in this setting is the local $\infty$-Poincaré inequality, recently established by Caputo, Cavallucci, and Ikonen in \cite{CaputoCavallucciIkonen2026}. Theorem \ref{thm:intro_GCBA_PI} extends the validity of
local $1$-Poincar\'e inequalities beyond the manifold
setting to the full class of purely dimensional LUO
$\GCBA$ spaces. In particular, it applies to spaces
that fail to be homology manifolds, showing that local
uniqueness of orientation can replace manifold topology
in establishing this fundamental analytic property.

\subsection{Organization of the paper}

After the preliminaries, we develop the local-current tools needed for
restriction, extension and control of boundary mass, and recall the
rectifiability and degree-theoretic ingredients. Section~\ref{sec:orientability_of_manifolds}
proves the orientation theorem for degree-regular metric manifolds.
Section~\ref{sec:n-regular} treats singular $n$-regular spaces, their
top-dimensional homology and branched orientation double covers.
Sections~\ref{sec:RCD} and~\ref{sec:GCBA} contain the applications to
non-collapsed $\RCD$ spaces and to $\GCBA$ spaces, respectively.
\subsection{Acknowledgements}
The authors acknowledge the use of AI solely for basic language proofreading and as an aid in drafting the proof of Theorem \ref{theo:boundary_finite_mass_hard}. AI was not used in any other aspect of this work. All mathematical content has been independently verified by the authors, who assume full responsibility for the paper.

\section{Preliminaries} 

Let $(X,\dist)$ be a metric space, $A\subseteq X$ and $r>0$. The open (resp. closed) $r$-neighbourhood of $A$ is denoted by $B(A,r)$ (resp. $\overline{B}(A,r)$). Given $s \ge 0$, we set $\omega_s:=\pi^{s/2}/\Gamma(s/2+1)$, where $\Gamma(\cdot)$ is the $\Gamma$-function. The $s$-dimensional Hausdorff measure of $A \subseteq X$ is 
\begin{equation}
    \label{eq:defin_Hausdorff_measure}
 \haus^s(A):=\frac{\omega_s}{2^s}
 \cdot \lim_{\delta\downarrow0}
 \inf\left\{
   \sum_i(\operatorname{diam}U_i)^s:
   A\subseteq\bigcup_iU_i,\ 
   \operatorname{diam}U_i\leq\delta
 \right\}.
\end{equation}
With this normalization, $\haus^k = \Leb^k$ on $\setR^k$.
The Hausdorff dimension of $A$ is $\textup{dim}_H(A) := \inf\{s\ge 0\,:\, \haus^s(A)=0\} = \sup\{s\ge 0\,:\, \haus^s(A) = \infty\}$.

\subsection{Lipschitz functions}

Let $f\colon (X,\dist_X) \to (Y,\dist_Y)$ be a function. The Lipschitz constant of $f$ is $\Lip(f) := \sup_{x,x'\in X, x\neq x'} \frac{\dist_Y(f(x),f(x'))}{\dist_X(x,x')}$. A function $f$ is Lipschitz if $\Lip(f)<\infty$. A function $f$ is locally Lipschitz if every point has an open neighbourhood $U$ such that the restriction $f|_U$ is Lipschitz. 
The asymptotic Lipschitz constant of a locally Lipschitz map
$f\colon (X,\dist_X)\to(Y,\dist_Y)$ is defined by
\begin{equation}\label{eq:defin_loc_Lip_const}
    \lip_a f(x)
    := \lim_{r\downarrow 0}
    \operatorname{Lip}\bigl(f|_{B(x,r)}\bigr).
\end{equation}
Here the Lipschitz constant on a singleton is understood to be zero.
The limit exists by monotonicity and is finite by local
Lipschitz continuity.
We denote by $\Lip(X,Y)$, $\Lip_c(X,Y)$ and $\Lip_\loc(X,Y)$ the set of Lipschitz functions, Lipschitz functions with compact support and locally Lipschitz functions from $(X,\dist_X)$ to $(Y,\dist_Y)$, respectively. When $Y=\setR$, we simply write $\Lip(X)$, $\Lip_c(X)$ and $\Lip_\loc(X)$.

We recall a version of the coarea inequality for the Hausdorff measure, expressed in terms of the upper integral. 

\begin{lemma}[{Eilenberg's Coarea Inequality, \cite[Theorem 1.1]{Esmayli_Hajlasz_2021}}]
\label{lemma:coarea_inequality}
    Let $(X,\dist_X)$, $(Y,\dist_Y)$ be metric spaces, let $f\colon X \to Y$ be Lipschitz and let  $0\leq t \leq s$. For every $A\subseteq X$ it holds
    $$ \int_Y^* \haus^{s-t}(f^{-1}(y)\cap A)\,d\haus^t(y)\le (\Lip f)^t \, \frac{\omega_{s-t} \omega_t}{\omega_s} \,\haus^s(A).$$
\end{lemma}

We will also use the following basic version of the area formula.

\begin{lemma}[{\cite[Theorem 3.2.5]{Federer}}]\label{lemma: area-formula}
    Let $\psi\colon \setR^k \to \setR^k$ be Lipschitz and $f \colon \setR^k\to \setR$ be Lebesgue integrable. Then
    $$\int_{\setR^k} f(x) |\det(D_x \psi)| \, d\mathcal{L}^k(x) =\int_{\setR^k} \sum_{x \in \psi^{-1}(y)} f(x)  \, d\mathcal{L}^k(y).$$
\end{lemma}

\subsection{Metric measure spaces}

A metric measure space is a triplet $(X, \dist, \meas)$ where $(X, \dist)$ is a 
(not necessarily complete) metric space and $\meas$ is an outer measure which
is Radon. In particular, $\meas(K)<\infty$ for every compact $K\subseteq X$. It is said to be $(D)$-doubling if $\meas(2B)\le D\cdot \meas(B)$ for every ball $B=B(x,r)\subseteq X$. Here and in the sequel, given $\lambda \ge 0$ and a ball $B=B(x,r)$, we denote by $\lambda B$ the ball $B(x,\lambda r)$. 

Given $s\ge 0$, $(X,\dist,\meas)$ is said to be upper (resp. lower) Ahlfors $s$-regular if there exists $C >0$ such that $\meas(B(x,r)) \le Cr^s$ (resp. $\meas(B(x,r)) \ge Cr^s$) for every $x\in X$ and $r\ge 0$. It is Ahlfors $s$-regular if it is both upper and lower Ahlfors $s$-regular.

The metric measure space $(X,\dist,\meas)$ is said to be locally doubling if for every $x\in X$ there exist $r_x, D_x>0$ such that $\meas(B(y,2r)) \le D_x\meas(B(y,r))$ for every $y\in B(x,r_x)$ and every $0<r<r_x$. Similarly, one defines the notions of being locally (upper/lower) Ahlfors regular.

The metric measure space $(X,\dist,\meas)$ is said to satisfy a local $1$-Poincaré inequality, according to \cite[Definition 1.1]{BjornBjorn2018}, if for every $x\in X$, there exist $r_x>0$ and
$C_x,\lambda_x\geq 1$ such that
\[
 \dashint_{B(y,r)}|f-f_{B(y,r)}|\,d\meas
 \leq C_xr
 \dashint_{B(y,\lambda_xr)}\lip_a f\,d\meas
\]
for every Lipschitz $f\colon X\to\setR$, every
$y\in B(x,r_x)$ and every $0<r<r_x$.  The radii are chosen so that
$0<\meas(B(y,r))\leq\meas(B(y,\lambda_xr))<\infty$
for all balls under consideration.

Let $(X,\dist,\meas)$ be a metric measure space and let  $s\ge 0$. The upper (resp.lower) $s$-density of $\meas$ at a point $x\in X$ is the quantity
\begin{equation}
    \label{eq:defin_densities}
    \Theta^s(x,\meas):= \limsup_{r\to 0} \frac{\meas(B(x,r))}{\omega_sr^s}\quad \left(\textup{resp. } \Theta_s(x,\meas):= \liminf_{r\to 0} \frac{\meas(B(x,r))}{\omega_sr^s}\right).
\end{equation}

\section{Metric currents}
        We follow the theory of currents in locally compact metric spaces developed by Lang in \cite{Lang}; for closely related but different theories see also \cite{AmbrosioKirchheim00, LangWenger} and Section \ref{subsec:AK_currents}. Let $(X,\dist)$ be a locally compact metric space. We equip $\Lip_c(X)$ with a locally convex vector space topology such that $f_j \to f$ in $\Lip_c(X)$ if and only if there exists a compact $K\subseteq X$ such that $\spt (f_j) \subseteq K$, $\Lip(f_j) \le L$ for every $j \in \setN$ and $f_j \to f$ pointwise. Furthermore, we endow $\Lip_{\loc}(X)$ with a locally convex vector space topology such that $f_j \to f$ in $\Lip_{\loc}(X)$ if and only if $f_j \to f$ pointwise and for each compact $K\subseteq X$ there exists $L_K\geq 0$ such that $\Lip(f_j|_K)\leq L_K$. For $k\geq 0$, we set $\mathcal{D}^k(X) := \Lip_c(X) \times \Lip_{\loc}(X)^k$. We say that a sequence $(f^j,\pi_1^j,\dots,\pi_k^j)$ converges to $(f,\pi_1,\dots,\pi_k)$ in $\mathcal{D}^k(X)$ if $f^j\to f$ in $\Lip_c(X)$ and $\pi_i^j\to \pi_i$ in $\Lip_\loc(X)$ for all $i = 1,\dots,k$. An element $(\pi_1,\ldots,\pi_k)\in \Lip_\loc(X)^k$ will often be denoted simply by $\pi$. Similarly, we will denote by $\pi\colon X \to \setR^k$ the locally Lipschitz map with coordinates $(\pi_1,\ldots,\pi_k)$.

        \begin{definition}[Currents]
        \label{def:current}
            Let $k \in \setN$. A $k$-dimensional metric current is a $(k+1)$-linear function $T\colon \mathcal{D}^k(X)\to \setR$ with the following properties.
            \begin{enumerate}
                \item (Continuity) $T(f^j,\pi^j_1,\dots,\pi_k^j) \to T(f,\pi_1,\dots,\pi_k)$ if $(f^j,\pi_1^j,\dots,\pi_k^j) \to \\(f,\pi_1,\dots,\pi_k)$ in $\mathcal{D}^k(X)$.
                \item (Locality) If $k\geq 1$, then $T(f,\pi_1,\dots,\pi_k) = 0 $ whenever some $\pi_i$ is constant on an open neighbourhood of $\spt (f)$.
            \end{enumerate}
        The vector space of $k$-dimensional metric currents in $X$ is denoted by $\mathcal{D}_k(X)$. As a consequence of the definition we have
        \begin{equation}
            \label{eq:Leibniz_currents}
            T(f, gh, \pi) = T(fg, h, \pi) + T(fh, g, \pi)
        \end{equation}
        for all $(f,\pi) \in \mathcal{D}^{k-1}(X)$ and $g,h\in \Lip_\loc(X)$, see \cite[Proposition 2.4]{Lang}.
        
        \end{definition}

        \begin{definition}[Support]
            Let $k\in \setN$ and $T \in \mathcal{D}_k(X)$. The support $\spt (T)$ of $T$ is the intersection of all closed sets $C\subseteq X$ such that $T(f, \pi) = 0$ for all $(f, \pi) \in \mathcal{D}_k(X)$ with $\spt(f) \cap  C = \emptyset$.
        \end{definition}
        
         By \cite[Lemma 3.2]{Lang}, we have the following equivalent description:
         \begin{equation}
             \label{eq:defin_alternative_support}
             \spt (T) = \{ x \in X \colon \forall \varepsilon>0 \,\, \exists(f,\pi)\in \mathcal{D}^k(X) \textup{ with } \spt (f) \subseteq B(x,\varepsilon) \textup{ and } T(f,\pi)  \neq 0\}.
         \end{equation}

        The following definition of push-forward corresponds to \cite[Definition 3.6]{Lang}.
        
        \begin{definition}[Push-forward]
        \label{defin:pushforward}
            Let $T\in \mathcal{D}_k(X)$ and let $(Y,\dist_Y)$ be a locally compact metric space. Let $A\subseteq X$ be a locally compact set containing $\spt (T)$ and let $\varphi\colon A\to Y$ be a locally Lipschitz map which is proper, i.e. $\varphi^{-1}(K)$ is compact whenever $K\subseteq Y$ is compact. The pushforward $\varphi_\#T$ of $T$ under $\varphi$ is the $k$-current in $Y$ given by
            $$\varphi_\#T(f,\pi_1,\dots,\pi_k) = T(f\circ\varphi,\pi_1\circ \varphi,\dots,\pi_k\circ\varphi)$$
            for all $(f,\pi_1,\ldots,\pi_k)\in \mathcal{D}^k(Y)$. It satisfies
            \begin{equation}
                \label{eq:support_vs_pushforward}
                \spt(\varphi_\#T)\subseteq \varphi(\spt(T)).
            \end{equation}
        \end{definition}

        Next, we recall the definition of the boundary of a current.

        \begin{definition}[Boundary]
            Let $k \geq 1$. The boundary $\partial T$ of $T\in \mathcal{D}_k(X)$ is the metric $(k-1)$-current defined by 
            \begin{equation}
                \label{eq:defin_boundary}
                \partial T(f,\pi_1,\dots,\pi_{k-1}) = T(\sigma,f,\pi_1,\dots,\pi_{k-1})
            \end{equation}
            for all $(f,\pi_1,\dots,\pi_{k-1})\in \mathcal{D}^{k-1}(X)$, where $\sigma\in \Lip_c(X)$ is any function such that $\sigma = 1$ on $\{f \neq 0\} \cap \spt(T)$.
        \end{definition}
        
        The definition does not depend on the choice of $\sigma$, by \cite[Lemma 3.2.(3)]{Lang}. Moreover, by \cite[Eq. 3.9]{Lang},
        \begin{equation}
            \label{eq:boundary_pushforward_commute}
            \partial(\varphi_\#T) = \varphi_\#(\partial T).
        \end{equation} 

        \begin{definition}[Restriction]
        \label{defin:restriction_currents}
            Let $T\in \mathcal{D}_k(X)$ and $(g,\tau) \in \Lip_\loc(X)\times \Lip_\loc(X)^m$, $m \leq k$. The restriction $T \rstr (g,\tau)$ of $T$ to $(g,\tau)$ is the $(k-m)$-current given by
            $$T\rstr(g,\tau)(f,\pi_1,\dots,\pi_{k-m}) = T(fg,\tau_1,\dots,\tau_m,\pi_1,\dots,\pi_{k-m})$$
            for all $(f,\pi_1,\dots,\pi_{k-m})\in \mathcal{D}^{k-m}(X)$.
        \end{definition}
        
        The following identities will be useful  
        \begin{equation}\label{eq:restr-boundary}
            (\partial T) \rstr f = T \rstr(1,f) + \partial(T\rstr f)
        \end{equation}
        and
        \begin{equation}\label{eq:boundary-restr-switch}
            \partial (T \rstr (1,f)) = - (\partial T)\rstr (1,f)
        \end{equation}
        for all $f \in \Lip_\loc(X)$. Both follow from \cite[Lemma 3.5]{Lang}: the first one directly, the second one by applying the locality property to infer that $T\rstr (1,1,f) = 0$.

        The notion of mass will be important in our considerations.

        \begin{definition}[Mass]
        \label{def:mass}
            Let $T \in \mathcal{D}_k(X)$. For every open set $V \subseteq X$, we define the mass $\bM_V (T)$ of $T$ in $V$ as the least number $M \in [0, \infty]$ such that
            $$\sum_{j \in J} T(f^j,\pi_1^j,\dots,\pi_k^j) \leq M$$
            whenever $J$ is a finite set, $(f^j,\pi_1^j,\dots,\pi_k^j)\in \mathcal{D}^k(X)$, $\spt (f^j) \subseteq V$, $\sum_{j \in J} |f^j| \leq 1 $ and $\Lip(\pi^j_i) \leq 1$ for all $i =1,\dots,k$.
        \end{definition}
        Given $T\in \mathcal{D}_k(X)$, we define the mass measure of $T$ by 
        $$\norm{T}(A)= \inf \{ \bM_V(T) \colon V\subseteq X \textup{ is open and } A\subseteq V\}$$
        for all subsets $A \subseteq X$. This defines a Borel measure on $X$; see \cite[Theorem 4.3]{Lang}. The total mass of $T$ is equal to $\bM(T) = \norm{T}(X)$ and in case $\bM(T) < \infty$ we say $T$ has finite mass. The space of $k$-currents in $X$ with finite mass is denoted by $\bM_k(X)$. The total mass $\bM$ defines a norm on $\bM_k$. We say a $k$-current $T\in \mathcal{D}_k(X)$ has locally finite mass if $\norm{T}$ is a locally finite measure. For $k \in \setN$, we write $\bM_{k,\loc}(X)$ for the set of all $k$-currents with locally finite mass. For $T \in \bM_{k,\loc}(X)$ we have 
        \begin{equation}
            \label{eq:support_current_=_support_mass}
            \spt (T) = \spt \norm{T},
        \end{equation}
        the measure $\norm{T}$ is Radon 
        and
        \begin{equation}\label{eq:mass-measure-ineq}
            |T(f,\pi_1,\dots,\pi_k) | \leq \prod_{i=1}^k \Lip\left(\pi_i|_{\spt (f)}\right) \int_X |f| \;d\norm{T}
        \end{equation}
        for all $(f,\pi_1,\dots,\pi_k)\in \mathcal{D}^k(X)$. These statements correspond to \cite[Eq. 4.3 and Theorem 4.3]{Lang}.
        The bound \eqref{eq:mass-measure-ineq} can be promoted to the following localized version.
        \begin{proposition}
        \label{prop:mass-measure-ineq-loc-Lip-constant}
            Let $(X,\dist)$ be a separable, locally compact metric space and let $T\in \bM_{k,\loc}(X)$. Then 
            \begin{equation}
                \label{eq:mass_measure_ineq_loc_Lipschitz_constant}
                |T(f,\pi_1,\dots,\pi_k) | \leq \int_X |f|\prod_{i=1}^k \lip_a \pi_i  \;d\norm{T}
            \end{equation}      for all $(f,\pi_1,\dots,\pi_k)\in \mathcal{D}^k(X)$.
        \end{proposition}
       
        \begin{proof}
            The proof is a straightforward adaptation of the one of \cite[Lemma 3.1]{FasslerOrponen2019}. The upper pointwise Lipschitz constant used in \cite[Lemma 3.1]{FasslerOrponen2019} is defined using a $\limsup$ of radial oscillations divided by the radius. It is bounded above by the asymptotic Lipschitz constant in \eqref{eq:defin_loc_Lip_const}. Consequently, the estimate obtained by adapting their argument also implies the inequality stated here.
        \end{proof}

        We denote by $\mathcal{B}_c^\infty(X)$ the algebra of all bounded Borel functions $f \colon X \to \setR$ such that $\spt (f)$ is compact. Every $T\in \bM_{k,\loc}(X)$ extends to a $(k+1)$-linear function $T \colon \mathcal{B}_c^\infty(X) \times \Lip_\loc(X)^k$ that satisfies the continuity and locality axioms of Definition \ref{def:current} and such that \eqref{eq:mass-measure-ineq} still holds, see \cite[Theorem 4.4]{Lang}. Using this extension, we can define the restriction $T\rstr u$ of $T$ to locally bounded Borel functions $u\in \mathcal{B}_\loc^\infty(X)$. More precisely, 
        $$T\rstr u (f,\pi_1,\dots,\pi_k)= T(u f,\pi_1,\dots,\pi_k)$$
        for all $(f,\pi_1,\dots,\pi_k) \in \mathcal{D}^k(X)$. If $u = \chi_B$ for some Borel set $B \subseteq X$, then we simply write $T\rstr B = T \rstr u$. In this case, we have $$\norm{T}\rstr B \leq  \norm{T\rstr B} \rstr B. $$
        If $\spt(T)$ is separable, or if $B$ is open or $\sigma$-finite with respect to $\norm{T}$, then the previous inequality becomes an equality and we have 
        \begin{equation}
            \label{eq:mass_of_restriction}
            \norm{T}(B) = \bM(T\rstr B);
        \end{equation} see \cite[Lemma 4.7]{Lang}. Moreover, a direct application of the definitions proves that
        \begin{equation}
            \label{eq:borel-restriction_vs_pushforward}
            (\varphi_\#T)\rstr B = \varphi_\#(T\rstr \varphi^{-1}(B))
        \end{equation}
        for every $T\in \bM_{k,\loc}(X)$ and every locally Lipschitz, proper map $\varphi \colon X \to Y$ with $Y$ locally compact.
        
        \begin{definition}[Normal currents]
        Let $T\in \mathcal{D}_k(X)$.
        A $k$-current $T \in \mathcal{D}_k(X)$ is called normal if $\bM(T) + \bM(\partial T) < \infty$. The space of normal $k$-currents in $X$ is denoted by $\bN_k(X)$. If $\norm{T} + \norm{\partial T}$ is locally finite, then we say that $T$ is a locally normal current. We write $\bN_{k,\loc}(X)$ for the space of locally normal $k$-currents in $X$.
        \end{definition}

        \begin{example}[Main example]
            The standard example of a metric $k$-current with local finite mass is given by integration with respect to a locally integrable function in $\setR^k$. Let $U \subseteq \setR^k$. Every $\theta\in L^1_\loc(U)$ induces a $k$-current $\br{\theta}\in \mathcal{D}_k(U)$ as follows
        \begin{equation}\label{eq:std-current}
            \br{\theta}(f,\pi) = \int_U \theta f \det(D\pi) \, d\haus^k
        \end{equation}
        for all $(f,\pi) \in \mathcal{D}^k(U)$; see \cite[Proposition 2.6]{Lang}. Moreover, $\norm{\br{\theta}} = |\theta| \haus^k$.  
        \end{example}

        \subsection{Restriction and extension of currents}

        Currents with locally finite mass can be restricted to define currents in locally compact subsets.
        
        \begin{proposition}
        \label{prop:restriction_of_current_to_open_subsets}
            Let $(X,\dist)$ be a separable, locally compact metric space, let $T\in \bM_{k,\loc}(X)$ and let $B\subseteq X$ be a locally compact subset. For $(f,\pi_1,\ldots,\pi_k) \in \Lip_c(B) \times (\Lip(B))^k$, set
            $$T|_B(f,\pi_1,\ldots,\pi_k) := T\rstr B(f',\pi'_1,\ldots,\pi_k'),$$
            where $f' \in \Lip_c(X)$ is a Lipschitz extension of $f$ and $\pi_i' \in \Lip(X)$ is a Lipschitz extension of $\pi_i$, for $i=1,\ldots, k$. This formula is independent of the chosen extensions and extends uniquely to a current $T|_B\in\bM_{k,\loc}(B)$. Moreover,
            \begin{itemize}
                \item[(i)] $T|_B \in \bM_{k,\loc}(B)$ and $\norm{T} \rstr B = \norm{T|_B}$,
                \item[(ii)] $\norm{\partial(T|_B)} \le \norm{\partial(T\rstr B)} \rstr B$, 
                \item[(iii)] if $B$ is closed, then $\varphi_\#(T\rstr B) = (\varphi|_B)_\#(T|_B)$ for every locally Lipschitz, proper map $\varphi\colon X \to Y$, where $Y$ is locally compact.
                \item[(iv)] if $B'\subseteq B$ is locally compact, then $T|_{B'} = (T|_B)|_{B'}$.
            \end{itemize} 
        \end{proposition}
        \begin{proof}
            If $f',\pi'_1,\ldots,\pi_k'$ and $f'',\pi''_1,\ldots,\pi_k''$ are two Lipschitz extensions as in the statement, then we can use the multilinearity of $T\rstr B$ (\cite[Theorem 4.4.(i)]{Lang}) to bound
            \begin{equation}
                    |T\rstr B(f',\pi'_1,\ldots,\pi_k') - T\rstr B(f'',\pi''_1,\ldots, \pi_k'')| 
            \end{equation}
            by
            \begin{equation}
                \begin{aligned}
                    |T\rstr B(f' - f'',\pi'_1,\ldots,\pi_k')|  + \sum_{i=1}^k|T\rstr B(f'',\pi_1'',\ldots,\pi'_i - \pi''_i,\ldots \pi_k')|.
                \end{aligned}
            \end{equation}
            For the first term, we use \eqref{eq:mass-measure-ineq} to get
            $$|T\rstr B(f' - f'',\pi'_1,\ldots,\pi_k')| \le \int_B|f'-f''|\prod_{i=1}^k \Lip(\pi_i')\,d\norm{T} = 0$$
            since $f'=f''$ on $B$. For each of the terms in the sum we use \cite[Theorem 4.4.(iii)]{Lang} to get
            $$|T\rstr B(f'',\pi_1'',\ldots,\pi'_i - \pi''_i,\ldots \pi_k')| = T(\mathbb{1}_Bf'',\pi_1'',\ldots,\pi'_i - \pi''_i,\ldots \pi_k') = 0$$
            because $\pi_i' = \pi_i''$ on $\spt(\mathbb{1}_Bf'') = \spt(f)$ which is compact in $B$.
            Multilinearity and locality of $T|_B$ follow from the corresponding properties of $T$ and its extension to Borel functions in the first argument. To verify continuity on $\Lip_c(B)^{k+1}$, choose Lipschitz extensions with uniformly bounded Lipschitz constants and multiply the extensions in the first argument by a fixed compactly supported Lipschitz cutoff equal to one near the common compact support of the original first arguments. Continuity then follows from that of $T\rstr B$. By \cite[Lemma 2.2]{Lang}, the resulting functional extends uniquely to a current on $\mathcal D^k(B)$.
            
            The facts that $T|_B$ has locally finite mass and that $\norm{T}\rstr B = \norm{T|_B}$ follow from \cite[Lemma 4.7]{Lang}. The inequality $\norm{\partial(T|_B)} \le \norm{\partial(T\rstr B)} \rstr B$ is a direct consequence of the definition of mass. The equality in (iii) follows directly by unpacking the definitions. Notice that if $B$ is closed, then $\varphi|_B$ is a proper map. 
            Finally, (iv) follows directly from the definitions.
        \end{proof}

        The restriction of a current to its support has full support.
        \begin{lemma}
            \label{lemma:restriction_to_support}
            Let $(X,\dist)$ be a separable, locally compact metric space and let $T\in \bM_{k,\loc}(X)$. Then, $\spt(T|_{\spt(T)}) = \spt(T)$.
        \end{lemma}
        \begin{proof}
            The general inequality $\spt(T|_B) \subseteq \spt(T) \cap B$ follows from the definition. For the other direction, let $x\in \spt(T)$ and let $U$ be an open neighbourhood of $x$ in $\spt(T)$. By definition, there exists an open subset $V \subseteq X$ such that $U = V\cap \spt(T)$. By \eqref{eq:support_current_=_support_mass}, $\norm{T}(U) = \norm{T}(V) > 0$. Therefore, by Proposition \ref{prop:restriction_of_current_to_open_subsets}.(i),
            $$\norm{T|_{\spt(T)}}(U) = \norm{T}\rstr \spt(T) (U) = \norm{T}(U) > 0.$$
            This shows, again by \eqref{eq:support_current_=_support_mass}, that $x\in \spt(T|_{\spt(T)})$ and concludes the proof.
        \end{proof}

        Cycles restricted to open sets remain cycles. For an open subset $B\subseteq X$, the restriction
$R|_B$ of an arbitrary current $R$ is defined locally
by extending the compactly supported first test
function by zero and extending the remaining test
functions near its support. Locality makes this
definition independent of the extensions.
For currents of locally finite mass, it agrees with
the restriction defined above.

        \begin{lemma}
            \label{cor:cycles_open_sets}
            Let $(X,\dist)$ be a separable, locally compact metric space, $T\in \bM_{k,\loc}(X)$ and $B \subseteq X$ be locally compact. Then
            \begin{equation}
                \label{eq:support_boundary_restriction}
                \spt(\partial(T|_B)) \subseteq (\spt(\partial T)\cup \partial B) \cap B.
            \end{equation}
            If $B$ is open then $(\partial T)|_B = \partial(T|_B)$.
        \end{lemma}
        \begin{proof}
            Let $x\in \spt(\partial(T|_B))$. We need to prove that if $x\notin \partial B \cap B$, then $x\in \spt(\partial T)$. By definition, $x \in \textup{int}(B)$. Let $\varepsilon > 0$ be such that $B(x,\varepsilon) \subseteq B$. By \eqref{eq:defin_alternative_support}, we can find $(f,\pi)\in \mathcal{D}^{k-1}(B)$ with $\spt(f)\subseteq B(x,\varepsilon/2)$ and $\partial(T|_B)(f,\pi) \ne 0$. By multiplying with a cutoff function relative to $\spt(f)$, we can assume that $\pi \in (\Lip(X))^{k-1}$. Let $\sigma \in \Lip_c(B)$ be such that $\sigma \equiv 1$ on $B(x,\varepsilon/2)$. Let $f', \sigma'$ be the zero extension of $f,\sigma$ to $X$: they are Lipschitz because $\spt(f)$, $\spt(\sigma) \subseteq \textup{int}(B)$. Let $\pi'$ be an arbitrary Lipschitz extension of $\pi$. Then, applying twice \eqref{eq:defin_boundary} we get
            $$0\neq \partial(T|_B)(f,\pi) = T|_B(\sigma,f,\pi) = T(\mathbb{1}_B\sigma',f',\pi') = T(\sigma',f',\pi') = \partial T(f',\pi').$$
            This shows that $x\in \spt(\partial T)$, concluding the proof of the first part of the statement.
        
            Suppose now that $B$ is open. Let $(f,\pi) \in \Lip_c(B) \times (\Lip(B))^{k-1}$ and let $\pi'$ be an arbitrary Lipschitz extension of $\pi$. Let $\sigma \in \Lip_c(B)$ be such that $\sigma \equiv 1$ on $\spt(f)$. Since $B$ is open, extending $f$ and $\sigma$ to $0$ outside of $B$ gives well-defined $f'$, $\sigma' \in \Lip_c(X)$. Then
            \begin{equation}
                \begin{aligned}
                    \partial(T|_B)(f,\pi) \stackrel{\eqref{eq:defin_boundary}}{=} &T|_B(\sigma,f,\pi) = T\rstr B(\sigma',f',\pi') = T(\sigma',f',\pi') \\
                    \stackrel{\eqref{eq:defin_boundary}}{=} &\partial T(f',\pi')= (\partial T)|_B(f,\pi).
                \end{aligned}
            \end{equation}
            
        \end{proof}

        The restriction commutes with the push-forward in the following sense.

        \begin{lemma}
        \label{lemma:restriction_vs_pushforward}
            Let $(X,\dist)$, $(Y,\dist_Y)$ be separable, locally compact metric spaces, let $\varphi \colon X \to Y$ be a locally Lipschitz, proper map, let $B\subseteq Y$ be a locally compact subset and let $T \in \bM_{k,\loc}(X)$. Then, \begin{equation}\label{eq:restriction_vs_pushforward}
                (\varphi_\#T)|_B = (\varphi|_{\varphi^{-1}(B)})_\#(T|_{\varphi^{-1}(B)}).
            \end{equation}
        \end{lemma}
        \begin{proof}
            Notice that $\varphi^{-1}(B)$ is locally compact since $B$ is and $\varphi$ is proper, so the two currents on $B$ are well-defined. Moreover, $\varphi|_{\varphi^{-1}(B)}$ is a proper map. For every $(f,\pi) \in \Lip_c(B) \times (\Lip(B))^k$ we have
            \begin{equation}
                \begin{aligned}
                    (\varphi_\#T)|_B(f,\pi) = T((\mathbb{1}_Bf') \circ \varphi, \pi'\circ \varphi) &= T(\mathbb{1}_{\varphi^{-1}(B)} f'\circ \varphi, \pi'\circ \varphi) \\
                    &= T(\mathbb{1}_{\varphi^{-1}(B)} f\circ \varphi|_{\varphi^{-1}(B)}, \pi\circ \varphi|_{\varphi^{-1}(B)})\\
                    &= (\varphi|_{\varphi^{-1}(B)})_\#(T|_{\varphi^{-1}(B)})(f,\pi),
                \end{aligned}
            \end{equation}
            where $f' \in \Lip_c(Y)$, $\pi' \in \Lip(Y)$ are Lipschitz extension of $f,\pi$ to $Y$.
        \end{proof}

        Currents defined on locally compact subsets $B$ can be extended to the whole space $X$. For that, we recall that every locally compact, separable metric space $B$ admits an exhausting sequence of compact sets $K_j \subseteq B$ such that $B=\bigcup_{j} K_j$ and $K_j \subseteq \textup{int}_B(K_{j+1}) =: W_j$. For every $j$, we consider $\varepsilon_j := \dist(K_j, B \setminus W_j) > 0$, with $\dist(K_j,\emptyset) = +\infty$. Consider the sequence of Lipschitz cutoff functions 
        \begin{equation}
            \label{eq:defin_psi_j}
            \psi_j\colon B \to \setR,\quad \psi_j(x) := \max\left(0, 1 - \frac{2}{\varepsilon_j} \dist(x, K_j)\right).
        \end{equation}
        Hence, $\psi_j\in \Lip_c(B)$, $\psi_j\equiv 1$ on $K_j$ and $\spt(\psi_j) \subseteq W_j$.

        \begin{lemma}\label{lemma:extension-of-currents}
            Let $(X,\dist)$ be a separable, locally compact metric space and $B \subseteq X$ be locally compact. Let $T \in \bM_{k,\loc}(B)$ be such that $\norm{T}(B\cap K) < \infty$ for every compact $K\subseteq X$. Then, the formula
            \begin{equation}
                \label{eq:defin_T^X}
                T^X(f,\pi) := \lim_{j\to +\infty} T(\psi_jf|_B,\pi|_B), \quad (f,\pi) \in \mathcal{D}^k(X),
            \end{equation}
            where $\psi_j$ are as in \eqref{eq:defin_psi_j}, defines a current $T^X \in \bM_{k,\loc}(X)$. It does not depend on the choice of the exhausting sequence of compact sets. Moreover,  $\spt(\partial (T^X)) \subseteq \spt(\partial T) \cup (X\setminus B)$.
        \end{lemma}

        \begin{proof}
            Notice that $\psi_j \equiv 1$ on $K_j$, its support $\spt(\psi_j) \subseteq \overline{W_j} \subseteq B$ is compact, i.e. $\psi_j \in \Lip_{c}(B)$, and $\psi_j \to 1$ pointwise on $B$.
            Let $(f,\pi) = (f, \pi_1, \dots, \pi_k) \in \mathcal{D}^k(X)$. Using \eqref{eq:mass-measure-ineq} for $i \le j$ we have that
            \begin{align*}
                |T(\psi_if|_B,\pi|_B)-T(\psi_jf|_B,\pi|_B) | &\leq \prod_{l=1}^k\Lip(\pi_l|_{\spt(f)})  \int_B  |f (\psi_i -\psi_j)| \, d\norm{T} 
                \\
                &\leq  \norm{f}_\infty  \prod_{l=1}^k\Lip(\pi_l|_{\spt(f)}) \norm{T}(\spt(f) \cap(B \setminus K_i)),
            \end{align*}
            which converges to $0$ as $i \to \infty$, since $\norm{T}(\spt(f) \cap B) <\infty$ by assumption. This shows that the limit in \eqref{eq:defin_T^X} exists. Therefore, $T^X$ defines a current on $X$. A similar proof shows that $T^X$ does not depend on the particular choice of the cut-off functions and of the exhausting sequence.    
            Next, we show that $T^X$ has locally finite mass. Let $V \subseteq X$ be open with compact closure, $J$ be a finite set, $(f^j,\pi_1^j,\dots,\pi_k^j)\in \mathcal{D}^k(X)$, $\spt (f^j) \subseteq V$, $\sum_{j \in J} |f^j| \leq 1 $ and $\Lip(\pi^j_i) \leq 1$ for all $i =1,\dots,k$. Using \eqref{eq:mass-measure-ineq} again we get
            $$\sum_{j \in J} T^X(f^j,\pi_1^j,\dots,\pi_k^j) \leq  \lim_{i \to \infty}  \int_B \sum_{j\in J} |f^j\psi_i|\,d\norm{T} \leq  \norm{T}( \overline{V} \cap B) < \infty.$$
            Therefore, $T^X\in \bM_{k,\loc}(X)$. 
            The inclusion $\spt(\partial (T^X)) \subseteq \spt(\partial T ) \cup (X\setminus B)$ is a direct consequence of the definition of $T^X$ and the support of a current. 
            This completes the proof. 
        \end{proof}

        The extension and the restriction interact in the expected way.

        \begin{corollary}
        \label{cor:extension_vs_restriction}
            Let $(X,\dist)$ be a separable, locally compact metric space and let $B\subseteq X$ be locally compact. Then:
            \begin{itemize}
                \item[(i)] $T^X|_B = T$ for every $T\in \bM_{k,\loc}(B)$ such that $\norm{T}(B\cap K)<\infty$ for every compact $K\subseteq X$, so $\|T^X\|\rstr B = \norm{T}$;
                \item[(ii)] $(T|_B)^X = T\rstr B$ for every $T \in \bM_{k,\loc}(X)$.
            \end{itemize}
        \end{corollary}
        \begin{proof}
            Let $(f,\pi) \in \Lip_c(B) \times (\Lip(B))^{k}$, let $f'\in \Lip_c(X)$ and $\pi' \in (\Lip(X))^k$ be Lipschitz extensions and let $T \in \bM_{k,\loc}(B)$. Then
            $$T^X|_B(f,\pi) = T^X\rstr B(f',\pi') = T^X(\mathbb{1}_Bf',\pi') = \lim_{j\to +\infty} T(\psi_jf'|_B,\pi'|_B) = T(f,\pi),$$
            where the last step follows by \cite[Theorem 4.4]{Lang}. Therefore, by Proposition \ref{prop:restriction_of_current_to_open_subsets}, $\norm{T^X}\rstr B = \norm{T^X|_B} = \norm{T}$. 
            
            Now, let $(f,\pi) \in \mathcal{D}^k(X)$ and $T\in \bM_{k,\loc}(X)$. Then
            $$(T|_B)^X(f,\pi) = \lim_{j\to +\infty} T|_B(\psi_jf|_B,\pi|_B) = \lim_{j\to +\infty} T\rstr B(\psi_jf,\pi) = T\rstr B(f,\pi),$$
            again by \cite[Theorem 4.4]{Lang}.
        \end{proof}

        We discuss the uniqueness of the extension $T^X$.
        \begin{lemma}
        \label{lemma:uniqueness_extension}
            Let $(X,\dist)$ be a separable, locally compact metric space and $B \subseteq X$ be a locally compact subset such that $\haus^k(X\setminus B)= 0$. Let $T \in \bM_{k,\loc}(B)$ be such that $\norm{T}(B\cap K) < \infty$ for every compact $K\subseteq X$. Then, every $\tilde{T}\in \bM_{k,\loc}(X)$ with $\tilde{T}|_B = T$ satisfies $\tilde{T} = T^X$.
        \end{lemma}
        \begin{proof}
            By Corollary \ref{cor:extension_vs_restriction} we have that
            $$T^X\rstr B = (T^X|_B)^X = T^X = (\tilde{T}|_B)^X = \tilde{T}\rstr B.$$
            Moreover, $T^X\rstr (X \setminus B) = 0$ by \eqref{eq:mass-measure-ineq} and Lemma \ref{lemma:mass-measure-abso-cont}, and similarly, $\tilde{T}\rstr (X\setminus B)=0$. Hence, $\tilde{T} = T^X$ by \cite[Theorem 4.4.(i)]{Lang}.
        \end{proof}

        The extension and the pushforward follow the next rule.
        \begin{lemma}
            \label{lemma:extension_vs_pushforward}
            Let $(X,\dist)$, $(Y,\dist_Y)$ be separable, locally compact metric spaces, let $\varphi \colon X \to Y$ be a locally Lipschitz, proper map, let $B\subseteq Y$ be a locally compact subset containing $\varphi(X)$ and let $\varphi^B\colon X \to B$, $x\mapsto \varphi(x)$. Then,
            $$(\varphi^B_\#T)^Y = \varphi_\# T$$
            for every $T \in \bM_{k,\loc}(X)$. 
        \end{lemma}
        \begin{proof}
            For every $(f,\pi) \in \mathcal{D}^k(Y)$ we have that
            \begin{equation}
                \begin{aligned}
                    (\varphi^B_\#T)^Y(f,\pi) = \lim_{j\to \infty} \varphi^B_\#T(\psi_j f|_B, \pi|_B) &= \lim_{j\to \infty} T(\psi_j f|_B \circ \varphi^B, \pi|_B \circ \varphi^B)\\ &= \lim_{j\to \infty} T(\psi_j f \circ \varphi, \pi \circ \varphi) = \varphi_\# T(f,\pi),
                \end{aligned}
            \end{equation}
            where the last step follows by \cite[Theorem 4.4]{Lang}.
        \end{proof}

        The extension and the boundary commute when $B$ is closed.
        \begin{lemma}
            \label{lemma:extension_vs_boundary}
            Let $(X,\dist)$ be a separable, locally compact metric space and $B \subseteq X$ be a closed subset. Then $(\partial T)^X = \partial (T^X)$
            for every $T\in \bM_{k,\loc}(B)$.
        \end{lemma}

       \begin{proof}
        Since $B$ is closed, the inclusion $\iota\colon B\hookrightarrow X$
        is proper. For every current $R$ on $B$, its pushforward
        $\iota_\#R$ is therefore well-defined and satisfies
        \[
        (\iota_\#R)(f,\pi)=R(f|_B,\pi|_B).
        \]
        For currents of locally finite mass this agrees with the
        extension $R^X$ defined above, since the cutoff functions
        are eventually identically one on $\spt_B(f|_B)$.
        For closed $B$, we use the notation $R^X=\iota_\#R$
        also for arbitrary currents, in particular for $R=\partial T$.
        
        Let $(f,\pi)\in\mathcal D^{k-1}(X)$ and choose
        $\sigma\in\Lip_c(X)$ equal to one on a neighbourhood of
        $\spt_X(f)$. Since $\sigma|_B\in\Lip_c(B)$, we obtain
        \[
        \begin{aligned}
            \partial(T^X)(f,\pi)
            &=T^X(\sigma,f,\pi)=T(\sigma|_B,f|_B,\pi|_B)=\partial T(f|_B,\pi|_B)\\
            &=(\partial T)^X(f,\pi).
        \end{aligned}
        \]
        This proves the claim.
        \end{proof}

        \subsection{Metric currents in the sense of Ambrosio--Kirchheim and Flat chains}
        \label{subsec:AK_currents}
            We quickly discuss the theory of metric currents developed by Ambrosio--Kirchheim \cite{AmbrosioKirchheim00}. Let $(X,\dist)$ be a complete metric space and let $\Lip_b(X)$ denote the space of bounded Lipschitz functions $f\colon X \to \setR$. An AK metric $k$-current $T\in \bM_k^\mathrm{AK}(X)$ is a $(k+1)$-linear function on $\Lip_b(X)\times \Lip(X)^k$ satisfying the locality property and the continuity property $\lim_{i \to \infty} T(f,\pi_1^j,\dots,\pi_k^j) = T(f,\pi_1,\dots,\pi_k)$ for all sequences in $\Lip_b(X)\times \Lip(X)^k$ such that $\pi_i^j\to \pi$ pointwise and $\Lip(\pi_i^j) \leq L$ for some constant $L>0$; cf. Definition \ref{def:current}. In addition, there exists a finite Borel measure $\mu$ on $X$ that is supported on a $\sigma$-compact set and $T$ satisfies \eqref{eq:mass-measure-ineq} with $\mu$ on the right hand side. The smallest such measure $\mu$ is denoted by $\mu_T$ and is called the mass measure of $T$. It follows that if $X$ is also locally compact, then every $T\in  \bM_k(X)$ induces a current in $\bM_k^\mathrm{AK}(X)$ and the measure $\mu_T$ coincides with $\norm{T}$. In particular, for $T \in \bM_k(X)$, the mass norm $\bM$ as defined above coincides with the mass norm $\bM^\mathrm{AK}$ in the sense of Ambrosio--Kirchheim; cf. \cite[Pages 710-711]{Lang}. On the other hand, every $T \in \bM^\mathrm{AK}_k(X)$ induces an element in $\bM_k(X)$ by putting
            $$T(f,\pi) = T(f,\pi')$$
            for all $(f,\pi) \in \mathcal{D}^k(X)$ and any Lipschitz map $\pi' \colon X \to \setR^k$ satisfying $\pi'|_{\spt (f)} = \pi|_{\spt (f)}$. Notice that by \cite[Theorem 3.5]{AmbrosioKirchheim00} this gives a well-defined $k$-current in the sense of Lang. It follows that $ \bM_k(X)= \bM_k^\mathrm{AK}(X)$ and $\bN_k(X) = \bN_k^\mathrm{AK}(X)$, where $\bN_k^\mathrm{AK}(X)$ denotes the space of normal currents in the sense of AK.
            
            Finally, we recall the notion of flat chains in Euclidean space. For a detailed exposition, we refer to \cite{Federer}. The space of $k$-dimensional flat chains $\mathbf{F}_k(\setR^m)$ in $\setR^m$ is the closure, with respect to the mass norm $\bM$, of normal currents $\bN_k(\setR^m)$. We denote by $\mathbf{F}_{k,c}(\setR^m)$ the subspace of flat chains with compact support. It is well-known that every flat chain $T \in \mathbf{F}_{k,c}(\setR^m)$ induces a metric $k$-current in $\bM_k(\setR^m)$; see \cite[Appendix A]{AmbrosioKirchheim00}. A deep result due to De Philippis--Rindler proves that in top-dimension, $\mathbf{F}_{k,c}(\setR^k)$ actually coincide with currents in $\bM_k(\setR^k)$ with compact support.

            \begin{theorem}[{\cite[Theorem 1.15]{De-Philippis-Rindler-flat-conj}}]\label{thme:De-Philippis-Rindler}
                Let $T \in \bM_k(\setR^k)$ be with compact support. Then $T\in \mathbf{F}_{k,c}(\setR^k)$ and $T$ is of the form $\br{\theta}$ for some $\theta\in L^1(\setR^k)$.
            \end{theorem}

            Combining this result with \cite[Proposition 7.4]{Lang}, we get the following.

            \begin{lemma}\label{lemma:mass-measure-abso-cont}
                Let $(X,\dist)$ be a locally compact metric space and $T \in \bM_{k,\loc}(X)$. Then $\norm{T}$ is absolutely continuous with respect to $\haus^k$.
            \end{lemma}

    \subsection{Bound on the mass of the boundary}

        The goal of this section is to provide sufficient geometric conditions on the support of the boundary of a current in order to guarantee that the boundary has locally finite mass, and therefore that the current is locally normal.
        We begin with a technical result.
        
    \begin{lemma}
    \label{lemma:T_X_flat_chain}
        Let $T \in \bM_{k,\loc}(X)$. For every $f \in \Lip_c(X)$ and every $\pi \in \mathcal{D}^{k-1}(X)$, the current $S_f := \pi_\#(\partial T \rstr f) \in \mathcal{D}_{k-1,c}(X)$, is a flat chain. That is, $S_f \in \mathbf{F}_{k-1,c}(\setR^{k-1})$.
    \end{lemma}
    \begin{proof}

        By \eqref{eq:restr-boundary}, we have
        $$S_f = \pi_\#(\partial T\rstr f) = \pi_\#(T\rstr(1,f)) + \pi_\#(\partial (T\rstr f)).$$
        Since $T \rstr f \in \mathcal{D}_k(X)$ then $\pi_{\#}(T \rstr f) = 0$ by Lemma \ref{lemma:mass-measure-abso-cont}. By \eqref{eq:boundary_pushforward_commute}, $\pi_{\#}(\partial(T \rstr f)) = \partial(\pi_{\#}(T \rstr f)) = 0$. Thus, $S_f = \pi_{\#}(T \rstr (1, f))$.
        We notice that $T \rstr (1, f) \in \bM_{k-1}(X)$ since $\bM(T\rstr(1,f)) \le \Lip(f)\norm{T}(\spt(f)) < \infty$ due to the compactness of $\spt(f)$. Here we used that $\spt(T\rstr (1,f)) \subseteq \spt(f)$, as it can be readily checked by \eqref{eq:defin_alternative_support} and Definition \ref{defin:restriction_currents}. By \cite[Lemma 4.6]{Lang}, $\pi_\#(T \rstr (1,f)) \in \bM_{k-1}(\setR^{k-1})$. Theorem \ref{thme:De-Philippis-Rindler} implies that $S_f\in \mathbf{F}_{k-1,c}(\setR^{k-1})$.
    \end{proof}

    This leads to a first condition that ensures that $\partial T$ is actually zero.

    \begin{theorem}
    \label{theo:boundary_zero_mass}
        Let $(X,\dist)$ be a locally compact metric space and let $T\in \bM_{k,\loc}(X)$. If $\haus^{k-1}(\spt(\partial T)) = 0$ then $\partial T = 0$.
    \end{theorem}
    
    \begin{proof}
        Let $J$ be a finite set, let $\{(f^j, \pi_1^j,\ldots, \pi_{k-1}^j)\}_{j\in J} \subseteq \mathcal{D}^{k-1}(X)$ with $\sum_{j\in J} |f^j| \le 1$ and $\Lip(\pi_i^j) \le 1$. The set $K := \bigcup_{j\in J} \spt(f^j)$ is compact. By Lemma \ref{lemma:T_X_flat_chain}, the current $S^j := \pi^j_\#(\partial T\rstr f^j) \in \mathbf{F}_{k-1,c}(\setR^{k-1})$ for every $j\in J$. By Theorem \ref{thme:De-Philippis-Rindler} there exist functions $\theta^j \in L^1(\setR^{k-1})$ such that $\Vert S^j \Vert = \vert\theta^j\vert\haus^{k-1}$, for every $j\in J$.
        Let $N_K := \spt(\partial T) \cap K$, which satisfies  $\haus^{k-1}(N_K) = 0$ and is compact. The set $N^j := \pi^j(N_K) \subseteq \mathbb{R}^{k-1}$ is compact too and satisfies $\haus^{k-1}(N^j) = 0$, so also $\Vert S^j \Vert(N^j) = 0$. By the $L^1$ integrability of $\theta^j$, for every $\varepsilon > 0$, there exists an open neighbourhood $V^j \subseteq \mathbb{R}^{k-1}$ of $N^j$ such that
        \begin{equation}
        \label{eq:estimate_on_mass_theta_j}
            \int_{V^j} \vert{}\theta^j\vert{} \, d\haus^{k-1} < \frac{\varepsilon}{|J|}.
        \end{equation}
        Since $N^j$ is closed, we can find another open neighbourhood $W^j$ of $N^j$ such that $\overline{W^j} \subseteq V^j$.
        Let $\chi^j \colon \setR^{k-1} \to [0,1]$ be a Lipschitz function such that $\chi^j \equiv 1$ on $W^j$ and $\spt(\chi^j) \subseteq V^j$. We define the Lipschitz functions $\phi^j \colon X \to [0,1]$, $\phi^j(x) = 1 - \chi^j(\pi^j(x))$. By construction, $\phi^j \equiv 0$ on $(\pi^j)^{-1}(W^j)$, hence $\spt(\phi^j ) \subseteq X \setminus (\pi^j)^{-1}(W^j) \subseteq X \setminus N_K$.        
        By linearity (see Definition \ref{def:current}) we have
        $$\sum_{j\in J} \partial T(f^j, \pi^j) = \sum_{j\in J} \partial T(f^j \cdot (\chi^j \circ \pi^j), \pi^j) + \sum_{j\in J} \partial T(f^j \phi^j, \pi^j).$$
        For the first term, by definition of pushforward current we get
        $$\partial T(f^j \cdot (\chi^j \circ \pi^j), \pi^j) = S^j(\chi^j, \text{id}).$$
        Thus, using \eqref{eq:estimate_on_mass_theta_j} we obtain
        \begin{equation}
            \label{eq:sum_bad_part_boundary}
            \left\vert\sum_{j\in J} \partial T(f^j \cdot (\chi^j \circ \pi^j), \pi^j)\right\vert \le \sum_{j\in J} \int_{\mathbb{R}^{k-1}} \chi^j \vert{}\theta^j\vert{} \, d\haus^{k-1} \le \sum_{j\in J} \int_{V^j} \vert{}\theta^j\vert{} \, d\haus^{k-1} < \varepsilon.
        \end{equation}
        We now look at the second term. The functions $f^j \phi^j \in \Lip_c(X)$ have support $$\spt(f^j\phi^j) \subseteq \spt(f^j)\cap \spt(\phi^j) \subseteq K \cap (X \setminus N_K) \subseteq X\setminus \spt(\partial T).$$
        Therefore, the second term vanishes.
        This and \eqref{eq:sum_bad_part_boundary} give
        $$\sum_{j\in J} \partial T(f^j, \pi^j) \le \varepsilon.$$
        By the arbitrariness of $f^j, \pi^j$ and $\varepsilon$ we obtain that $\bM(\partial T) = 0$, whence $\partial T = 0$.
    \end{proof}

    The next sufficient condition for the finiteness of the mass of the boundary is of different nature and it allows non-zero boundaries. We recall that the density functions appearing in the statement have been defined in \eqref{eq:defin_densities}.

    \begin{theorem}
    \label{theo:boundary_finite_mass_hard}
    Let $(X,\dist)$ be a locally compact metric space, let $k\ge 1$, and let
    $T\in \mathbf{M}_{k,\loc}(X)$. Assume that:
    \begin{enumerate}
        \item[(i)] $\Theta^k(\cdot,\|T\|)$ is locally bounded above on
        $\spt(\partial T)$;
        \item[(ii)] $\mathcal{H}^{k-1}\llcorner \spt(\partial T)$ is locally finite 
        and, for every compact set $K\subseteq X$, there exists $c_K>0$ such that \[\Theta_{k-1}
        (\cdot,\mathcal{H}^{k-1}\llcorner\spt(\partial T))\ge c_K \] for $\mathcal{H}^{k-1}$-a.e.\ $x\in K\cap \spt(\partial T)$.
        \end{enumerate}
    Then $\partial T\in \mathbf{M}_{k-1,\loc}(X)$ and hence $T\in \mathbf{N}_{k,\loc}(X)$.
    \end{theorem}
    
    \begin{proof}
    It is enough to prove that $\partial T$ has finite mass in a neighbourhood of
    every point $x\in \spt(\partial T)$.    
    Fix $x\in \spt(\partial T)$. By assumptions (i) and (ii) we can find
    $r_x,c_x,C_x,C'_x>0$ such that
    \begin{enumerate}
        \item[(a)] $\overline{B}(x,2r_x)$ is compact;
        \item[(b)] $\Theta^k(y,\norm{T})\le C_x/2$ for every $y\in \spt(\partial T)\cap B(x,2r_x)$;
        \item[(c)] $ \haus^{k-1}(B(x,2r_x) \cap \spt(\partial T))\le C'_x$;
        \item[(d)] $\Theta_{k-1}(y,\haus^{k-1}\rstr \spt(\partial T))
        \ge 2c_x$ for $\haus^{k-1}$-a.e. $y\in \spt(\partial T)\cap B(x,2r_x)$.
    \end{enumerate}
    Thus, for $\haus^{k-1}$-almost every $y\in \spt(\partial T)\cap B(x,r_x)$ there exists
    $r_y>0$ such that
    \begin{equation}\label{eq:density-estimates-416}
        \norm{T}(B(y,t))\le C_x\omega_k t^k,
        \qquad
        \haus^{k-1}(B(y,t) \cap \spt(\partial T))\ge c_x\omega_{k-1} t^{k-1}
    \end{equation}
    for every $0<t<r_y$.
    
    Fix $\varepsilon>0$. By 
    \eqref{eq:density-estimates-416} and the inner regularity of $\haus^{k-1}$,
    we can find a compact set
    \[
    G_\varepsilon\subseteq \spt(\partial T)\cap B(x,r_x)
    \]
    and a number $0<\rho_\varepsilon<r_x$ such that
    \begin{equation}\label{eq:good-complement}
        \haus^{k-1}\big( B(x,r_x) \cap \spt(\partial T)\setminus G_\varepsilon\big)<\varepsilon
    \end{equation}
    and
    \begin{align}
        \norm{T}(B(y,t))&\le C_x \omega_k t^k,
        \label{eq:good-upper-density}\\
        \haus^{k-1}(B(y,t)\cap \spt(\partial T))&\ge c_x \omega_{k-1} t^{k-1}
        \label{eq:good-lower-density}
    \end{align}
    for every $y\in G_\varepsilon$ and every
    $0<t\le \rho_\varepsilon$.
    
    We prove that
    \[
    \bM_{B(x,r_x)}(\partial T)<\infty.
    \]
    Let $J$ be finite and, for every $j\in J$, let
    \[
    (f^j,\pi^j_1,\ldots,\pi^j_{k-1})\in\mathcal{D}^{k-1}(X)
    \]
    satisfy
    \[
    \spt(f^j)\subseteq B(x,r_x),
    \qquad
    \sum_{j\in J}|f^j|\le 1,
    \qquad
    \Lip(\pi^j_i)\le 1.
    \]
    Set
    \[
    K:=\bigcup_{j\in J}\spt(f^j),
    \qquad
    E_K:=\spt(\partial T)\cap K,
    \qquad
    G^K_\varepsilon:=G_\varepsilon\cap K,
    \qquad
    B^K_\varepsilon:=E_K\setminus G^K_\varepsilon.
    \]
    Notice that $K$ and $E_K$ are compact and, by
    \eqref{eq:good-complement},
    \begin{equation}\label{eq:bad-small}
        \haus^{k-1}(B^K_\varepsilon)<\varepsilon.
    \end{equation}
    
    We first construct a cutoff function around the good set $G^K_\varepsilon$.
    Fix $0<\rho<\frac{\rho_\varepsilon}{4}$.
    If $G^K_\varepsilon=\emptyset$, set $\varphi_\rho\equiv 0$.
    Otherwise, let $\{y_1,\ldots,y_N\}$ be a maximal $2\rho$-separated
    subset of $G^K_\varepsilon$. Then the balls $B(y_i,\rho)$ are pairwise
    disjoint and
    \[
    Nc_x\omega_{k-1}\rho^{k-1}
    \stackrel{\eqref{eq:good-lower-density}}{\le}
    \sum_{i=1}^N \haus^{k-1}(B(y_i,\rho)\cap \spt(\partial T))
    \le
    \haus^{k-1}(B(x,2r_x) \cap \spt(\partial T))
    \stackrel{\text{(c)}}{\le} C'_x.
    \]
    Therefore,
    \begin{equation}\label{eq:number-good-balls}
        N\le \frac{C'_x}{c_x\omega_{k-1}\rho^{k-1}}.
    \end{equation}
    By maximality, $G^K_\varepsilon\subseteq\bigcup_{i=1}^N B(y_i,2\rho)$, and hence $B(G^K_\varepsilon,2\rho)
    \subseteq
    \bigcup_{i=1}^N B(y_i,4\rho)$.
    Using \eqref{eq:good-upper-density} and
    \eqref{eq:number-good-balls}, we obtain
    \begin{equation}\label{eq:good-neighbourhood-mass}
        \norm{T}(B(G^K_\varepsilon,2\rho))
        \le
        \sum_{i=1}^N \norm{T}(B(y_i,4\rho))
        \le
        NC_x\omega_k(4\rho)^k
        \le
        \frac{4^k C_x \omega_k C'_x}{c_x \omega_{k-1}}\rho.
    \end{equation}
    
    Define
    \[
    \varphi_\rho(z)
    :=
    \max\left\{
    0,\,
    1-\frac{2}{\rho}
    \max\left\{0,\dist(z,G^K_\varepsilon)-\frac{\rho}{2}\right\}
    \right\}.
    \]
    Then $\varphi_\rho\equiv 1$  on $B(G^K_\varepsilon,\rho/2)$, $\spt(\varphi_\rho)
    \subseteq
    \overline{B}(G^K_\varepsilon,\rho)
    \subseteq B(G^K_\varepsilon,2\rho)$, and $\lip_a\varphi_\rho\le \frac{2}{\rho}$.
    Consequently,
    \begin{equation}\label{eq:good-energy}
        \int_X \lip_a\varphi_\rho\,d\norm{T}
        \le
        \frac{2}{\rho}
        \norm{T}(B(G^K_\varepsilon,2\rho)) \le
        \frac{2\cdot 4^k C_x\omega_kC'_x}{c_x\omega_{k-1}}.
    \end{equation}
    
    We now construct a cutoff function around the bad set $B^K_\varepsilon$.
    For $m\in\mathbb{N}$, set
    \[
    A_m
    :=
    \left\{
    y\in B^K_\varepsilon:
    \norm{T}(B(y,t))\le C_x\omega_kt^k
    \ \text{for every }0<t<\frac1m
    \right\}.
    \]
    By (b) and the definition of upper density, $B^K_\varepsilon=\bigcup_{m\in\mathbb{N}}A_m$.
    Set
    \[
    A'_1:=A_1,
    \qquad
    A'_m:=A_m\setminus\bigcup_{\ell<m}A_\ell,
    \quad m\ge2.
    \]
    Then the sets $A'_m$ are pairwise disjoint and $B^K_\varepsilon=\bigsqcup_{m\in\mathbb N} A'_m$.
    For every $m$, using the definition of Hausdorff measure, choose a countable
    family of balls
    \[
    B_{m,i}:=B(y_{m,i},r_{m,i}),
    \qquad
    y_{m,i}\in A'_m,
    \]
    such that
    \[
    A'_m\subseteq\bigcup_i B_{m,i},
    \qquad
    r_{m,i}<\min\left\{\frac{1}{3m},\rho\right\},
    \]
    and
    \begin{equation}\label{eq:bad-cover}
        \sum_i r_{m,i}^{k-1}
        \le
        a_{k-1}\left(
        \haus^{k-1}(A'_m)+\frac{\varepsilon}{2^m}
        \right),
    \end{equation}
    where $a_{k-1}>0$ is a constant depending only on $k$.
    
    The family
    \[
    B(G^K_\varepsilon,\rho/2)
    \cup
    \bigcup_{m,i}B_{m,i}
    \]
    is an open cover of the compact set $E_K$. Hence there exists a finite set
    $F\subset\mathbb N\times\mathbb N$ such that
    \begin{equation}\label{eq:finite-cover}
        E_K
        \subseteq
        B(G^K_\varepsilon,\rho/2)
        \cup
        \bigcup_{(m,i)\in F} B_{m,i}.
    \end{equation}
    For $(m,i)\in F$, define
    \[
    \psi_{m,i}(z)
    :=
    \max\left\{
    0,\,
    1-\frac{\dist(z,B_{m,i})}{r_{m,i}}
    \right\},
    \]
    and set $ \psi_\rho:=\max_{(m,i)\in F}\psi_{m,i}$.
    Then $\psi_\rho\equiv1$ on
    $\bigcup_{(m,i)\in F}B_{m,i}$. Moreover,
    $\spt(\psi_{m,i})
    \subseteq
    \overline B(y_{m,i},2r_{m,i})
    \subseteq
    B(y_{m,i},3r_{m,i})$, and $\Lip(\psi_{m,i})\le \frac1{r_{m,i}}$.
    Therefore,
    \[
    \lip_a\psi_\rho
    \le
    \sum_{(m,i)\in F}\lip_a\psi_{m,i}
    \le
    \sum_{(m,i)\in F}
    \frac{\chi_{B(y_{m,i},3r_{m,i})}}{r_{m,i}}.
    \]
    Since $y_{m,i}\in A_m$ and $3r_{m,i}<1/m$, we infer
   
   \begin{align}
\int_X\lip_a\psi_\rho\,d\norm{T}
    \le
    \sum_{(m,i)\in F}
    \frac{\norm{T}(B(y_{m,i},3r_{m,i}))}{r_{m,i}}
     &\le
    3^k C_x \omega_k
    \sum_{(m,i)\in F}r_{m,i}^{k-1}
    \notag\\
    &\le
    3^k\omega_k a_{k-1}C_x
    \sum_m\left(
    \haus^{k-1}(A'_m)+\frac{\varepsilon}{2^m}
    \right)
    \notag\\
    &\le
    2\cdot 3^k\omega_k a_{k-1}C_x\varepsilon,  \label{eq:bad-energy}
    \end{align}    
    where in the last inequality we used
    \eqref{eq:bad-small}, the disjointness of the $A'_m$ and \eqref{eq:bad-cover}.
    
    We also record that
    \begin{align}
    \norm{T}(\spt(\psi_\rho))
    \le
    \sum_{(m,i)\in F}
    \norm{T}(B(y_{m,i},3r_{m,i}))
     &\le
    3^kC_x \omega_k
    \sum_{(m,i)\in F} r_{m,i}^k
    \nonumber\\
    &\le
    3^kC_x\omega_k\rho
    \sum_{(m,i)\in F}r_{m,i}^{k-1} \nonumber
    \\
    &\le
    2\cdot 3^k\omega_k a_{k-1}C_x\varepsilon\rho.
    \label{eq:bad-support-mass}
    \end{align}
    
    Let
    \[
    \Phi_\rho:=\max\{\varphi_\rho,\psi_\rho\}.
    \]
    By \eqref{eq:finite-cover}, $\Phi_\rho\equiv1$ on an open neighbourhood of
    $E_K$. Therefore, for every $j\in J$,
    $\spt((1-\Phi_\rho)f^j)\cap E_K=\emptyset$,
    and hence, by the definition of $\spt(\partial T)$, $
    \partial T((1-\Phi_\rho)f^j,\pi^j)=0$.
    Consequently,
    \begin{equation}\label{eq:boundary-localization}
        \partial T(f^j,\pi^j)
        =
        \partial T(\Phi_\rho f^j,\pi^j).
    \end{equation}
    For every $j$, choose
    $\sigma_j\in\Lip_c(X)$ such that $0\le\sigma_j\le1$ and
    $\sigma_j\equiv1$ on a neighbourhood of
    $\spt(f^j)\cap\spt(T)$. Then
    \begin{equation}
    \label{eq:split-boundary}
    \begin{aligned}
    \partial T(f^j,\pi^j)
    \stackrel{\eqref{eq:boundary-localization} + \eqref{eq:defin_boundary}}{=}
    T(\sigma_j,\Phi_\rho f^j,\pi^j)
    \stackrel{\eqref{eq:Leibniz_currents}}{=}
    T(\Phi_\rho\sigma_j,f^j,\pi^j)
    +
    T(f^j\sigma_j,\Phi_\rho,\pi^j).
    \end{aligned}
    \end{equation}
    We estimate first the second term. Since $\lip_a\Phi_\rho
    \le \lip_a\varphi_\rho+\lip_a\psi_\rho$, 
    Proposition \ref{prop:mass-measure-ineq-loc-Lip-constant}, \eqref{eq:good-energy},
    \eqref{eq:bad-energy}, and
    $\sum_j|f^j|\le1$ yield

    \begin{equation}
    \label{eq:second-term-416}
    \begin{aligned}
    \left|
    \sum_{j\in J}
    T(f^j\sigma_j,\Phi_\rho,\pi^j)
    \right|
    \le
    \int_X
    \sum_{j\in J}|f^j|\,
    \lip_a\Phi_\rho\,d\norm{T}
     &\le
    \int_X\lip_a\Phi_\rho\,d\norm{T}
    \\
    &\le
    \frac{2\cdot 4^k C_x\omega_kC'_x}{c_x\omega_{k-1}}
    +
    2\cdot 3^k\omega_k a_{k-1}C_x\varepsilon.
    \end{aligned}
    \end{equation}
    
    It remains to show that the first term in
    \eqref{eq:split-boundary} vanishes as $\rho\to 0$.
    By the mass estimate \eqref{eq:mass-measure-ineq},
    \begin{align}
    \left|
    \sum_{j\in J}
    T(\Phi_\rho\sigma_j,f^j,\pi^j)
    \right|
    &\le
    \left(\sum_{j\in J}\Lip(f^j)\right)
    \norm{T}(\spt (\Phi_\rho)).
    \label{eq:first-term-bound}
    \end{align}
    Now, $\spt(\Phi_\rho)
    \subseteq
    \spt(\varphi_\rho)\cup\spt(\psi_\rho)$. By \eqref{eq:good-neighbourhood-mass} and \eqref{eq:bad-support-mass} we have that
    \[
    \norm{T}(\spt(\Phi_\rho)) \le \frac{4^k C_x \omega_k C'_x}{c_x \omega_{k-1}}\rho + 2\cdot 3^k\omega_k a_{k-1}C_x\varepsilon\rho\longrightarrow0
    \qquad\text{as }\rho\to 0.
    \]
    Therefore, by \eqref{eq:first-term-bound},
    \begin{equation}\label{eq:first-term-vanishes}
    \lim_{\rho\to 0}
    \left|
    \sum_{j\in J}
    T(\Phi_\rho\sigma_j,f^j,\pi^j)
    \right|
    =0.
    \end{equation}
    Combining \eqref{eq:split-boundary},
    \eqref{eq:second-term-416}, and
    \eqref{eq:first-term-vanishes}, and letting $\rho\to 0$, we obtain
    \[
    \left|
    \sum_{j\in J}\partial T(f^j,\pi^j)
    \right|
    \le
    \frac{2\cdot 4^k C_x\omega_kC'_x}{c_x\omega_{k-1}}
    +
    2\cdot 3^k\omega_k a_{k-1}C_x\varepsilon.
    \]
    Finally, letting $\varepsilon\to 0$ and taking the supremum over all
    admissible families $(f^j,\pi^j)$ in the definition of the mass gives
    \[
    \bM_{B(x,r_x)}(\partial T)
    \le
    \frac{2\cdot 4^k C_x\omega_kC'_x}{c_x\omega_{k-1}}
    <\infty.
    \]
    Thus $\partial T\in \bM_{k-1,\loc}(X)$.
    \end{proof}

\section{Rectifiability and integer rectifiable currents}
    In this section we recall the notions of rectifiability, integer rectifiable currents and homology via integral currents. Every argument in this section is well known or is a slight generalization of a well-known fact.
    
    \subsection{Rectifiability}
        Let $(X,\dist)$ be a metric space. A $\haus^k$-measurable set $E \subseteq X$ is called \textit{$k$-rectifiable} if there exist a countable collection of compact sets $K_i \subseteq \setR^k$ and bi-Lipschitz maps $\varphi_i \colon K_i \to X$ such that the $\varphi_i(K_i)$ are pairwise disjoint and
        \begin{equation}\label{eq: def-rect}
            \haus^k\left( E \setminus \bigcup_i \varphi_i(K_i)\right) = 0.
        \end{equation}

        It follows from the fundamental result about rectifiable sets in metric spaces due to Kirchheim (\cite[Lemma 4]{kirchheim94}), that the definition of rectifiability given above is equivalent to the standard definition\footnote{In the standard terminology, a $\haus^k$-measurable set
$E\subseteq X$ is called \emph{$k$-rectifiable} if there exist
countably many Lipschitz maps $\varphi_i\colon B_i\to X$,
with $B_i\subseteq\setR^k$ Borel, such that
$    \haus^k\Bigl(E\setminus\bigcup_i\varphi_i(B_i)\Bigr)=0.$}; see also \cite[Lemma 4.1]{AmbrosioKirchheim00}. The fact that the sets $K_i$ in the next statement can be chosen to be compact follows by inner regularity of $\haus^k$ on $\setR^k$.
        
        \begin{lemma}[{\cite[Lemma 4]{kirchheim94} as stated in \cite[Lemma 7.2]{bate-perturb}}]
        \label{lemma: lip-bilip-decomp}
            Let $B \subseteq \setR^k$ be a Borel set and $\varphi \colon B \to X$ be a Lipschitz map. Then, there exist countably many compact sets $K_i \subseteq B$ such that
            \begin{enumerate}
                \item $\haus^k(\varphi(B) \setminus \bigcup_i \varphi(K_i)) = 0$;
                \item $\varphi$ is bi-Lipschitz on each $K_i$.
            \end{enumerate}
        \end{lemma}

        A consequence of rectifiability is the following density property.

        \begin{lemma}
            \label{lemma:density_points_rectifiable_space}
            Let $E\subseteq (X,\dist)$ be $k$-rectifiable and assume that $\haus^k\rstr E$ is locally finite. Then, $\Theta^k(x,\haus^k\rstr E) = \Theta_k(x,\haus^k \rstr E) = 1$ for $\haus^k$-a.e. $x\in E$.
        \end{lemma}

        \begin{proof}
            This is a direct consequence of \cite[Theorem 9]{kirchheim94} together with \cite[2.10.18(3)]{Federer}; see also \cite[Theorem 2.4]{bate-tangents}.
        \end{proof}
        
        A $\haus^k$-measurable set $P \subseteq X$ is said to be \textit{purely $k$-unrectifiable} if $\haus^k(P \cap E)=0$ for every $k$-rectifiable set $E \subseteq X$. The following deep result due to Bate is known as the perturbation theorem.

        \begin{theorem}\label{thme: bate-perturb}{\normalfont (\cite[Theorem~1.2]{bate-perturb})}
            Let $(X,\dist)$ be a complete metric space and let $P \subseteq X$ be purely $k$-unrectifiable with $\haus^k(P)< \infty$. Then, for every $k \in \setN$ the set 
            $$\{f \in \Lip^1_b(X,\setR^k) \colon \haus^k(f(P)) = 0\}$$
            is residual. 
        \end{theorem}

        Here, $\Lip^1_b(X,\setR^k) = \{f\in\Lip(X,\setR^k) \,:\, \Lip(f)\le 1 \text{ and } f \text{ is bounded}\}$. If $(X,\dist)$ is a complete metric space, then $\Lip^1_b(X,\setR^k)$, equipped with the uniform distance $\dist_\infty$, is also a complete metric space.
        Thus, the Baire category theorem implies that every residual subset of $\Lip^1_b(X,\setR^k)$ is dense. The theorem was first proven in \cite{bate-perturb} with the additional assumption that the lower density of $P$ is positive at almost every point. However, \cite[Theorem 1.5]{bate24} shows that this assumption is not necessary. We record the following consequence of the perturbation theorem; see \cite[Theorem 1.1]{Enrico-david} and \cite[Theorem 1.7]{bate24}

        \begin{theorem}\label{thm: topological-dim-existence-rect}
            Let $(X,\dist)$ be a compact metric space with topological dimension equal to $k$ and $\haus^k(X) < \infty$. Then $X$ contains a $k$-rectifiable subset with positive $\haus^k$-measure. In particular, if every open subset of $X$ has topological dimension $k$, then $X$ contains a dense $k$-rectifiable subset with locally positive $\haus^k$-measure. 
        \end{theorem}

        Every separable and locally compact metric space with locally finite $\haus^k$ measure admits a unique a.e. decomposition into a rectifiable and a purely unrectifiable part.      
        \begin{lemma}\label{lemma: rect-part}
            Let $(X,\dist)$ be separable and locally compact. If $X$ has locally finite $\haus^k$-measure, then there exist $\haus^k$-measurable subsets $E,P \subseteq X$ such that $E$ is $k$-rectifiable, $P$ is purely $k$-unrectifiable and
            $$X = E \cup P.$$
        \end{lemma} 

        Each decomposition as in the lemma is unique up to a set of $\haus^k$-measure zero. Any set $E$ as in the statement is called \emph{the rectifiable part} of $X$.

        \begin{proof}
            Since $X$ is separable and locally compact, there exists a sequence of compact sets $K_i \subseteq X$ such that $K_i \subseteq \textup{int}(K_{i+1})$ and $X = \bigcup_{i\in \setN} K_i$. Each $K_i$ has finite Hausdorff $k$-measure, by assumption. Thus, \cite[Lemma 2.2]{marti2025lipschitz-rigidity} implies that for each $i \in \setN$ there exist $\haus^k$-measurable sets $E_i,P_i \subseteq K_i$ such that $E_i$ is $k$-rectifiable, $P_i$ is purely $k$-unrectifiable and $K_i = E_i \cup P_i.$ It follows that $E = \bigcup_{i\in \setN} E_i$ and $P = \bigcup_{i\in \setN} P_i$ satisfy the required properties.
        \end{proof}

    \subsection{Integer rectifiable currents}

        Let us assume again that $(X,\dist)$ is a locally compact metric space.
        Integer rectifiable currents are defined using the standard example \eqref{eq:std-current} as a   building block. We use the following operative definition; see \cite[Theorem 8.3]{Lang}.
    
    \begin{definition}\label{def: loc-integer-rect}
        A current $T \in \bM_{k,\loc}(X)$ is a locally integer rectifiable current if and only if there exist countably many compact sets $K_i \subseteq \setR^k$, functions $\theta_i \in L^1(\setR^k,\mathbb{Z})$ with $\spt (\theta_i) \subseteq K_i$, and bi-Lipschitz maps $\varphi_i \colon K_i \to X$ such that
        \begin{equation}\label{eq:def-integer-rect-current}
            T = \sum_{i\in \setN} \varphi_{i\#}\br{\theta_i} \quad \textup{and} \quad \norm{T}(B) = \sum_{i\in \setN} \norm{\varphi_{i\#}\br{\theta_i}}(B)
        \end{equation}
        for all Borel sets $B \subseteq X$. Any such collection $\{\varphi_i,K_i,\theta_i\}$ is called a \textit{parametrization} of $T$. 
        The space of locally integer rectifiable $k$-currents in $X$ is denoted by $\mathcal{I}_{k,\loc}(X)$. 
    \end{definition}

    \begin{definition}[Locally integral currents]
    \label{defin:locally_integral_currents}
        We say $T \in \mathcal{I}_{k,\loc}(X)$ is a locally integral current if $\partial T\in \mathcal{I}_{k-1,\loc}(X)$. We write $\bI_{k,\loc}(X)$ for the space of locally integral currents in $X$.
    \end{definition}
        
    By the boundary rectifiability theorem \cite[Theorem 8.7]{Lang}, $\bI_{k,\loc}(X)= \mathcal{I}_{k,\loc}(X) \cap \bN_{k,\loc}(X)$. A locally integral $k$-current $T \in \bI_{k,\loc}(X)$ is said to be a \emph{$k$-cycle} if $\partial T = 0$. The most important examples of integral currents are given by Riemannian manifolds. 
    
    \begin{definition}\label{def: fund-class-riemann}
        Let $M$ be a compact, connected, oriented Riemannian $k$-manifold. Then 
        \[
        \br{M}(f,\pi) = \int_Mf\det(D\pi)\,d\hspace{-0.14em}\haus^k
        \]
        for all \((f, \pi)\in \mathcal{D}^k(M)\), defines an integral $k$-current on $M$. When $M$ has no boundary, Stokes' theorem implies that $\partial \br{M}=0$. Moreover, the mass measure of $\br{M}$ is equal to the Hausdorff $k$-measure on $M$. We call $\br{M}$ the fundamental class of $M$.
    \end{definition}

    Let $\mathcal{I}_k(X) = \mathcal{I}_{k,\loc}(X) \cap \bM_k(X)$ and $\bI_k(X) = \mathcal{I}_{k,\loc}(X) \cap \bN_k(X) $. We call $\mathcal{I}_k(X)$ and $\bI_k(X)$ the space of integer rectifiable currents and integral currents, respectively. There is also a notion of integer rectifiable currents and integral currents for AK currents that coincides with our notion of integer rectifiable currents whenever the underlying space is complete. Indeed, let $T \in \mathcal{I}_k(X)$ and assume that $X$ is complete. Then, $T\in \bM_k^\mathrm{AK}(X)$. Thus, combining Definition \ref{def: loc-integer-rect} with \cite[Theorem 4.5]{AmbrosioKirchheim00} yields that $T$ induces an integer rectifiable current in the sense of Ambrosio--Kirchheim. An analogous arguments holds in case $T \in \bI_k(X)$.

    Now, let $X$ be locally compact and separable. Given $T\in \mathcal{I}_{k,\loc}(X)$, we define the \textit{characteristic set of $T$} by
    $$\set(T) = \left\{ x\in X \,:\, \Theta_k(x,\norm{T})>0\right\}.
    $$
    Suppose that $\bM(T)<\infty $ and $\spt(T)$ is complete. By the above, $T$ is a integer rectifiable current in the sense of AK. Therefore, \cite[Theorem 4.6]{AmbrosioKirchheim00} implies that $E = \set(T)$ is $k$-rectifiable. Moreover, by \cite[Theorem 9.5]{AmbrosioKirchheim00} (see also \cite[Theorem 9.5]{AmbrosioKirchheim00} and the preceding paragraph), there exist $\lambda_E \in L^\infty(\norm{T})$ and $\theta\in L^1(\norm{T})$ such that 
    \begin{equation}\label{eq: represnt-mass-measure-integer-AK}
        \norm{T}= |\theta| \lambda_E \haus^k \rstr E.
    \end{equation}
    The function $\theta$ is defined as follows. Let $\{\varphi_i,K_i,\theta_i\}$ be a parametrization of $T$. Then
    \begin{equation}\label{eq: def-theta}
        \theta := \sum_{i \in \setN} \theta_i \circ \varphi_i^{-1}.    
    \end{equation}
    On the other hand, the function $\lambda_E$ depends only on $E$. More precisely, for $\haus^k$-almost all $x \in E$, the rectifiable set $E$ has an approximate tangent plane $\textup{Tan}(E,x)$, which is a $k$-dimensional normed plane; see \cite[Section 9]{AmbrosioKirchheim00} and \cite{ambrosio-kirchheim-rect}. Note that $\textup{Tan}(E,x)$ is defined using an isometric embedding into $\ell^\infty$. However, for $\haus^k$-almost all $x \in E$, the approximate tangent planes are intrinsic and do not depend on the embedding. The function $\lambda_E$ at $x$ is defined to be the area factor of  $\textup{Tan}(E,x)$ and we will call it the \emph{tangential function} of $E$. We summarize some of its properties.
    \begin{lemma}
        \label{lemma:tangential_function}
        Let $E\subseteq X$ be a $k$-rectifiable set. If $F \subseteq E$ is $k$-rectifiable then $\textup{Tan}(F,x) = \textup{Tan}(E,x)$, and so $\lambda_F(x) = \lambda_E(x)$, for $\haus^k$-almost all $x \in F$. Moreover, $k^{-k/2} \leq \lambda_E \leq 2^k/\omega_k$. If $\textup{Tan}(E,x)$ exists and is isometric to Euclidean $\setR^k$, then $\lambda_E(x) = 1$. Finally, for $\haus^k$-a.e. $x\in E$, the measured Gromov-Hausdorff tangent to $E$ at $x$ is unique and isometric to $\textup{Tan}(E,x)$.
    \end{lemma}
    \begin{proof}
        The first part of the statement follows from \cite[Lemma 5.6]{ambrosio-kirchheim-rect}, while the estimates on the tangential function, see \cite[Lemma 9.2]{AmbrosioKirchheim00}. The next part of the statement is \cite[Lemma 9.2]{AmbrosioKirchheim00}. The last one is \cite[Theorem 6.6]{bate-tangents}.
    \end{proof}
        
    We prove an analogous representation formula for locally integer rectifiable currents.

    \begin{lemma}\label{lemma: charact-set}
        Let $(X,\dist)$ be a locally compact and separable metric space and let $T\in \mathcal{I}_{k,\loc}(X)$. Then $E=\set(T)$ is $k$-rectifiable and
        \begin{equation}\label{eq:mass-meaure-rect-current}
            \norm{T} = |\theta| \lambda_E \haus^k\rstr \set(T),
        \end{equation}
        where $\theta  \in L_{\loc}^1(X,\haus^k)$ is defined as in \eqref{eq: def-theta}.
    \end{lemma}

    \begin{proof}
        Since $X$ is separable and locally compact, there exists a sequence of compact sets $K_i \subseteq X$ such that $K_i \subseteq \textup{int}(K_{i+1})$ and $X = \bigcup_{n\in \setN} K_i$. We can easily find a parametrization $\{\varphi_i^j,K_i^j,\theta_i^j\}_{i,j}$ of $T$ that is adapted to $K_i$, that is, $\varphi_i^j(K_i^j)\subseteq K_i$ for all $i,j\in \setN$. In other words, $\{\varphi_i^j,K_i^j,\theta_i^j\}_{j\in \setN}$ is a parametrization for $T_i := T\rstr K_i$. For every $i \in \setN$, the current $T_i$ is an integer rectifiable current in the sense of AK. Hence, as each $K_i$ is compact, 
        \begin{equation}\label{eq: quality-mass-measure-integer-rect-currents}
            \norm{T}\rstr K_i \stackrel{\eqref{eq:mass_of_restriction}}{=} \norm{T\rstr K_i} = \norm{T_i} = |\theta_i| \lambda_{E_i} \haus^k\rstr E_i 
        \end{equation}
        where $E_i = \set(T_i)$ and $\theta_i$ is defined as in \eqref{eq: def-theta} with respect to the parametrization $\{\varphi_i^j,K_i^j,\theta_i^j\}_{j\in \setN}$. We first show that, up to a set of $\norm{T}$-measure zero, 
        \begin{equation}\label{eq: equality-char-sets}
            E_i = \set(T_i) = \set(T\rstr K_i) = \set(T) \cap K_i.
        \end{equation}
        Since $X$ is separable and locally compact and $\norm{T}$ is locally finite, there exists a countable open cover $\{U_l\}_l$ of $X$ such that $\norm{T}(U_l)<\infty$ for all $l\in \setN$. It follows from \cite[2.10.19(4)]{Federer} that for all $l\in \setN$ and $\norm{T}$-almost all $x \in K_i \cap U_l$
        \begin{align*}
           \liminf_{r \to 0}\frac{\norm{T}(B(x,r)\setminus K_i)}{\omega_k r^k}&= \lim_{r\to 0}\frac{\norm{T}(B(x,r)\setminus K_i)}{\omega_k r^k}
           \\
           &=\lim_{r \to 0}\frac{\norm{T}((B(x,r) \cap  U_l)\setminus K_i)}{\omega_k r^k} = 0. 
        \end{align*}
        The second equality follows from the fact that $U_l$ is open. Therefore, for $\norm{T}$-almost all $x\in \set(T) \cap K_i$ we have
        \begin{align*}
            \liminf_{r \to 0} \frac{\norm{T}(B(x,r))}{\omega_k r^k} &= \liminf_{r \to 0} \frac{\norm{T}(B(x,r)\cap K_i)}{\omega_k r^k} +\lim_{r \to 0} \frac{\norm{T}(B(x,r)\setminus K_i)}{\omega_k r^k}
            \\
            &=\liminf_{r \to 0} \frac{\norm{T}(B(x,r)\cap K_i)}{\omega_k r^k}.
        \end{align*}
        In particular,  
        $$\liminf_{r \to 0} \frac{\norm{T}(B(x,r))}{\omega_k r^k} = \liminf_{r \to 0} \frac{\norm{T\rstr K_i}(B(x,r))}{\omega_k r^k} = \liminf_{r \to 0} \frac{\norm{T_i}(B(x,r))}{\omega_k r^k} $$
        for $\norm{T}$-almost all $x\in K_i$. This, together with \eqref{eq: quality-mass-measure-integer-rect-currents}, proves \eqref{eq: equality-char-sets}. We conclude that $E = \set(T) = \bigcup_i \set(T) \cap E_i$ is $k$-rectifiable. By Lemma \ref{lemma:tangential_function} we have $\lambda_E(x) = \lambda_{E_i}(x)$ for $\haus^k$-almost all $x \in E_i$. Furthermore, by our choice of the parametrization $\{\varphi_i^j,K_i^j,\theta_i^j\}_{i,j}$, it holds that $\theta = \theta_i$ on $E_i$. We conclude that 
        $$\norm{T\rstr E_i} = |\theta| \lambda_E \haus^k\rstr K_i$$
        for all $i \in \setN$. This implies that 
        $$\norm{T} = |\theta| \lambda_E \haus^k\rstr \set(T)$$
        and completes the proof. 
    \end{proof}

    \subsection{Restriction and extension of integer rectifiable currents}

    Restrictions and extensions of locally integer rectifiable currents are locally integer rectifiable. 

    \begin{proposition}
    \label{prop:extension_restriction_integral_currents}
        Let $(X,\dist)$ be a separable, locally compact metric space and let $B \subseteq X$ be locally compact.
        \begin{itemize}
            \item[(i)] If $T\in \mathcal{I}_{k,\loc}(X)$, then $T|_B \in \mathcal{I}_{k,\loc}(B)$.
            \item[(ii)] If $T\in \mathcal{I}_{k,\loc}(B)$ and $\norm{T}(B\cap K)<\infty$ for every compact $K\subseteq X$, then $T^X \in \mathcal{I}_{k,\loc}(X)$.
        \end{itemize}
    \end{proposition}
    \begin{proof}
        Let $T\in \mathcal{I}_{k,\loc}(X)$, so that $T = \sum_i \varphi_{i\#}\br{\theta_i}$ and $\norm{T} = \sum_i \norm{\varphi_{i\#}\br{\theta_i}}$. By linearity of the restriction operator, $T|_B = \sum_i (\varphi_{i\#}\br{\theta_i})|_B$. For  each $i$, by the definition of $\br{\theta_i}$, we have that $$(\varphi_{i\#}\br{\theta_i})|_B \stackrel{\eqref{eq:restriction_vs_pushforward}}{=} (\varphi_{i}|_{\varphi_i^{-1}(B)})_\#(\br{\theta_i}|_{\varphi_i^{-1}(B)}) = (\varphi_{i}|_{\varphi_i^{-1}(B)})_\#(\br{\theta_i|_{\varphi_i^{-1}(B)}}).$$
        Moreover, $\norm{(\varphi_{i\#}\br{\theta_i})|_B} = \norm{(\varphi_{i\#}\br{\theta_i})} \rstr B$ and $\norm{T|_B} = \norm{T}\rstr B$, by Proposition \ref{prop:restriction_of_current_to_open_subsets}.(i). Therefore, $T|_B$ is locally integer rectifiable.

        A similar proof applies to (ii), using the linearity of the extension operator, Lemma \ref{lemma:extension_vs_pushforward} and Corollary \ref{cor:extension_vs_restriction}.(i).
    \end{proof}
    
    \subsection{Slicing}\label{sec: slicing}
        We recall the slicing theory for locally normal $k$-currents; see \cite[Section 6]{Lang}. Let $T \in \bN_{k,\loc}(X)$ and $u \in \Lip_{\loc}(X)$. The slice of $T$ at $t\in \setR$ with respect to $u$ is denoted by $\langle T,u,t\rangle$. Suppose that $\spt(T)$ is separable. By \cite[Theorem 6.2]{Lang}, we have that for almost all $t\in \setR$, the slice $\langle T,u,t\rangle$ is well-defined, it is an element of $\bN_{k-1,\loc}(X)$ and $\spt\big(\langle T, u,t\rangle\big)\subseteq \pi^{-1}(t) \cap \spt(T)$. Moreover, for such $t$ we have
        \begin{equation}\label{eq: char-slice}
            \langle T,u,t\rangle = (\partial T) \rstr \{u \geq t\} - \partial (T \rstr \{u \geq t\}).
        \end{equation}
        Furthermore, if in addition $T\in \bI_{k,\loc}(X)$, then $\langle T, u,t\rangle \in \bI_{k-1,\loc}(X)$ for almost all $t\in \setR$. Indeed, It follows from \cite[Theorem 8.5]{Lang} that for almost all $t \in \setR$, we have $\langle T, u,t\rangle \in \mathcal{I}_{k-1,\loc}(X)$. Since for almost all $t$, we also have $\langle T,u,t\rangle \in \bN_{k-1,\loc}(X)$, it follows from \cite[Theorem 8.7]{Lang} that $\langle T, u,t\rangle \in \bI_{k-1,\loc}(X)$ for almost all $t$.

    \subsection{Homology via integral currents}
        Recall that $\partial T\in \bI_{k-1,\loc}(X)$ whenever $T \in \bI_{k,\loc}(X)$. Therefore, the boundary operator $\partial$ for integral currents maps $\bI_{k,\loc}(X)$ into $\bI_{k-1,\loc}(X)$ and in particular, 
        $$
        \dotsm \overset{\partial_{k+1,}}{\longrightarrow} \bI_{k,\loc }(X) \overset{\partial_{k}}{\longrightarrow} \bI_{k-1,\loc}(X) \overset{\partial_{k-1}}{\longrightarrow} \dotsm \overset{\partial_1}{\longrightarrow} \bI_{0,\loc}(X) 
        $$
        is a chain complex. We write $H_k^\textup{IC}(X)$ for the $k$-th homology group of this chain complex. Given a locally compact subset $A \subseteq X$, we let $$\bI_{k,\loc}(A) = \{ T\in \bI_{k,\loc}(X) \colon \spt(T) \subseteq A\}.$$ 
        Since $\spt(\partial T) \subseteq \spt(T)$, the boundary operator $\partial$ maps $\bI_{k,\loc }(A)$ into $\bI_{k-1,\loc }(A)$. Therefore, $\bI_{\bullet,\loc}(A)$ is a chain subcomplex of $I_{\bullet,\loc}(X)$. For $k \geq 0$, we define the relative chain groups 
        $$\bI_{k,\loc}(X,A) := \bI_{k,\loc}(X)/\bI_{k,\loc}(A).$$ The boundary operator descends to the quotient because of the previous observation. Therefore,
        $$
        \dotsm \overset{\partial_{k+1,}}{\longrightarrow} \bI_{k,\loc }(X,A) \overset{\partial_{k}}{\longrightarrow} \bI_{k-1,\loc}(X,A) \overset{\partial_{k-1}}{\longrightarrow} \dotsm \overset{\partial_1}{\longrightarrow} \bI_{0,\loc}(X,A) 
        $$
        is a chain complex and we denote by $H_k^\textup{IC}(X,A)$ its $k$-th homology group. 
        \begin{lemma}
        \label{lemma:relative_homology_dimension_n}
            Let $(X,\dist)$ be a locally compact metric space such that $\haus^{k+1}(X) = 0$ and let $A\subseteq X$ be locally compact. Then
        $$H_k^\textup{IC}(X) = \{ T \in \bI_{k,\loc}(X) \,:\, \partial T= 0\}$$
        and
        \begin{equation}
            \label{eq:top_dim_relative_homology}
            H_k^\textup{IC}(X,A) = \{ T \in \bI_{k,\loc}(X) \,:\, \spt( \partial T)\subseteq  A\}/ \bI_{k,\loc}(A).
        \end{equation}
        \end{lemma}
        \begin{proof}
            It follows from Lemma \ref{lemma:mass-measure-abso-cont} that $\bI_{k+1,\loc}(X)=0$. The conclusions readily follow.
        \end{proof}

        This homology theory is the analytic counterpart of the singular homology with \emph{locally finite} chains $H^{\textup{lf-sing}}_\bullet$. For example $H^\textup{IC}_n(\setR^n) \cong H_n^\textup{lf-sing}(\setR^n) \cong \mathbb{Z}$, while $H_n^\textup{sing}(\setR^n) = \{0\}$. The analytic counterpart of the classical singular homology theory $H^\textup{sing}_\bullet$ with finite chains is the homology of integral currents with compact support. 
The analogy may be made more precise under additional geometric assumptions. For instance, if $(X,\dist)$ is a locally Lipschitz contractible metric space, then this latter homology and the classical singular homology are isomorphic, see \cite[Theorem 1.1]{Mitsuishi2019}.
        
\section{Metric manifolds}

    A topological $n$-manifold is a Hausdorff, second-countable topological space which is locally homeomorphic to $\setR^n$.
    A \emph{metric $n$-manifold} is a topological $n$-manifold $X$ equipped with a (possibly incomplete) metric $\dist$ inducing its topology and such that the corresponding $n$-dimensional Hausdorff measure $\haus^n$ is locally finite. It follows from  \cite[Théorème 2]{Szpilrajn1937} that $\haus^n$ is always locally positive. The notion of a metric manifold considered here differs from the notion used in \cite{BassoMartiWenger25, Marti2026,marti2024characterization}. Therein, it is additionally assumed that $X$ is homeomorphic to a compact, smooth manifold. In order to use the results available for such metric manifolds, we introduce the following definition.

    \begin{definition}\label{def: good-ball}
        A subset $D \subseteq X$ of a metric space is said to be a \emph{regular disk} if $D$ is homeomorphic to the closed ball $\overline{B}(0,1) \subseteq \setR^n$, $\haus^n(D) < \infty$ and $\haus^n(\partial D) = 0$.
    \end{definition}

    In metric manifolds, regular disks always exist.

    \begin{lemma}\label{lemma: good-subsets-existence}
        Let $(X,\dist)$ be a metric $n$-manifold, let $x\in X$ and $r>0$. Then there exists a regular disk $D \subseteq B(x,r)$ that contains $x$ in its interior.
    \end{lemma}

    \begin{proof}
        Since $X$ is a topological manifold and has locally finite Hausdorff $n$-measure, there exist an open neighbourhood $V \subseteq B(x,r)$ of $x$ that has finite Hausdorff $n$-measure and a homeomorphism $\varphi\colon V \to \setR^n$ with $\varphi(x)=0$. For $\rho\geq 0$, we write $S_\rho = \{ y \in \setR^n \,:\, \norm{y}=\rho\}$ and $D_\rho = \varphi^{-1}(\overline{B}(0,\rho))$. Then $D_\rho$ is homeomorphic to $\overline{B}(0,1)$ and satisfies $\haus^n(D_\rho) < \infty$, for every $\rho > 0$. It remains to show that there exists $\rho > 0$ such that $\haus^n(\partial D_\rho) =0$. For every $\rho>0$,
        we have $\partial D_\rho = \varphi^{-1}(S_\rho)$ and, $\partial D_\rho \cap \partial D_{\rho'} = \emptyset$ for $\rho \neq  \rho'$ nonnegative. Let $B := \bigcup_{\rho \ge 0} \partial D_\rho \subseteq V$. Then $\haus^n(B) <\infty$ and thus, as the $\partial D_\rho$ are pairwise disjoint, the set $I = \{\rho \in (0,\infty) \,:\, \haus^n(\partial D_\rho)>0\}$ is at most countable. In particular, $D_\rho$  is a regular disk for every $r \in (0,\infty) \setminus I$. This completes the proof. 
    \end{proof}

    Every regular disk $D\subseteq X$ is a metric $n$-manifold with boundary, as defined in \cite{Marti2026}. Usually, it is easier to work with closed manifolds. For this, we will often work with the space obtained by gluing two copies of $D$ along their boundaries; cf. \cite[Definition 4.3]{Marti2026}

    \begin{definition}\label{def: double-good-subset}
        Given a regular disk $D \subseteq X$, we let $\hat{D}$ be the quotient $\hat{D}= D_1 \sqcup D_2 / \sim$, where $D_1$ and $D_2$ are two copies of $D$ and $x\sim y$ if and only if $x\in \partial D_1$, $y\in \partial D_2$ and $x=y$. We call $\hat{D}$ the double of $D$ and equip it with the quotient distance. 
    \end{definition}

    If $D$ is a regular disk then $\hat{D}$ is homeomorphic to $\mathbb{S}^n$. Moreover, there exist isometric embeddings $\iota_1,\iota_2\colon D\hookrightarrow \hat{D}$ that identify $D$ with $D_1$ and $D_2$, respectively. If $D$ is oriented, see Section \ref{subsec:orientation_topological}, then we equip $\hat{D}$ with the orientation such that $\iota_1$ and $\iota_2$ are sense-preserving and sense-reversing, respectively. This means $\deg(\iota_1,x)= 1$ and $\deg(\iota_2,x) = -1$ for every $x\in \textup{int}(D)$. See Section \ref{subsec:orientation_topological} for the definition of the local degree.

    \subsection{Orientation and degree theory}
    \label{subsec:orientation_topological}
        We review the topological notion of orientation of a topological manifold and of the degree of continuous maps between oriented topological manifolds. 
        For a detailed exposition of the topic, we refer to \cite{dold-1980}. 
        
        Let $X$ be a topological $n$-manifold. For every $x\in X$, the relative singular homology group $H_n^\textup{sing}(X, X\setminus \{x\})$ with integer coefficients is isomorphic to $\mathbb{Z}$. An orientation $\{o_x\}_{x\in X}$ of an orientable manifold $X$ is a choice of local orientations $o_x$ for each $x\in X$ such that the following continuity property holds. For every $x\in X$, there exists an open neighbourhood $U\subseteq X$ of $x$ and $\eta\in H_n^\textup{sing}(X, X\setminus U)$ such that the homomorphism from $H_n^\textup{sing}(X, X\setminus U)$ to $H_n^\textup{sing}(X, X\setminus \{x\})$ induced by inclusion sends $\eta$ to $o_x$. If such an orientation exists, we say $X$ is orientable. In case $X$ is not connected, we say that $X$ is orientable if each connected component of $X$ is orientable and an orientation of $X$ is an orientation for each component. 
        
        Suppose that $X$ has an orientation $\{o_x\}_{x\in X}$. Then for every compact set $K\subseteq X$, there exists a unique homology class $o_K\in H_n^\textup{sing}(X, X\setminus \hspace{-0.1em}K)$ such that for every $x\in K$, the homomorphism $H_n^\textup{sing}(X, X\setminus \hspace{-0.1em}K)\to H_n^\textup{sing}(X, X\setminus \{x\})$ induced by the inclusion maps $o_K$ to $o_x$. 
                
        Let $f\colon X\to Y$ be a continuous map between oriented topological $n$-manifolds. If $K\subseteq Y$ is a non-empty compact connected set such that $f^{-1}(K)$ is compact, then  \(f_\ast \colon H_n^\textup{sing}(X, X \setminus \hspace{-0.1em}f^{-1}(K)) \to H_n^\textup{sing}(Y, Y\setminus \hspace{-0.1em}K)\) sends \(o_{f^{-1}(K)}\) to an integer multiple of \(o_K\). This integer is called the \textit{degree of $f$ over $K$} and is denoted by $\deg_{K}(f)$. By \cite[Chapter 8, Proposition 4.4]{dold-1980}, we have $\deg_{K}(f) = \deg_y(f)$ for all $y \in K$. In particular, if $X$ is compact and connected, then $\deg_K(f)$ is the same integer for all compact, connected sets $K \subseteq Y$. We denote this number by $\deg(f)$ and call it the \textit{degree of $f$}. Moreover, $H_n^\textup{sing}(X)$ is infinite cyclic and its generator $[X]$ is called the fundamental class of $X$. In case $Y$ is also compact and connected, we have $f_*[X] =  \deg(f)\cdot [Y]$. Otherwise, if $Y$ is non-compact, then it always holds that $\deg(f)=0$ because $f$ is not surjective. 

        Let $X$ and $Y$ be topological $n$-manifolds and let $f\colon X \to Y$ be continuous.
        Let $x\in X$ and $U \subseteq X$ be an open neighbourhood of $x$. Suppose that there exists an open neighbourhood $V\subseteq Y$ of $f(x)=y$ such that $f(U) \subseteq V$ and $y \notin f(U\setminus \{x\})$. Then $f_* \colon H_n^\textup{sing}(U,U\setminus \{x\}) \to H_n^\textup{sing}(V,V\setminus \{y\})$ is well-defined and sends the local orientation $o_x$ to an integer multiple of the local orientation $o_{y}$. We call this number the \textit{local degree of $f$ at} $x$ and denote it by $\deg(f,x)$. Notice that we can always find such neighbourhoods whenever there exists an open neighbourhood $W$ of $x$ such that $W \cap f^{-1}(y)$ is finite.
        
        The local degree satisfies the following additivity property; see \cite[Chapter 8, Proposition 4.7]{dold-1980}. Let $y \in Y$ be a point such that $f^{-1}(y)$ is finite. Then the local degree of $f$ at each $x \in f^{-1}(y)$ is well-defined and we have
        \begin{equation}\label{eq: local-degree-additivity}
            \deg_y(f) = \sum_{x \in f^{-1}(y)} \deg(f,x).
        \end{equation}
        Notice that in case $X$ is compact and connected we have $\deg(f) = \deg_y(f)$ for all $y \in Y$. 
        
        Let $g \colon Y \to Z$ be a continuous map into another oriented topological $n$-manifold $Z$. If the local degrees $\deg(f,x)$ and $\deg(g,f(x))$ are well-defined, then
        \begin{equation}\label{eq: local-degree-multiplicity}
            \deg(g\circ f, x) = \deg(g,f(x)) \cdot \deg(f,x).
        \end{equation}
        We call this the multiplicity property of the local degree; see \cite[Chapter 8, Corollary 4.6]{dold-1980}. Furthermore, the local degree is a homotopy invariant. More precisely, let $h\colon X \to Y$ be another continuous map with $f(x) = h(x)$ and $U \subseteq X$ be an open neighbourhood of $x$. If there exists a homotopy $H$ between $f$ and $h$ satisfying $H(z,t) \neq f(x)$ for all $z \in U \setminus \{x\}$ and every $t$, then $\deg(f,x) = \deg(h,x)$. Similarly, in case $X$ is compact and connected, we have $\deg(f) = \deg(h)$ whenever $h\colon X \to Y$ is a continuous map homotopic to $h$. Finally, let $\psi \colon \setR^n \to \setR^n$ be Lipschitz. Using the homotopy invariance, it is not difficult to show that 
        \begin{equation}\label{eq: local-deg-Rn-diffb}
            \deg(\psi,x) = \deg(D_x\psi) = \sgn(\det(D_x\psi))
        \end{equation}
        whenever $\psi$ is differentiable at $x\in \setR^n$ with non-degenerate derivative.

    \subsection{Local degree of Lipschitz maps between metric $n$-manifolds and porosity}

        We now specialize the local degree theory to Lipschitz maps between metric $n$-manifolds.
        We begin with the following simple observation; cf. \cite[Remark 2.7]{marti2024characterization}.

        \begin{lemma}\label{lemma: bi-lip-local-degree-ae}
            Let $(X,\dist)$ be a metric $n$-manifold and let $\rho \colon X \to \setR^n$ be Lipschitz. Suppose that there exists a compact set $K \subseteq X$ such that $\rho$ restricted to $K$ is bi-Lipschitz. Then, the local degree $\deg(\rho,x)$ is well-defined for almost all $x \in K$.
        \end{lemma}

        \begin{proof}
             Since $X$ is locally compact, we can find an open neighbourhood $U$ of $K$ such that $\haus^n(U) < \infty$. By Lemma \ref{lemma:coarea_inequality} we have that $U\cap \rho^{-1}(y)$ is finite for $\haus^n$-almost all $y\in \setR^n$. Since $\rho|_K$ is bi-Lipschitz, we have that $U\cap \rho^{-1}(\rho(x))$ is finite for $\haus^n$-a.e. $x\in K$. For such $x$, the local degree $\deg(\rho,x)$ is well-defined.
        \end{proof}

        To make full use of local degree theory, we recall the notion of porous sets.

        \begin{definition}[Porous sets]
            Let $(X,\dist)$ be a metric space. A subset $A \subseteq X$ is said to be \textit{porous} at $x \in A$ if there exist $\varepsilon>0$ and a sequence $x_i \to x$ satisfying 
        $$\dist(x_i,A) \geq \varepsilon \dist(x_i,x)$$ 
        for every $i \in \setN$. We call $A$ porous if it is porous at every $x\in A$.
        \end{definition}

        The condition that porous sets have measure zero provides good control on the local degree of Lipschitz maps. This is formulated in the next lemma; cf. \cite[Lemma 4.7]{Marti2026}.

        \begin{lemma}\label{lemma:degree-behaves-well}
            Let $(X,\dist)$ be a metric $n$-manifold and assume that $\haus^n(P) = 0$ for every porous set $P\subseteq X$. Let $K\subseteq X$ be compact, $\rho\colon K \to \setR^n$ be bi-Lipschitz and $\pi \colon X \to \setR^n$ be Lipschitz. Then, for every Lipschitz extension $\psi\colon \setR^n \to \setR^n$ of $\pi \circ \rho^{-1}$ and every Lipschitz extension $\tilde{\rho}\colon X\to \setR^n$ of $\rho$, we have
            $$\deg(\pi,x) = \deg(\psi\circ \tilde{\rho}, x ) = \deg(\psi,\rho(x)) \cdot\deg(\tilde{\rho},x)$$
            for almost 
            every $y\in \setR^n$ and every $x \in \pi^{-1}(y) \cap K$. Implicitly contained in the statement is the assertion that the local degrees in the equality are well-defined.
        \end{lemma}

        \begin{proof}
            We first verify that for almost all $y\in \setR^n$ and every $x \in \pi^{-1}(y) \cap K$, the local degrees in the previous equality are well-defined. Applying Lemma \ref{lemma: lip-bilip-decomp} to $\pi \circ \rho^{-1}\colon \rho(K) \to \setR^n$ we find a decomposition $K = \bigcup_i K_i \cup N$ such that for each $i \in \setN$, the set $K_i$ is compact, $\pi|_{K_i}$ is bi-Lipschitz and $\haus^n(\pi(N)) = 0$. Notice that $\pi(x) = \psi \circ \Tilde{\rho}(x)$ for all $x\in K$. Thus, $\psi \circ \Tilde{\rho}$ restricted to each $K_i$ is bi-Lipschitz as well. Since $\tilde{\rho}|_K = \rho$ is bi-Lipschitz, we conclude that $\psi|_{\rho(K_i)}$ is also bi-Lipschitz for each $i \in \setN$. It follows from Lemma \ref{lemma: bi-lip-local-degree-ae} that for every $i \in \setN$, there exists a set of measure zero $N_i \subseteq K_i$ such that for all $x \in K_i \setminus N_i$, the local degrees $\deg(\pi,x), \deg(\psi\circ \tilde{\rho}, x ), \deg(\psi,\rho(x))$ and $\deg(\tilde{\rho},x)$ are well-defined. Here, we used again that $\rho$ is bi-Lipschitz. Thus, by the multiplicity property of the local degree
            \begin{equation}\label{eq: degree-behaves-well-first-eq-loc-deg}
                \deg(\psi\circ \tilde{\rho}, x ) \stackrel{\eqref{eq: local-degree-multiplicity}}{=} \deg(\psi,\rho(x)) \cdot\deg(\tilde{\rho},x)
            \end{equation}
            for all $i \in \setN$ and every $x \in K_i \setminus N_i$. 
            
            Now, for $i \in \setN$, let $A_i \subseteq K_i $ be the set of points where $K_i$ is porous. Put $K_i' = K_i \setminus (N_i\cup A_i)$ and $N' = N \cup \bigcup_i (N_i \cup A_i)$. We have $\haus^n(\pi(N')) =0$ because porous sets in $X$ have measure zero. Let $y\in \setR^n \setminus \pi(N')$, a condition that holds for $\haus^n$-a.e. point of $\setR^n$. If $y\notin \pi(K)$ then the thesis is true simply because $\pi^{-1}(y) \cap K = \emptyset$. If $y\in \pi(K)$, then $y\in \pi(\bigcup_i K_i') \setminus \pi(N')$. Therefore, any $x\in \pi^{-1}(y) \cap K$ actually belongs to $\bigcup_i K_i'$. It remains to show that $\deg(\pi,x)  = \deg(\psi\circ \tilde{\rho}, x )$ for every $x \in \bigcup_i K_i'$. 
            
            Fix $i \in \setN$ and $x \in K_i'$. Let $L\geq1$ be sufficiently large such that $\pi$ and $\psi \circ \Tilde{\rho}$ are $L$-Lipschitz on $X$ as well as $L$-bi-Lipschitz on $K_i$. Since $K_i$ is not porous at $x$, there exists $r>0$ with the following property. For every $z \in B(x,r)$, there exists $z' \in K_i$ satisfying $\dist(z,z')\leq (3L)^{-2} \dist(x,z)$. For every $z\in B(x,r)$, we have
            \begin{equation*}\label{eq: degree-ineq-1}
                \norm{\pi(x)-\pi(z)} \geq \norm{\pi(x)-\pi(z')} - \norm{\pi(z)-\pi(z')} \geq\frac{7}{9L}\dist(x,z).
            \end{equation*}
            An analogous estimate holds for $(\psi\circ \Tilde{\rho})$ in place of $\pi$. Moreover,
            \begin{align*}\label{eq: degree-ineq-2}
                \norm{\pi(z)-(\psi\circ \Tilde{\rho})(z)}& \leq \norm{\pi(z)-\pi(z')} + \norm{(\psi\circ \Tilde{\rho})(z')-(\psi\circ \Tilde{\rho})(z)}
                \\
                &\leq 2L \dist(z,z') \leq \frac{2}{9L}\dist(x,z).
            \end{align*}
            Now, let $H \colon X \times [0,1]\to \setR^n$ be the straight-line homotopy between $\pi$ and $\psi\circ \tilde{\rho}$. The two inequalities above imply that $H(z,t)\neq \pi(x) = \psi \circ \tilde{\rho}(x)$ for every $z \in B(x,r)\setminus \{x\}$ and every $t\in [0,1]$. The homotopy invariance of the local degree, see the last paragraph of Section \ref{subsec:orientation_topological}, implies that $\deg(\pi,x) = \deg(\psi\circ \tilde{\rho}, x )$. This completes the proof.
        \end{proof}

        We record the following corollary, which will be useful later. 

        \begin{corollary}\label{cor: easy-version-property-ii}
            Let $(X,\dist)$ be a metric $n$-manifold and $\rho, \pi\colon X \to \setR^n$ be Lipschitz. Suppose that $\haus^n(P)=0$ for every porous set $P\subseteq X$. Let $K \subseteq X$ be a compact subset such that $\rho|_K$ and $\pi|_K$ are bi-Lipschitz. If $\deg(\rho, x) = \pm 1$ for almost all $x \in K$, then $\deg(\pi, x) = \pm 1$ for almost all $x \in K$.
        \end{corollary}
    
        \begin{proof}
            Let $\psi \colon \setR^n \to \setR^n$ be a Lipschitz extension of $\pi \circ \rho^{-1} \colon \rho(K) \to \setR^n$. By
             Lemma \ref{lemma:degree-behaves-well} for almost all $x \in K$, we have
            $$\deg(\pi,x) = \deg(\psi\circ \rho, x ) = \deg(\psi,\rho(x)) \cdot\deg(\rho,x).$$
            Furthermore, since $\rho$ and $\pi$ restricted to $K$ are bi-Lipschitz, there exists a $\haus^n$-measurable set $K' \subseteq K$ of full measure such that for all $x \in K'$, the map $\psi$ is differentiable at $\rho(x)$ with non-degenerate differential $D_{\rho(x)}\psi$. Therefore,
            $$\deg(\pi,x) = \deg(\psi,\rho(x)) \cdot\deg(\rho,x) = \sgn(\det (D_{\rho(x)} \psi) ) \cdot\deg(\rho,x) = \pm 1 $$
            for almost all $x \in K$. This completes the proof.
        \end{proof}

        It is relevant to have sufficient conditions under which every porous set of a metric $n$-manifold has $\haus^n$-measure zero. 
        A full characterization of this property for metric measure spaces is provided in \cite[Section 3]{MeraMoranPreissZajicek2003}. 
        In practice, their condition is not easy to check. 
        However, 
        some sufficient conditions can be found in the literature. 
        For instance, porous sets have $\haus^n$-measure zero if $\haus^n$ is locally doubling \cite[Proposition 3.4]{doubling-porous} or if $(X,\dist,\haus^n)$ is a Lipschitz differentiability space \cite[Theorem 2.4]{bate-speight-lip-diffp-space}. We prove now that it is also the case when $\haus^n$ is $\haus^n$-a.e. locally lower Ahlfors $n$-regular.

        \begin{lemma}\label{lemma: lower-ahlfors-no-porous}
            Let $(X,\dist)$ be a metric space with locally finite $\haus^n$ measure. Suppose that for $\haus^n$-almost every $x \in X$ there exist $c>0$ and $R>0$ such that for every $y \in B(x,R)$ and all $0<r<R$ we have
            $$\haus^n(B(y,r)) \geq  cr^n.$$
            Then, every porous set in $X$ has $\haus^n$-measure zero.
        \end{lemma}

        \begin{proof}
            Let $A \subseteq X$ be a porous set and let $x \in A$. We may suppose that there exist $c>0$ and $R>0$ such that for every $y \in B(x,R)$ and all $0<r<R$ we have
            $$\haus^n(B(y,r)) \geq  cr^n.$$
            By definition of porosity, there exist $\varepsilon>0$ and a sequence $x_i  \to x$ satisfying $\dist(x_i,A) \geq \varepsilon \dist(x_i,x)$
            for every $i \in \setN$. For $i \in \setN$, we define $R_i = \dist(x_i,x)$ and $r_i = \varepsilon R_i$. Notice that $B(x_i,r_i) \subseteq B(x,2R_i)\setminus A$ for every $i \in \setN$. Therefore, for $i \in \setN$ sufficiently large
            \begin{align*}
                \frac{\haus^n(B(x,2R_i)\setminus A)}{ 2^nR_i^n} \geq \frac{\haus^n(B(x_i,r_i))}{2^nR_i^n}\geq \frac{c r_i^n}{2^nR_i^n} = \frac{\varepsilon^n}{2^n} c>0.
            \end{align*}
            In particular, 
            $$\limsup_{r\to 0}\frac{\haus^n(B(x,r)\setminus A)}{r^n}>0.$$
            However, by \cite[2.10.19(4)]{Federer} we have
            $$\lim_{r\to 0}\frac{\haus^n(B(x,r)\setminus A)}{r^n}= 0$$
            for $\haus^n$-almost all $x \in A$. We conclude that $\haus^n(A) = 0$.
        \end{proof}

    \subsection{Degree-regular metric manifolds}

        We are interested in metric manifolds having properties that guarantee some control of the geometry. The following definition is justified by the results of the previous section.

        \begin{definition}\label{def: good-metric-mfld}
            A metric $n$-manifold is called \emph{degree-regular} if the following properties are satisfied. 
        \begin{itemize}
            \item[(i)] (Control of the local degree) Whenever $K \subseteq X$ is compact and $\rho\colon K\to \setR^n$ is bi-Lipschitz, then there exists a Lipschitz extension $\tilde{\rho}\colon X \to \setR^n$ such that
            \begin{equation}
                \label{eq:defin_good_rectifiability}
                \deg(\tilde{\rho},x)=\pm 1 \text{ for }\haus^n\text{-a.e. } x\in K.
            \end{equation}
            \item[(ii)] (No porous sets) Every porous set has $\haus^n$-measure zero.
            \end{itemize}
        \end{definition}

        \begin{remark}\label{rmk: metric-mfld}
        Let $(X,\dist)$ be a metric $n$-manifold. The following remarks are in order.
        \begin{enumerate}
            \item A Lipschitz extension of $\rho\colon K \to \setR^n$ as in (i) always exists by the McShane extension theorem. Moreover, given such an extension $\Tilde{\rho}$, Lemma \ref{lemma: bi-lip-local-degree-ae} implies that the local degree $\deg(\rho,x)$ is well-defined for $\haus^n$-almost all $x\in K$. However, it may happen, a priori, that none of the possible extensions $\tilde{\rho}\colon X \to \setR^n$ satisfies \eqref{eq:defin_good_rectifiability}.
            \item Because of condition (ii), property \eqref{eq:defin_good_rectifiability} does not depend on the extension in the following sense: let $\tilde{\rho},\tilde{\rho}'\colon X \to \setR^n$ be two Lipschitz extensions of $\rho\colon K \to \setR^n$. Then $\tilde{\rho}$ satisfy \eqref{eq:defin_good_rectifiability} if and only if $\tilde{\rho}'$ does. This follows from Corollary \ref{cor: easy-version-property-ii}.
            \item If (ii) holds, (i) is equivalent to the same property with $K$ replaced by any finite $\haus^n$-measure Borel set. This follows by the inner regularity of $\haus^n$ on $X$, see \cite[Corollary 3.3.47]{HKST15}.
            \item Being a degree-regular metric $n$-manifold is a local property in the following sense. Let $(X,\dist)$ be a metric $n$-manifold. If it is degree-regular, then $(U,\dist)$ is a degree-regular metric $n$-manifold for every open subset $U\subseteq X$. Vice versa, if for every $x\in X$ there exists an open neighbourhood $U_x$ of $x$ such that $(U_x,\dist)$ is a degree-regular metric $n$-manifold, then $(X,\dist)$ is degree-regular. Indeed, $X$ is separable and hence it can be covered by degree-regular open neighbourhoods $U_i$ associated to a countable subset $\{x_i\} \subseteq X$. If $P\subseteq X$ is porous then $P\cap U_i$ is porous for every $i$. So $\haus^n(P) \le \sum_{i\in \setN} \haus^n(P\cap U_i)) = 0$. Condition (i) can be checked similarly, since the local degree is defined locally.
            \item Condition (i) is known to hold for a vast classes of metrics, e.g. for linearly locally contractible spaces or for spaces satisfying the ETR property (see below for the definitions). This is the content of Lemma \ref{lemma:ETR=>Control_of_the_degree} and Proposition \ref{prop:ETR+LLC_satisfy_good_rectifiability}. We know of no examples of metric $n$-manifolds that do not satisfy (i).
            \item If $n=2$, by \cite[Lemma 5.2]{rado1938absolutely} metric $n$-manifolds
            $X$ automatically satisfy (i). However, even for $n=2$, property (ii) does not hold in general. Indeed, the metric $n$-manifold $X$ in \cite[Theorem A.1]{sormani-wenger-cancellation} contains a (bi-Lipschitz embedding) subset $Z$
            of positive $\haus^2$-measure, see  
             \cite[Lemma A.2]{sormani-wenger-cancellation}, which is porous in $X$, as it can be readily checked.
        \end{enumerate}
    \end{remark}
    
    Moreover, the degree-regularity property can be upgraded from suitably dense submanifolds to the ambient manifold.
    
    \begin{lemma}
    \label{lemma:degree-regularity_from_submanifold_to_manifold}
        Let $(X,\dist)$ be a metric $n$-manifold and let $X'\subseteq X$ be open and such that $\haus^n(X \setminus X') = 0$. Then $(X,\dist)$ is degree-regular if and only if $(X',\dist)$ is.
    \end{lemma}
    \begin{proof}
        If $X$ is degree-regular then $X'$ is degree-regular too by Remark \ref{rmk: metric-mfld}.(4). Assume that $X'$ is degree-regular.
        Let $P\subseteq X$ be porous. The condition $\haus^n(X\setminus X') = 0$, together with the local positivity of $\haus^n$ on $X$, implies that $X'$ is dense. Therefore, $P\cap X'$ is porous as a subset of $X'$. So $\haus^n(P) = \haus^n(P\cap X') = 0$.

        We now check condition (i) in Definition \ref{def: good-metric-mfld} for $X$. Let $K\subseteq X$ be compact and $\rho \colon K \to \setR$ be a bi-Lipschitz map. Let $\tilde{\rho}\colon X \to \setR$ be a Lipschitz extension of $\rho$, which exists by McShane extension Theorem. Then $\tilde{\rho}' :=\tilde{\rho}|_{X'}$ is a Lipschitz extension of the bi-Lipschitz map $\rho':=\rho|_{K\cap X'}$. By assumption and items (2) and (3) of Remark \ref{rmk: metric-mfld}, $\deg(\tilde{\rho}',x) = \pm 1$ for $\haus^n$-a.e. $x\in K \cap X'$. Since the local degree is defined locally, and using the fact that $X'$ is open in $X$ and that $\haus^n(X\setminus X') = 0$, we deduce that $\deg(\tilde{\rho},x) = \pm 1$ for $\haus^n$-a.e. $x\in K$, which shows that $(X,\dist)$ satisfies condition (i) of Definition \ref{def: good-metric-mfld}.
        This concludes the proof.
    \end{proof}

    We extend the definition of the ETR property from \cite{MP26} to metric $n$-manifolds. The metric $n$-manifolds provided by regular points of $\RCD$ spaces and $\GCBA$ spaces satisfy this property.
    \begin{definition}\label{defin:ETR}
        A metric $n$-manifold $(X,\dist)$ satisfies the property ETR if for every $x \in X$ there exists an open set $U \subseteq X$ containing $x$ and a homeomorphism $\varphi \colon U\to \varphi(U) \subseteq \setR^n$ with the following properties:
        \begin{enumerate}
            \item $\varphi$ is Lipschitz;
            \item there exist countably many disjoint compact sets $K_i \subseteq U$ such that $\haus^n(U\setminus \bigcup_i K_i) = 0$ and $\varphi|_{K_i}$ is bi-Lipschitz for each $i \in \setN$.
        \end{enumerate}
    \end{definition}

    This condition implies the degree-regularity, when porous sets have measure zero.

    \begin{lemma}
    \label{lemma:ETR=>Control_of_the_degree}
        Every \textup{ETR} metric $n$-manifold which satisfies Definition \ref{def: good-metric-mfld}.(ii), also satisfies Definition \ref{def: good-metric-mfld}.(i), i.e., it is degree-regular.
    \end{lemma}
    \begin{proof}
        Fix a countable atlas $(U^j,\varphi^j)$ as in Definition \ref{defin:ETR}, with corresponding sets $K^j_i$. We define bi-Lipschitz maps $\rho_i^j\colon K_i^j \to \setR^n$ as $\rho_i^j := \varphi^j|_{K_i^j}$. Then $\haus^n(X\setminus \bigcup_{i,j}K_i^j)=0$, and for every $i,j$ it is enough to take as $\tilde{\rho}^j_i\colon X \to \setR^n$ any Lipschitz extension of $\varphi^j$ to the whole of $X$, which exists by the McShane extension theorem. It is also a Lipschitz extension of $\rho^j_i$. Then each $\rho_i^j$ satisfies \eqref{eq:defin_good_rectifiability} because $\varphi^j$ is a homeomorphism.
        
        Take $K \subseteq X$ compact, $\rho\colon K \to \setR^n$ bi-Lipschitz and let $\tilde{\rho}\colon X \to \setR$ be any Lipschitz extension. By the first part and Corollary \ref{cor: easy-version-property-ii} applied to $\rho|_{K \cap K_i^j}$ and $\rho_i^j|_{K \cap K_i^j}$ we conclude that \eqref{eq:defin_good_rectifiability} holds.
    \end{proof}

    In \cite{BassoMartiWenger25, Marti2026} the theory we will generalize is applied to linearly locally contractible manifolds. This setting fits into our framework too.

    \begin{definition}[Linear local contractibility]
        Let $(X,\dist)$ be a metric space. Let $R>0$ and $\lambda \geq 1$. A subset $B \subseteq X$ is called $(\lambda,R)$-linearly locally contractible if for every $x \in B$ and every $r \in (0,R)$, the ball $B(x,r)$ is contractible within $B(x,\lambda r)$.
    \end{definition}

    We first prove that a linearly locally contractible metric manifold $X$ is locally lower Ahlfors regular as in Lemma \ref{lemma: lower-ahlfors-no-porous}. Notice that this would follow by applying \cite[Theorem 1.2 and Lemma 5.2]{Marti2026} to regular disks in $X$. However, these results rely on the existence of suitable integral currents in regular disks and one goal of this article is to present a unified approach to obtain such currents. Therefore, we present a proof that does not rely on the theory of metric currents.

    \begin{proposition}\label{prop: llc-lower-bound}
            Let $n\in \setN$ and $\lambda \ge 1$. Then there exists $c>0$ satisfying the following property.
            Let $(X,\dist)$ be a metric $n$-manifold which is $(\lambda,R)$-linearly locally contractible. Then, for every $x\in X$, there exists $r_x>0$ such that 
            $$\haus^n(B(y,r)) \geq c r^n$$
            for every $y \in B(x,r_x)$ and $0<r<r_x$.        
    \end{proposition}

    The proof of the proposition is essentially the proof of \cite[Theorem 4.1]{BassoMartiWenger25} applied to regular disks in $X$. We thus only provide a sketch of the proof.

    \begin{proof}
        Fix $x \in X$ and let $0<r_x <R$ be such that $B(x,r_x)$ is contained in a regular disk $D \subseteq X$ as in Definition \ref{def: good-ball}. Such a $r_x$ exists by Lemma \ref{lemma: good-subsets-existence}. Furthermore, let $\hat{D}$ be the closed metric manifold obtained by gluing two copies of $D$ along their boundaries as explained in Definition \ref{def: double-good-subset}. We write $\iota_1 \colon D \hookrightarrow \hat{D}$ for the isometric embedding that identifies $D$ with one of its copies in $\hat{D}$. Put $B \coloneqq  \iota_1(\overline{B}(x,r_x/(2\lambda)))$. If $y\in B$ and $0<r<r_x/(2\lambda)$, then $B(y,\lambda r) \subseteq B(x,r_x) \subseteq X$. Moreover, the corresponding balls in the double $\hat{D}$ stay inside the first copy of $D$. It follows that $B$ is $(\lambda,r_x/(2\lambda))$-linearly locally contractible as a subset of $\hat{D}$. We show that there exists $c=c(n,\lambda)>0$ such that
\[
 \haus^n(B(y,r))\geq c r^n
\]
for every
$y\in B'\coloneqq\iota_1(B(x,r_x/(5\lambda)))$
and every $0<r<r_x/(5\lambda)$.
This proves the proposition after shrinking the local
radius.
       
        The case $n=1$ is not difficult and is left to the reader. We therefore assume that $n\geq 2$. Fix $y \in B'$. For $s>0$, we define $S_s\coloneqq \{y'\in \hat{D}: \dist(y, y')=s\}$. By the coarea inequality, it suffices to prove that
there exists $a=a(n,\lambda)>0$ such that
\[
 \haus^{n-1}(S_s)\geq a s^{n-1}
\]
for every $0<s<r_x/(5\lambda)$.
We choose $a>0$ sufficiently small below.
Suppose, towards a contradiction, that
$\haus^{n-1}(S_s)<a s^{n-1}$ for some such $s$,
and put $S:=S_s$. Embed $\hat{D}$ into $\ell^\infty$. Let $\eta>0$ be such that $\eta< a^\frac{1}{n-1}s$. It follows from \cite[Theorem~3.1]{LLN-22} (see also \cite[Theorem~1.5]{avvakumov2023boxing}) that there exist a finite $(n-2)$-dimensional simplicial complex $\Sigma \subseteq \ell^\infty$ and a Lipschitz map $f\colon S \to \Sigma$ such that
        \begin{equation}\label{eq:homotopy-f}
            \|x - f(x)\|_\infty \leq c_1 \cdot \haus^{n-1}(S)^\frac{1}{n-1} +\eta \leq 2c_1 a^{\frac{1}{n-1}}s
        \end{equation}
        for all $x \in S$, where $c_1\geq 1$ is a constant depending only on $n$. Set $Z \coloneqq f(S)$. Let $g_0\colon Z \to S$ be a (possibly discontinuous) map satisfying $\norm{g_0(z)-z}_\infty\leq 2\dist(z,S)$ for all $z \in Z$. Let $z,z' \in Z$. Then by \eqref{eq:homotopy-f}
        $$\norm{g_0(z)-g_0(z')}_\infty \leq 2(\dist(z,S)+\dist(z',S)) + \norm{z-z'}_\infty\leq 4c's + \norm{z-z'}_\infty,$$
        where $c'=2c_1a^{\frac{1}{n-1}}$. In particular, 
        \begin{equation}\label{eq: g_0-epsilon-cont}
            \norm{g_0(z)-g_0(z')}_\infty \leq 5c's
        \end{equation}
        whenever $\norm{z-z'}_\infty\leq c's$. Since $g_0(Z) \subseteq S \subseteq B$ we may apply \cite[Proposition 4.1]{Marti2026} to obtain a continuous map $g\colon Z \to \hat{D}$ satisfying $\norm{g(z)-g_0(z) }_\infty< 5Qc's$ for all $z \in Z$. Here, $Q>0$ depends only on $\lambda$ and $n$. Notice the condition that $g_0$ is $\varepsilon$-continuous in \cite[Proposition 4.1]{Marti2026} is fulfilled by \eqref{eq: g_0-epsilon-cont} with $\varepsilon= 5c's$. Thus, by decreasing $a>0$ if necessary, we can also guarantee that $\varepsilon < r_x/(5\lambda Q)$. 
        
        \medskip
        
        Let $\iota_{S_s}$ be the inclusion of $S_s$ into $\hat{D}$. Then, by \eqref{eq:homotopy-f}
        \begin{align*}
            \|\iota_{S_s}(w)&-g\circ f(w)\|_\infty
            \\
            &\leq \norm{w -f(w)}_\infty + \norm{f(w) -g_0\circ f(w)}_\infty+\norm{g_0\circ f(w)-g\circ f(w)}_\infty
            \\
            &\leq c's+2\dist(f(w),S) + 5Qc's \leq (3+5Q)c's
        \end{align*}
        for all $w\in S$. Hence, if $a>0$ is sufficiently small, we have $(g\circ f)(S) \subseteq B$ and we can use \cite[Proposition 4.1]{Marti2026} again to find a homotopy between $\iota_{S_s}$ and $h =g\circ f$ whose image lies in the $s/4$-neighbourhood of $S$ (see the comment after the statement of the proposition). Consider the following commutative diagram in \v{C}ech cohomology:
 
        \begin{center}
        \begin{tikzcd}
        \check{H}^{n-1}(h(S_s);\mathbb Z_2)  \arrow{r}{h^\ast} \arrow{d}{g^\ast} &\check{H}^{n-1}(S_s;\mathbb Z_2)  \\
        \check{H}^{n-1}(Z;\mathbb Z_2) \arrow[ur, swap, "f^\ast"] 
        \end{tikzcd}
        \end{center}
        Since $Z$ is a compact subset of the $(n-2)$-dimensional simplicial complex $\Sigma$ its covering dimension is at most $(n-2)$. Therefore, $\check{H}^{n-1}(Z; \mathbb Z_2)$ is trivial and thus the induced homomorphism $h^*$ vanishes. However, for $z\in\iota_1(\partial D)$ we have
\[
 \dist_{\hat D}(y,z)
 \geq r_x-\dist_{\hat D}(x,y)
 >r_x-\frac{r_x}{5\lambda}>2s.
\]
Consequently, $S_s$ separates $y$ and $z$, and its
$s/4$-neighbourhood avoids both points. By the above, the inclusion \(\iota_{S_s}\) is homotopic to \(h\) through maps whose images do not meet \(\{y, z\}\). Therefore, \cite[Lemma~2.1]{Kinneberg-2018} implies that $h^*$ is non-trivial, which is a contradiction. 
We have therefore proved that
$\haus^{n-1}(S_s)\geq a s^{n-1}$ for every
$0<s<r_x/(5\lambda)$.
Fix $0<r<r_x/(5\lambda)$.
If $\haus^n(B_{\hat D}(y,r))=\infty$, the desired
lower bound is immediate. Otherwise, since $\hat D$
is compact and $B_{\hat D}(y,r)$ is Borel, the
measurability assertion in
\cite[Theorem 1.1]{Esmayli_Hajlasz_2021}, applied to
the distance function $\dist_{\hat D}(y,\cdot)$, shows
that $s\mapsto\haus^{n-1}(S_s)$ is Lebesgue measurable
on $(0,r)$. The coarea inequality therefore gives
\[
 \haus^n(B_{\hat D}(y,r))
 \geq
 \frac{\omega_n}{\omega_{n-1}\omega_1}
 \int_0^r \haus^{n-1}(S_s)\,ds
 \geq
 \frac{a\omega_n}{n\omega_{n-1}\omega_1}r^n.
\]
The balls under consideration coincide isometrically
with the corresponding balls in $X$. Thus the
proposition follows with
$\rho_x=r_x/(5\lambda)$ as the local radius and
$c=a\omega_n/(n\omega_{n-1}\omega_1)$.
    \end{proof}

    Therefore, linearly locally contractible metric $n$-manifolds are degree-regular.
        
    \begin{proposition}
    \label{prop:ETR+LLC_satisfy_good_rectifiability}
         Let $(X,\dist)$ be a linearly locally contractible metric $n$-manifold. Then $X$ is degree-regular.
    \end{proposition}

    \begin{proof}
        Definition \ref{def: good-metric-mfld}.(i) is a consequence of \cite[Lemma 4.6, Lemma 4.7 and the comment after the statement]{Marti2026}. The statement is for closed manifolds, but the argument is local and applies equally well to our case, see Remark \ref{rmk: metric-mfld}.(4). Furthermore, Lemma \ref{lemma: lower-ahlfors-no-porous} together with Proposition \ref{prop: llc-lower-bound} directly implies that porous sets in $X$ have $\haus^n$-measure zero. 
    \end{proof}

\section{Orientability of degree-regular metric manifolds}
\label{sec:orientability_of_manifolds}

    Let $(X,\dist)$ be a degree-regular metric $n$-manifold as defined in Definition \ref{def: good-metric-mfld}. We denote by $E$ the rectifiable part of $X$, which is well-defined up to $\haus^n$-measure zero sets, cf. Lemma \ref{lemma: rect-part}. We recall that we denote by $\lambda_E \in L^\infty(E)$ the tangential function of $E$. The goal of this section is to prove that the topological and analytical notions of orientability coincide for degree-regular metric manifolds.

    \begin{theorem}\label{thme: characterization-orientability-metric-manifold}
        Let $(X,\dist)$ be a degree-regular metric $n$-manifold. The following two properties are equivalent.
        \begin{itemize}
            \item[(i)] $X$ is topologically orientable.
            \item[(ii)] There exists $T\in \mathcal{I}_{n,\textup{loc}}(X)$ satisfying $\partial T = 0$ and $\norm{T} = \lambda_E\haus^n \rstr E$.
        \end{itemize}
    \end{theorem}

    The condition on the mass measure in (ii) guarantees that $T$ is supported on the whole of $X$. Indeed, $E$ is dense in $X$ by Theorem \ref{thm: topological-dim-existence-rect}, and thus $\spt(T) = X$ since $\lambda_E>0$ and $\haus^n\rstr E$ has locally positive measure by Theorem \ref{thm: topological-dim-existence-rect}. We also note that if $X$ is $n$-rectifiable, then $E = X$. Before we prove Theorem \ref{thme: characterization-orientability-metric-manifold}, we discuss the orientability of regular disks.

    \subsection{Orientability of regular disks}\label{sec: orient-regular-disks}
        A regular disk $D$ (cf. Definition \ref{def: good-ball}) is called degree-regular if $\textup{int}(D)$ is degree-regular in the sense of Definition \ref{def: good-metric-mfld}. Notice that every regular disk $D$ in a degree-regular metric manifold $X$ is degree-regular. 

        \medskip

        We begin with the following construction. Let $X$ a metric space (not necessarily a metric manifold) and $D \subseteq X$ be a regular disk that is degree-regular.
        We fix an orientation $\mathfrak{o}$ of $\textup{int}(D)$, which is homeomorphic to the open ball $B(0,1)\subseteq \setR^n$. By Lemma \ref{lemma: rect-part}, we can write $D = E_D \cup P_D$ where $E_D$ is $n$-rectifiable and $P_D$ is purely $n$-unrectifiable. Let $K_i \subseteq \textup{int}(D)$ be countably many, pairwise disjoint, compact sets such that $\haus^n(E_D\setminus \bigcup_i K_i) = 0$ and let $\rho_i\colon K_i \to \setR^n$ be bi-Lipschitz maps. This is possible since $\haus^n(\partial D) = 0$. For each $i \in \setN$, let $\tilde{\rho}_i\colon \textup{int}(D) \to \setR^n$ be a Lipschitz extension of $\rho_i$ and $$\theta_i^\mathfrak{o} \colon \rho_i(K_i) \to \mathbb{Z}, z \mapsto \theta_i^\mathfrak{o}(z) :=\deg(\tilde{\rho}_i,\rho_i^{-1}(z)),$$ 
        where the degree is computed with respect to the orientation $\mathfrak{o}$ on $\textup{int}(D)$. We define
        \begin{equation}\label{eq: definition-T_U}
            T_D^\mathfrak{o} := \sum_{i\in \setN} \rho^{-1}_{i\#} \br{\theta_i^\mathfrak{o}}.
        \end{equation}
        Since $D$ is assumed to be degree-regular, Remark \ref{rmk: metric-mfld}.(2) implies that the definition of $T_D^\mathfrak{o}$ does not depend on the extensions $\tilde{\rho}_i$ of the functions $\rho_i$, and that $\theta_i^\mathfrak{o} = \pm 1$ $\haus^n$-a.e. on $\rho_i(K_i)$. The next result establishes that $T_D^\mathfrak{o}$ does not depend on the choice of the compact sets $K_i$ and of the bi-Lipschitz maps $\rho_i$. Moreover,  it is a well-defined, canonical integer rectifiable current on $D$ that satisfies natural properties.

        \begin{proposition}\label{prop: characterization-T_U}
            The current $T_D^\mathfrak{o}$ defined in \eqref{eq: definition-T_U} is integer rectifiable and satisfies
            \begin{equation}\label{eq: mass-measure-T_U}
                \norm{T_D^\mathfrak{o}} = \lambda_{E_D} \haus^n\rstr E_D.
            \end{equation}
            Moreover, for all $(f,\pi)\in \mathcal{D}^n(D)$ such that $\haus^n(\pi(P_D)) = 0$, we have
            \begin{equation}\label{eq: chracterization-T_U}
                T_D^\mathfrak{o}(f,\pi) =  \int_{\setR^n} \sum_{x \in \pi^{-1}(y) \cap \textup{int}(D)}f(x) \deg(\pi,x) \, d\haus^n(y).
            \end{equation}
            The equality \eqref{eq: chracterization-T_U} completely characterizes $T_D^\mathfrak{o}$, and $T_D^\mathfrak{o}$ does not depend on the choice of the rectifiable charts $(K_i,\rho_i)$ used in the definition.
            Furthermore, $T_D^{-\mathfrak{o}} = -T_D^\mathfrak{o}$, where $-\mathfrak{o}$ is the orientation of $\textup{int}(D)$ opposite to $\mathfrak{o}$.
            Finally, if $D' \subseteq X$ is another regular disk with orientation $\mathfrak{o}'$ such that $D'$ is degree-regular and $\mathfrak{o} = \mathfrak{o}'$ on $D \cap D'$, then $T_D^\mathfrak{o}|_ B = T_{D'}^{\mathfrak{o}'}|_ B$ for every Borel $B\subseteq D\cap D'$.
        \end{proposition}
    
        \begin{proof}
            By Definition \ref{def: good-metric-mfld}.(i), we have $|\theta_i^\mathfrak{o}(z)| = 1$ for every $i \in \setN$ and almost every $z \in \rho_i(K_i)$. We recall that the sets $K_i$ are pairwise disjoint. Thus, Lemma \ref{lemma: charact-set} implies $\norm{T_D^\mathfrak{o}} = \lambda_{\textup{set}(T_D^\mathfrak{o})} \haus^n\rstr \textup{set}(T_D^\mathfrak{o})$. By Lemma \ref{lemma:tangential_function}, in order to conclude \eqref{eq: mass-measure-T_U} we need to show that $\haus^n(E_D \setminus\textup{set}(T_D^\mathfrak{o})) = 0$. Let $K_i$, $\rho_i$ and $\tilde{\rho}_i$ as in the definition of $T_D^\mathfrak{o}$. Let $i\in \setN$ and let $x$ be a $\haus^n$-density point of $K_i$, that is, $\Theta^n(x,\haus^n\rstr K_i)=\Theta_n(x,\haus^n\rstr K_i)=1$. Notice that $\haus^n$-a.e. point of $E_D$ is of this form. Call $L_i$ the bi-Lipschitz constant of $\rho_i$. We estimate
            \begin{equation*}
                \begin{aligned}
                    \norm{T_D^\mathfrak{o}}(B(x,r)) \ge \norm{\rho_{i\#}^{-1}\br{\theta_i^\mathfrak{o}}}(B(x,r)) &\ge \frac{1}{L_i^n}\norm{\br{ \theta_i^\mathfrak{o}}}(\rho_i(B(x,r))) \\
                    &\ge \frac{1}{L_i^n}\norm{\br{ \theta_i^\mathfrak{o}}}(B(\rho_i(x),r/L_i)) \\
                    &\ge \frac{1}{L_i^n}\haus^n(B(\rho_i,r/L_i)\cap \rho_i(K_i)).
                \end{aligned}
            \end{equation*}
            Since $x$ is a density point of $K_i$, $\rho_i(x)$ is a density point of $\rho_i(K_i)$. Therefore, taking the limit as $r \to 0$ we get
            $$\liminf_{r\to 0}\frac{\norm{T_D^\mathfrak{o}}(B(x,r))}{r^n} \ge \frac{1}{L_i^{2n}} > 0.$$
            This means that $x\in \textup{set}(T_D^\mathfrak{o})$ and concludes the proof of \eqref{eq: mass-measure-T_U}.
            
            Now, let $(f,\pi)\in \mathcal{D}^n(D)$ be such that $\haus^n(\pi(P_D)) = 0$ and let $i \in \setN$. For every Lipschitz extension $\psi\colon \setR^n \to \setR^n$ of $\pi\circ \rho_i^{-1}$ we have
             \begin{align*}
                   \rho^{-1}_{i\#} \br{\theta_i}(f,\pi) &= \int_{\rho_i(K_i)} f(\rho_i^{-1}(z)) \theta_i(z) \det(D_z \psi) \, d\haus^n(z)
                    \\
                    &= \int_{\rho_i(K_i)} f(\rho_i^{-1}(z)) \deg(\tilde{\rho}_i,\rho_i^{-1}(z))\sgn(\det(D_z \psi)) |\det(D_z \psi)| \, d\haus^n(z)
                    \\
                    &=\int_{\rho_i(K_i)} f(\rho_i^{-1}(z)) \deg(\tilde{\rho}_i,\rho_i^{-1}(z))\deg(\psi,z) |\det(D_z \psi)|  d\haus^n(z).
                \end{align*}
                In the last equality we used \eqref{eq: local-deg-Rn-diffb}.
                The area formula, Lemma \ref{lemma: area-formula}, yields
                \begin{align*}
                    \int_{\rho_i(K_i)}& f(\rho_i^{-1}(z)) \deg(\tilde{\rho}_i,\rho_i^{-1}(z))\deg(\psi,z) |\det(D_z \psi)|  d\haus^n(z)
                    \\
                    &= \int_{\setR^n} \sum_{z \in \psi^{-1}(y) \cap \rho_i(K_i)}f(\rho_i^{-1}(z)) \deg(\tilde{\rho}_i,\rho_i^{-1}(z))\deg(\psi,z) \, d\haus^n(y).
                \end{align*}
                Since $\psi=\pi \circ \rho_i^{-1}$ on $\rho_i(K_i)$ and $\rho_i \colon K_i \to \rho_i(K_i)$ is bi-Lipschitz, we have
                
                \begin{align*}
                   \int_{\setR^n} &\sum_{z \in \psi^{-1}(y) \cap \rho_i(K_i)}f(\rho_i^{-1}(z)) \deg(\tilde{\rho}_i,\rho_i^{-1}(z))\deg(\psi,z) \, d\haus^n(y)
                    \\
                    &= \int_{\setR^n} \sum_{x \in \pi^{-1}(y) \cap K_i} f(x)\deg(\tilde{\rho}_i,x) \deg(\psi,\rho(x)) \, d\haus^n(y).
                \end{align*}
    
                Thus, Lemma \ref{lemma:degree-behaves-well} implies
        
                \begin{align*}
                    \int_{\setR^n} &\sum_{x \in \pi^{-1}(y) \cap K_i} f(x)\deg(\tilde{\rho}_i,x) \deg(\psi,\rho(x)) \, d\haus^n(y)
                    \\
                    &=  \int_{\setR^n} \sum_{x \in \pi^{-1}(y) \cap K_i}f(x) \deg(\pi,x) \, d\haus^n(y).
                \end{align*}
                Finally, using that the $K_i \subseteq \textup{int}(D)$ are pairwise disjoint, $\haus^n(\pi(\textup{int}(D)\setminus \bigcup_i K_i)) = \haus^n(\pi(E_{D} \setminus \bigcup_i K_i)) + \haus^n(\pi(P_{D})) = 0$ and \eqref{eq: definition-T_U}, we get
                $$T_D^\mathfrak{o}(f,\pi) = \sum_{i \in \setN} \rho^{-1}_{i\#} \br{\theta_i}(f,\pi) = \int_{\setR^n} \sum_{x \in \pi^{-1}(y) \cap \textup{int}(D)}f(x) \deg(\pi,x) \, d\haus^n(y),$$
                which is \eqref{eq: chracterization-T_U}. Theorem \ref{thme: bate-perturb} says that the set 
                $$ \{\pi \in \Lip(D,\setR^n) \colon \Lip(\pi) \leq L, \; \haus^n(\pi(P_D)) = 0\}$$ 
                is dense in $\Lip_L(D,\setR^n)$ with respect to the supremum norm, for every $L>0$. Hence, Definition \ref{def:current}.(1) implies that \eqref{eq: chracterization-T_U} completely characterizes $T_D^\mathfrak{o}$. This immediately gives that the definition of $T_D^\mathfrak{o}$ does not depend on the choice of the rectifiable charts $(K_i,\rho_i)$ used in the definition. Moreover, if we flip the orientation of $D$ to $-\mathfrak{o}$, the degree changes sign and \eqref{eq: chracterization-T_U} gives that $T_D^{-\mathfrak{o}} = -T_D^\mathfrak{o}$. The characterization provided by \eqref{eq: chracterization-T_U} is purely topological. This proves the last statement of the thesis.
        \end{proof}
    
        The next step is to show that $\spt(\partial T_D^\mathfrak{o}) \subseteq \partial D$. In order to do that, we consider the closed metric $n$-manifold $\hat{D}$ obtained by gluing two copies of $D$ along their boundaries, as explained in Definition \ref{def: double-good-subset}. Let $\iota_1,\iota_2\colon D \hookrightarrow\hat{D}$ be the isometric embeddings that identify $D$ with its copies $D_1$ and $D_2$ in $\hat{D}$, respectively. Recall that $\hat{D}$ is homeomorphic to $\mathbb{S}^n$. In particular, $\hat{D}$ is a closed metric $n$-manifold in the sense of \cite{marti2024characterization}. We fix an orientation and a Riemannian metric on $\mathbb{S}^n$. We denote by $\br{\mathbb{S}^n} \in \bI_n(\mathbb{S}^n)$ the fundamental class of $\mathbb{S}^n$; cf. Definition \ref{def: fund-class-riemann}.
    
        \begin{proposition}\label{prop: existence-metric-fund-class-double}
            The current defined as $\hat{T}_D^\mathfrak{o} = \iota_{1\#} T_D^\mathfrak{o} - \iota_{2\#}T_D^\mathfrak{o}$ satisfies $\partial \hat{T}_D^\mathfrak{o} = 0$ and 
            $$\varphi_\# \hat{T}_D^\mathfrak{o} = \deg(\varphi)\cdot \br{\mathbb{S}^n}$$
            for every Lipschitz map $\varphi\colon \hat{D}\to \mathbb{S}^n$. In particular, $\spt(\partial T_D^\mathfrak{o}) \subseteq \partial D$.
        \end{proposition}
    
        The proposition implies that $\hat{T}_D^\mathfrak{o}$ is a metric fundamental class of $X$ as defined in \cite[Definition 1.1]{marti2024characterization}.    
    
        \begin{proof}
            By Definition \ref{def: double-good-subset}, $\iota_1$ and $\iota_2$ are sense-preserving and sense-reversing, respectively. Recall that $D = E_D \cup P_D$ where $E_D$ is $n$-rectifiable and $P_D$ is purely $n$-unrectifiable. Put $\hat{P}_D= \iota_1(P_D) \cup \iota_2(P_D)$. Let $\pi \colon \hat{D}\to \setR^n$ be Lipschitz such that $\haus^n(\pi(\hat{P}_D))=0$. Proposition \ref{prop: characterization-T_U} implies that
            \begin{align*}
                \partial \hat{T}_D^\mathfrak{o}(\pi) &= \hat{T}_D^\mathfrak{o}(1,\pi) = T_D^\mathfrak{o}(1,\pi \circ\iota_1)-T_D^\mathfrak{o}(1,\pi \circ\iota_2)
                \\
                &= \int_{\setR^n} \sum_{x \in (\pi \circ\iota_1)^{-1}(y)\cap \textup{int}(D)} \deg(\pi \circ\iota_1,x) - \sum_{x \in (\pi \circ\iota_2)^{-1}(y)\cap \textup{int}(D) } \deg(\pi \circ\iota_2,x) \, d\haus^n(y)
                \\
                &= \int_{\setR^n} \sum_{x \in \pi^{-1}(y)\cap \textup{int}(D_1) } \deg(\pi,x) - \sum_{x \in \pi^{-1}(y)\cap \textup{int}(D_2) } - \deg(\pi,x) \, d\haus^n(y)
                \\
                &= \int_{\setR^n} \sum_{x \in \pi^{-1}(y) } \deg(\pi,x)  \, d\haus^n(y).
            \end{align*}
            In the last step we used that $\haus^n(\partial D_1) = \haus^n(\partial D_2) = 0$. It follows from the additivity property of the degree, \eqref{eq: local-degree-additivity}, that
            $$ \partial \hat{T}_D^\mathfrak{o}(\pi)  = \int_{\setR^n} \sum_{x \in \pi^{-1}(y) } \deg(\pi,x)  \, d\haus^n(y) = \int_{\setR^n} \deg(\pi)  \, d\haus^n(y).$$
            Since $\hat{D}$ is a closed manifold, we have $\deg(\pi) = 0$ and thus $\partial \hat{T}_D^\mathfrak{o}(\pi) = \hat{T}_D^\mathfrak{o}(1,\pi) = 0$. Notice that $\hat{P}_D$ is purely $n$-unrectifiable. Combining Theorem \ref{thme: bate-perturb} 
            and Definition \ref{def:current}.(1) as in the proof of Proposition \ref{prop: characterization-T_U},
            we get that
            $$\partial \hat{T}_D^\mathfrak{o}(\pi) = 0$$
            for every Lipschitz map $\pi\colon \hat{D}\to \setR^n$.
    
            We prove the second part of the statement. Let $\varphi\colon \hat{D} \to \mathbb{S}^n$ be Lipschitz. We first assume that $\haus^n(\varphi(\hat{P}_D))=0$. Arguing as above and using Lemma \ref{prop: characterization-T_U}, yields
            $$\varphi_\#\hat{T}_D^\mathfrak{o}(g,\tau) = \int_{\setR^n} \sum_{x \in \varphi^{-1}(\tau^{-1}(y))} g(\varphi(x)) \deg(\tau \circ \varphi, x) \, d\haus^n(y)$$
            for all $(g,\tau) \in \mathcal{D}^n(\mathbb{S}^n)$. By the coarea inequality, Lemma \ref{lemma:coarea_inequality}, for almost every $y \in \setR^n$ the preimage $\tau^{-1}(y)$ is finite, does not intersect $\varphi(\hat{P}_D)$ and for every $z \in \tau^{-1}(y)$ the preimage $\varphi^{-1}(z)$ is finite as well. Furthermore, it follows from the area formula, Lemma \ref{lemma: area-formula}, that for almost every such $y$ and every $z \in \tau^{-1}(y)$, the derivative of $\tau$ at $z$ exists and is nondegenerate. Therefore,
            $$\deg(\tau \circ \varphi,x) = \deg(\tau,z) \cdot \deg(\varphi,x) = \sgn(\det(D_z \tau)) \cdot \deg(\varphi,x)$$
            for almost every $y \in \setR^n$, every $z \in \tau^{-1}(y)$ and each $x \in \varphi^{-1}(z)$. The additivity property of the degree and the fact that $\hat{P}_D\cap \varphi^{-1}(\tau^{-1}(y)) = \emptyset$ holds for almost every $y \in \setR^n$ imply
            \begin{align*}
                \varphi_\#\hat{T}_D^\mathfrak{o}(g,\tau) &= \int_{\setR^n} \sum_{z \in \tau^{-1}(y)} \left( \sum_{x \in \varphi^{-1}(z)} \deg(\varphi,x) \right)  g(z) \sgn(\det D_z\tau) \, d\haus^n(y)
                \\
                &= \deg(\varphi) \cdot \int_{\setR^n} \sum_{z \in \tau^{-1}(y)} g(z) \sgn(\det D_z\tau) \,d \haus^n(y).
            \end{align*}
            Finally, using the area formula again, Lemma \ref{lemma: area-formula}, we conclude that
            $$\varphi_\#\hat{T}_D^\mathfrak{o}(g,\tau) = \deg(\varphi) \cdot\int_{\mathbb{S}^n} g(z) \det(D_z \tau) \, d\haus^n(z) = \deg(\varphi) \cdot\br{\mathbb{S}^n}(g,\tau).$$
            This proves the second part of the statement under the additional assumption that $\haus^n(\varphi(\hat{P}_D))=0$. Now, let $\varphi\colon \hat{D}\to \mathbb{S}^n$ be an arbitrary Lipschitz map. We embed $\mathbb{S}^n$ into some $\setR^m$ via a bi-Lipschitz embedding; see e.g \cite[Theorem 3.10]{Heinonen-geom-embeddings}. Its image in $\setR^m$ is a Lipschitz neighbourhood retract by \cite[Theorem 3.1]{hohti-1993}. Thus, it follows from Theorem \ref{thme: bate-perturb} that there exists a sequence of maps $\varphi_k \colon \hat{D} \to \mathbb{S}^n$ with uniformly bounded Lipschitz constant, that converges uniformly to $\varphi$ and such that $\haus^n(\varphi_k(\hat{P}_D)) = 0$ for every $k \in \setN$. For $k \in \setN$ sufficiently large, $\varphi$ and $\varphi_k$ are homotopic and in this case we have $\deg(\varphi) = \deg(\varphi_k)$. Hence, by the above and Definition \ref{def:current}.(1) we get
            $$\varphi_\#\hat{T}_D^\mathfrak{o}(g,\tau) = \lim_{k \to \infty} \varphi_{k\#}\hat{T}_D^\mathfrak{o}(g,\tau) = \lim_{k \to \infty}\deg(\varphi_k) \cdot \br{\mathbb{S}^n}(g,\tau) = \deg(\varphi) \cdot \br{\mathbb{S}^n}(g,\tau)$$
            for all $(g,\tau) \in \mathcal{D}^k(\mathbb{S}^n)$. This completes the proof of the second part of the statement. 
    
            Finally, observe that 
             $$\spt (\iota_{i\#}\partial T_D^\mathfrak{o}) \subseteq \overline{ \iota_i(\spt(T_D^\mathfrak{o})) } = \iota_i(\spt(T_D^\mathfrak{o}))$$
             for $i=1,2$. Since $$0 = \partial\hat{T}_D^\mathfrak{o} = \iota_{1\#}\partial T_D^\mathfrak{o} - \iota_{2\#}\partial T_D^\mathfrak{o},$$ then 
             $\spt (\iota_{1\#}\partial T_D^\mathfrak{o}) = \spt (\iota_{2\#}\partial T_D^\mathfrak{o}) \subseteq \iota_1(\spt(T_D^\mathfrak{o})) \cap \iota_2(\spt(T_D^\mathfrak{o})) = \iota_1(\partial D) = \iota_2(\partial D)$ and so $\spt (\partial T_D^\mathfrak{o}) \subseteq \partial D$.
        \end{proof}
    
        As a consequence, we obtain the uniqueness of the local orientations in the following sense; cf. \cite[Proposition 4.4]{marti2024characterization}.
    
        \begin{corollary}\label{cor: uniqueness-T_U}
            Let $S \in \mathcal{I}_n(D)$ be such that $\spt(\partial S) \subseteq \partial D$. Then there exists $k \in \mathbb{Z}$ such that $ S = k \cdot T_D^\mathfrak{o}$.
        \end{corollary}
    
        \begin{proof}
            Define $\hat{S} = \iota_{1\#}S - \iota_{2\#}S \in \mathcal{I}_n(\hat{D})$. Since $\iota_1(x) = \iota_2(x)$ for all $x \in \partial D$ and $\spt(\partial S) \subseteq \partial D$, we have $\partial \hat{S} = 0$. In particular, $S\in \bI_n(\hat{D})$. By \cite[Proposition 4.4]{marti2024characterization}, there exists $k \in \mathbb{Z}$ such that $\hat{S} = k \cdot \hat{T}_D^\mathfrak{o}$. Therefore, 
            $$ S =  \iota^{-1}_{1\#}\hat{S}\rstr \iota_1(D)) = \iota^{-1}_{1\#}(\hat{T}_D^\mathfrak{o}\rstr \iota_1(D)) = k\cdot T_D^\mathfrak{o}.$$
            This completes the proof.
        \end{proof}

    \subsection{Orientability of metric manifolds}

    We are ready to prove the following more technical version of Theorem \ref{thme: characterization-orientability-metric-manifold}.
    \begin{theorem}\label{theo: characterization-orientability-metric-manifold-with-details}
        Let $(X,\dist)$ be a degree-regular metric $n$-manifold and let $E\subseteq X$ be the rectifiable part of $X$.
        \begin{enumerate}
            \item[(a)] If $X$ is oriented then there exists $T\in \bI_{n,\textup{loc}}(X)$ satisfying $\partial T = 0$,
            $\norm{T} = \lambda_E\haus^n \rstr E$ and $T|_ D =  T_D^\mathfrak{o}$ whenever $D \subseteq X$ is a regular disk and $\mathfrak{o}$ is the orientation on $\textup{int}(D)$ induced by the orientation on $X$.
            \item[(b)] If there exists $T\in \bI_{n,\textup{loc}}(X)$ satisfying $\partial T = 0$ and
            $\norm{T} = \lambda_E\haus^n \rstr E$, then there exists an orientation on $X$ such that $T|_ D = T_D^\mathfrak{o}$ whenever $D \subseteq X$ is a regular disk and $\mathfrak{o}$ is the orientation induced on $\textup{int}(D)$.
        \end{enumerate}
    \end{theorem}

    \begin{proof}
        We first assume that $X$ is oriented and we construct the current $T \in \bI_{n,\textup{loc}}(X)$ as in the statement. Let $\{D_i\}_i$ be a family of regular disks of $X$ such that their interiors $\{\textup{int}(D_i)\}_i$ form an open, locally finite cover of $X$. Such a family exists by Lemma \ref{lemma: good-subsets-existence} and the paracompactness of $X$. For $i \in \setN$, let $\mathfrak{o}_i$ be the orientation of $\textup{int}(D_i)$ induced by the orientation of $X$. Recall that each $D_i$ is degree-regular as $X$ is degree-regular. Let $T_{D_i}^{\mathfrak{o}_i}\in \mathcal{I}_n(D_i)$ be the current defined in \eqref{eq: definition-T_U}. We set $T_i := \iota_{D_i\#}T_{D_i}^{\mathfrak{o}_i}$, where $\iota_{D_i}$ denotes the inclusion of $D_i$ into $X$. The $T_i$'s are well-defined as each $D_i$ is compact. Let $\{\varphi_i\}_i$ be a Lipschitz partition of unity associated to the cover $\{\textup{int}(D_i)\}_i$. Define
        \begin{equation}\label{eq: construction-global-current-partition-of-unity}
            T(f,\pi) := \sum_{i \in \setN} T_i \rstr \varphi_i(f,\pi) = \sum_{i \in \setN} T_i (f  \varphi_i,\pi)
        \end{equation}
        for all $(f,\pi ) \in \mathcal{D}^n(X)$. For every $(f,\pi ) \in \mathcal{D}^n(X)$, only finitely many $\textup{int}(D_i)$ intersect $\spt(f)$ because $f$ has compact support and $\{\textup{int}(D_i)\}_i$ is a locally finite cover. Hence, the sum in \eqref{eq: construction-global-current-partition-of-unity} is finite and defines a metric $n$-current in $X$. Moreover, Proposition \ref{prop: characterization-T_U}, \eqref{eq:borel-restriction_vs_pushforward} and Proposition \ref{prop:restriction_of_current_to_open_subsets}.(iii) imply that for every $i,j\in \setN$ and every Borel set $B \subseteq D_i \cap D_j$ we have $T_i \rstr B = T_j \rstr B$. Therefore, if $(f,\pi ) \in \mathcal{D}^n(X)$ is such that $\spt(f) \subseteq D_j$ for some $j$, then
        $$T(f,\pi) = \sum_{i\in \setN} T_i(f\varphi_i,\pi) = \sum_{i\in \setN} T_j(f\varphi_i,\pi) =   T_j\left(f\sum_{i\in \setN}\varphi_i,\pi\right) = T_j(f,\pi).$$
        Thus, by the definition of the mass we have that $\bM_V(T) = \bM_V(T_j)$ whenever $V$ is an open subset of some $U_j$. In particular, $\norm{T}(B) = \norm{T_j}(B)$ for every Borel set $B$ contained in some $U_j$. It follows that $T \in \mathcal{I}_{n,\loc}(X)$. Furthermore, $\norm{T_i} = \lambda_{E_{D_i}} \haus^n\rstr E_{D_i}$ for every $i \in \setN$ by Proposition \ref{prop: characterization-T_U}. Hence, the discussion above implies that $\norm{T}= \lambda_E \haus^n\rstr E$, where $E = \bigcup_{i\in \setN}E_{D_i}$. Since, $X \setminus E$ is purely $n$-unrectifiable, Lemma \ref{lemma: rect-part} implies that $E$ is the rectifiable part of $X$, up to a set of $\haus^n$-measure zero. The next step is to show that $\partial T=0$. Let $(f,\pi)\in \mathcal{D}^{n-1}(X)$, then \eqref{eq:restr-boundary} yields
        \begin{equation}\label{eq: glued-T-is-T_U-on-U}
            \partial T(f,\pi) = \sum_{i\in \setN} \partial (T_i \rstr \varphi_i)(f,\pi) = \sum_{i \in \setN} (\partial T_i)\rstr \varphi_i (f,\pi) - T_i\rstr(1,\varphi_i) (f,\pi).
        \end{equation}
        By Proposition \ref{prop: existence-metric-fund-class-double}, we have $\spt(\partial T_i)\subseteq \partial D_i$, hence $\spt(f\varphi_i) \subseteq \textup{int}(D_i)$ and
        $$(\partial T_i)\rstr \varphi_i (f,\pi) = \partial T_i(f\varphi_i,\pi) = 0$$
        for every $i\in \setN$. Moreover, by  Definition \ref{def:current}.(2),
        $$\sum_{i \in \setN} T_i\rstr(1,\varphi_i) (f,\pi) =\sum_{i \in \setN}  T(f, \varphi_i ,\pi) = T\left(f,\sum_{i \in \setN} \varphi_i,\pi\right) = T(f,1,\pi) =0.$$
        Hence, $\partial T = 0$. Finally, let $D \subseteq X$ be a regular disk equipped with the orientation $\mathfrak{o}$ induced by the orientation of $X$. The current $T|_D \in \mathcal{I}_n(D)$ satisfies $\spt(\partial(T|_D)) \subseteq \partial D$ because of \eqref{eq:support_boundary_restriction}, since $\partial T = 0$. Corollary \ref{cor: uniqueness-T_U} gives $k\in \mathbb{Z}$ such that $T|_D = k\cdot T_D^\mathfrak{o}$. Let $i\in \setN$ be such that $\haus^n(D\cap D_i) > 0$. Since $T|_{D_i} = T_{D_i}^{\mathfrak{o}_i}$ and since the orientations of $D_i$ and $D$ agree, we conclude that $k=1$ by Proposition \ref{prop: characterization-T_U}.
        
        \medskip
        
        We now prove the other implication. Assume that there exists $T\in \bI_{n,\loc}(X)$ satisfying $\partial T = 0$ and $\norm{T}= \lambda_E \haus^n\rstr E$, where $E$ is the rectifiable part of $X$. As before, let $\{D_i\}_i$ be a family of regular disks such that $\{\textup{int}(D_i)\}_i$ form an open, locally finite cover of $X$. By Lemma \ref{lemma: good-subsets-existence}, we can further require that $\norm{T}(D_i) < \infty$ for every $i \in \setN$. 
        Equip each $\textup{int}(D_i)$ with an orientation $\mathfrak{o}_i$ and let $T_{D_i}^{\mathfrak{o}_i}\in \mathcal{I}_{n}(D_i)$ be the corresponding current as defined in \eqref{eq: definition-T_U}. By \eqref{eq:support_boundary_restriction}, since $\partial T = 0 $ we have $\spt(\partial (T|_{D_i})) \subseteq \partial D_i$. Corollary \ref{cor: uniqueness-T_U} implies that there exists $k_i \in \mathbb{Z}$ such that $T|_{D_i} = k_i \cdot T_{D_i}^{\mathfrak{o}_i}$. Let $E_i := E \cap D_i$, which is the rectifiable part of $D_i$ up to a set of $\haus^n$-measure zero. By Proposition \ref{prop: characterization-T_U},
        $$\norm{T_{D_i}^{\mathfrak{o}_i}} = \lambda_{E_i} \haus^n \rstr E_i = \lambda_{E \cap D_i} \haus^n \rstr E \cap D_i = \norm{T\rstr D_i} \rstr D_i = \norm{T|_{D_i}},$$
        where the last equality is Proposition \ref{prop:restriction_of_current_to_open_subsets}.(i). In particular, $k_i=\pm1$. By Proposition \ref{prop: characterization-T_U}, up to changing $\mathfrak{o}_i$ with $-\mathfrak{o}_i$, we can suppose that $k_i = 1$, so that $T_{D_i}^{\mathfrak{o}_i} = T|_{D_i}$. We claim that this induces an orientation on $X$. Indeed, let $D_i$ and $D_j$ be such that the intersection $U = \textup{int}(D_i) \cap \textup{int}(D_j)$ is non-empty. By the choice of the orientations on $\textup{int}(D_i)$ and $\textup{int}(D_j)$ we have that 
        $$T_{D_i}^{\mathfrak{o}_i}|_U =  T|_U = T_{D_j}^{\mathfrak{o}_j}|_ U.$$
        Proposition \ref{prop: characterization-T_U} implies that the  orientations on $U$ induced from $\mathfrak{o}_i$ and $\mathfrak{o}_j$ are equal. This shows that $X$ is orientable and that this orientation we found is such that $T|_{D_i} = T_{D_i}^{\mathfrak{o}_i}$ for every $i\in \setN$.
        The last property, namely that for every regular disk $D\subseteq X$ it holds that $T|_D = T_D^\mathfrak{o}$, where $\mathfrak{o}$ is the induced orientation on $\textup{int}(D)$, can be proved exactly as in the first part of the proof. This completes the proof.
    \end{proof}

    This leads to the following natural definition.
    \begin{definition}
        A degree-regular metric $n$-manifold
        $(X,\dist)$ is called \emph{orientable} if any of the equivalent conditions of Theorem \ref{thme: characterization-orientability-metric-manifold} holds. If it is the case, every $T\in \mathcal{I}_{n,\loc}(X)$ such that $\norm{T} = \lambda_E\haus^n\rstr E$ and $\partial T = 0$ is called an \emph{orientation} of $X$.
    \end{definition}

    We can upgrade Corollary \ref{cor: uniqueness-T_U} to a global version.

    \begin{corollary}\label{cor: any-bdry-less-current-implies-orient}
        Let $(X,\dist)$ be a connected, degree-regular metric $n$-manifold. Suppose there exists $S \in \bI_{n,\loc}(X)$, $S\neq 0$, satisfying $\partial S = 0$. Then there exists an orientation $T\in \mathcal{I}_{n,\loc}(X)$ and $k \in \mathbb{Z}$ such that $k \cdot T = S$.
    \end{corollary}

    \begin{proof}
        Let $\{D_i\}_i$ be a family of regular disks such that $\{\textup{int}(D_i)\}_i$ form an open, locally finite cover of $X$ and $\norm{S}(D_i) < \infty$ for every $i \in \setN$. Such a family exists by Lemma \ref{lemma: good-subsets-existence}. Since $\partial S = 0$, we have $\partial (S|_{D_i}) \subseteq \partial D_i$ by \eqref{eq:support_boundary_restriction}. 
        For every $i \in \setN$, fix an orientation $\mathfrak{o}_i$ on $\textup{int}(D_i)$. By Corollary \ref{cor: uniqueness-T_U}, there exists $k_i \in \mathbb{Z}$ such that $S|_{D_i} = k_i \cdot T_{D_i}^{\mathfrak{o}_i}$. Proposition \ref{prop: characterization-T_U} implies that $T_{D_i}^{\mathfrak{o}_i}|_B = \pm T_{D_j}^{\mathfrak{o}_j}|_B$ for every $i,j\in \setN$ and every Borel set $B \subseteq D_i \cap D_j$. By connectedness, $k\coloneqq |k_i| = |k_j|$ for every $i,j \in \setN$. Since $S\neq 0$, we deduce that $k\neq 0$.        
        Hence, $T = \frac{1}{k} S$ is a local integral current. Clearly, $\partial T =0$. Moreover, by \eqref{eq: mass-measure-T_U} and Proposition \ref{prop:restriction_of_current_to_open_subsets}.(i),
        $$\norm{T\rstr D_i} = \norm{T|_{D_i}} = \lambda_{E_{D_i}} \haus^n \rstr E_{D_i},$$
        where $E_{D_i}$ denotes the rectifiable part of $D_i$. Arguing as in the proof of Theorem \ref{theo: characterization-orientability-metric-manifold-with-details} we conclude that 
        $$\norm{T} = \lambda_E \haus^n \rstr E$$
        where $E$ is the rectifiable part of $X$. Hence, $T$ is an orientation of $X$.
    \end{proof}

    We have the following consequence.
    \begin{corollary}\label{cor: uniq-orientation-by-current-good-mfld}
        Let $(X,\dist)$ be a degree-regular metric $n$-manifold. If $X$ is connected then it admits at most two orientations.
        In particular, $H_n^\textup{IC}(X) = \mathbb{Z}^{|\mathcal{O}|}$, where $\mathcal{O}$ is the set of orientable connected components of $X$. 
    \end{corollary}
    \begin{proof}
        If $X$ is non-orientable then there are no orientations by Theorem \ref{thme: characterization-orientability-metric-manifold}. If $X$ is orientable then there exists an orientation $T\in \mathcal{I}_{n,\loc}(X)$. Clearly, also $-T$ is an orientation. Let now $S$ be any other orientation. If $X$ is connected, then there exists $k\in \mathbb{Z}$ such that $S = k\cdot T$, by Corollary \ref{cor: any-bdry-less-current-implies-orient}. Since $\norm{S} = \norm{T}$, we deduce that $S = \pm T$. This concludes the proof of the first statement.

        In the general case, let us write $X = \bigcup_{i\in \mathcal{C}} X_i$, where each $X_i$ is a connected component of $X$. Each $X_i$ is an open submanifold of $X$, hence it is a degree-regular metric $n$-manifold by Remark \ref{rmk: metric-mfld}.(4). Let us denote by $\mathcal{O}\subseteq \mathcal{C}$ the subset of connected components that are orientable. For every orientable component $X_i$ we fix an orientation $T_i$. Let $S\in H_n^\textup{IC}(X)$, i.e. $S \in \mathcal{I}_{n,\loc}(X)$ with $\partial S = 0$. Then, $S_i:=S|_{X_i} \in \mathcal{I}_{n,\loc}(X_i)$ satisfies $\partial S_i = 0$ for every $i \in \mathcal{C}$, by Lemma \ref{cor:cycles_open_sets} since $X_i$ is open in $X$. For every $i \in \mathcal{C}$ either $S_i = 0$ or $X_i$ is orientable and $S_i = k_i\cdot T_i$ for some $k_i\in \mathbb{Z}$, by Corollary \ref{cor: any-bdry-less-current-implies-orient}. We define the homomorphism of groups $\Phi \colon H_n^\textup{IC}(X) \to \mathbb{Z}^{|\mathcal{O}|}$, $S\mapsto (k_i)_{i\in \mathbb{Z}}$. It is injective: if $\Phi(S) = \Phi(S')$ then $S_i = S_i'$ for every $i\in \mathcal{O}$ and $S_i = S_i' = 0$ for every $i\in \mathcal{C} \setminus \mathcal{O}$, i.e. $S = S'$. It is also surjective since the current $\sum_{i\in \mathcal{O}} k_i\cdot T_i$ is sent to $(k_i) \in \mathbb{Z}^{|\mathcal{O}|}$.
    \end{proof}    

We conclude this section with the following consequence.
\begin{corollary}
\label{cor:orientability_stable_for_submanifolds}
    Let $(X,\dist)$ be a degree-regular metric $n$-manifold and let $Y\subseteq X$ be a submanifold with $\haus^{n-1}(X\setminus Y) = 0$. Then $X$ is orientable if and only if $Y$ is.
\end{corollary}
\begin{proof}
    If $X$ is orientable, then there exists an orientation $T\in \mathcal{I}_{n,\loc}(X)$ by Theorem \ref{theo: characterization-orientability-metric-manifold-with-details}. Then $T|_Y \in \mathcal{I}_{n,\loc}(Y)$ is an orientation of $Y$ since $\partial(T|_Y) = 0$ by Corollary \ref{cor:cycles_open_sets}. Therefore $Y$ is orientable by Theorem \ref{theo: characterization-orientability-metric-manifold-with-details}. Vice versa, assume that $Y$ is orientable and $T \in \mathcal{I}_{n,\loc}(Y)$ is an orientation. Notice that $Y$ is open in $X$. Moreover, $\norm{T}(K\cap Y) = \lambda_E\haus^n(K\cap E) \le \lambda_E\haus^n(K) < \infty$ for every compact $K\subseteq X$, since $\haus^n$ is locally finite on $X$. Here, $E$ is the rectifiable part of $Y$. By Lemma \ref{lemma:extension-of-currents}, the current $T^X \in \mathcal{I}_{n,\loc}(X)$ satisfies $\spt(\partial (T^X)) \subseteq X\setminus Y$. Since $\haus^{n-1}(X\setminus Y) = 0$, we have that $\partial (T^X) = 0$, by Theorem \ref{theo:boundary_zero_mass}. Applying Corollary \ref{cor:cycles_open_sets} and Corollary \ref{cor: any-bdry-less-current-implies-orient} to each connected component of $X$, which is open in $X$, we conclude that $X$ is orientable. 
\end{proof}

The previous result does not hold if $\haus^{n-1}(X\setminus Y) >0$. For instance, consider $X$ to be the M\"obius strip and $Y$ to be $X$ minus a vertical segment. Then $X$ is non-orientable while $Y$ is orientable. One can recover the result by assuming, in addition, that $(Y,\dist)$ is a length space (actually, the following proof shows that it is enough that $(Y,\dist)$ is locally quasiconvex).

\begin{proposition}
\label{prop:orientability_from_submanifold_to_manifold_length}
    Let $X$ be a topological $n$-manifold equipped with a metric $\dist$ and let $Y\subseteq X$ be an open, dense submanifold such that $(Y,\dist)$ is a length space. Then $X$ is orientable if and only if $Y$ is.
\end{proposition}

\begin{proof}
    If $X$ is orientable then $Y$ is, being a submanifold. Let $Z:= X\setminus Y$. Since $Y$ is a length space, it is connected. Hence, also $X$ is connected. We claim that $Z$ does not locally separate $X$ in the following sense: the set $U\setminus Z$ is connected for every open connected set $U \subseteq X$. If this is not true, then we can find an open and connected $U\subseteq X$ such that $U\setminus Z$ is disconnected. Let $V,W \subseteq U\setminus Z$ be disjoint, non-empty open sets such that $V\cup W = U\setminus Z$. By density, $(\overline{V} \cap U) \cup (\overline{W} \cap U) = U$. By connectedness of $U$ we find $x\in \overline{V} \cap \overline{W} \cap U$. Let $\{v_k\}\subseteq V$ and $\{w_k\} \subseteq W$ be sequences converging to $x$. Every path joining $v_k$ to $w_k$ in $Y$ must leave $U$. Since $(Y,\dist)$ is a length space we have that
    $$\dist(v_k,w_k) \ge \dist(v_k,\partial U) + \dist(w_k,\partial U) \to 2\dist(x,\partial U) > 0,$$
    while $\dist(v_k,w_k) \to 0$, as $k\to \infty$. This contradiction shows that $Z$ does not locally separate $X$. Now, let $p\colon \tilde{X} \to X$ be the orientation double cover of $X$. To prove that $X$ is orientable it is enough to show that $p$ is trivial, i.e. that there exists a continuous section $s\colon X \to \tilde{X}$. The restriction $p|_{Y}$ is the orientation double cover of $Y$. Since $Y$ is orientable, there exists a global section $s\colon Y \to p^{-1}(Y)$. Let $U$ be an open, connected, orientable set. Since $U\setminus Z$ is connected and orientable, $p^{-1}(U) \cong U\times \{\pm 1\}$. The function $s|_U$ is of the form $(\textup{id}, \sigma)$ with $\sigma \colon U \to \{\pm 1\}$. Since $\sigma$ is continuous, it is constant on the connected set $U \setminus Z$, therefore it admits a unique continuous extension on $U$. This gives a continuous extension of $s$ on $U$. Patching together the extensions of $s$ on each open set of a covering of $X$ by open, connected, orientable sets, we find a global section $s\colon X \to \tilde{X}$, which gives the conclusion.
\end{proof}

\section{Orientability of $n$-regular metric spaces}
\label{sec:n-regular}

Let $(X,\dist)$ be a locally compact, separable metric space. We denote by $\mathcal{M}(X)$ the set of topological manifold points of $X$, i.e. the points having an open neighbourhood homeomorphic to $\setR^n$, which is an open subset and a topological $n$-manifold since it is separable. We also define the set of singular points $\mathcal{S}(X) := X \setminus \mathcal{M}(X)$.

\begin{definition}
    Let $n\in \setN$. A locally compact metric space $(X,\dist)$ is $n$-regular if $\haus^n$ is locally finite on $X$, $(\mathcal{M}(X),\dist)$ is a degree-regular metric $n$-manifold, $\mathcal{M}(X)$ is dense in $X$ and $\haus^n(\mathcal{S}(X)) = 0$.
\end{definition}

Immediately from the definition we have the following corollary.
    
    \begin{corollary}
    \label{cor:extension_of_T_X}
        Let $(X,\dist)$ be a $n$-regular space and let $T \in \mathcal{I}_{n,\loc}(\mathcal{M}(X))$ be an orientation. Then $T^X \in \mathcal{I}_{n,\loc}(X)$ is the unique current $\tilde{T} \in \mathcal{I}_{n,\loc}(X)$ such that $\tilde{T}|_{\mathcal{M}(X)} = T$. Moreover, $\spt(\partial (T^X)) \subseteq \mathcal{S}(X)$ and $\Vert T^X \Vert = \lambda_E\haus^n\rstr E$, where $E$ is the rectifiable part of $X$.
    \end{corollary}
    \begin{proof}
        By definition, every compact subset of $X$ has finite $\haus^n$-measure. Moreover, the rectifiable part of $X$ coincides with the rectifiable part of $\mathcal{M}(X)$, up to $\haus^n$-negligible sets. Let $K\subseteq X$ be a compact subset. Then
        $$\norm{T}(K\cap \mathcal{M}(X)) = \lambda_E\haus^n(K\cap E) \le \lambda_E \haus^n(K) < \infty.$$
        Moreover, $\mathcal{M}(X)$ is locally compact. Hence, the thesis follows by Lemma \ref{lemma:extension-of-currents}, Proposition \ref{prop:extension_restriction_integral_currents} and Lemma \ref{lemma:uniqueness_extension}.
    \end{proof}

    We extend the definition of orientation to $n$-regular metric spaces.
    \begin{definition}
    \label{defin:n-regular_orientation}
        A $n$-regular metric space $(X,\dist)$ is \emph{orientable} if $\mathcal{M}(X)$ is orientable.
    \end{definition}

    \begin{remark}
        \label{rmk:n-regular_orientation}
        By Theorem \ref{theo: characterization-orientability-metric-manifold-with-details} and Corollary \ref{cor:extension_of_T_X}, a $n$-regular space is orientable if and only if there exists a current $T\in \mathcal{I}_{n,\loc}(X)$ such that $\norm{T} = \lambda_E\haus^n\rstr E$ and $\spt(\partial T) \subseteq \mathcal{S}(X)$, where $E$ is the rectifiable part of $X$. Every such $T$ is called an \emph{orientation} of $X$. If $T$ is an orientation of $X$, then $T|_{\mathcal{M}(X)}$ is an orientation of $\mathcal{M}(X)$. Indeed, $\partial(T|_{\mathcal{M}(X)}) = 0$ by \eqref{eq:support_boundary_restriction}, since $\mathcal{M}(X)$ is open.
    \end{remark}

    \begin{remark}
    \label{rmk:n-regular_open_sets}
        Let $(X,\dist)$ be a $n$-regular space and let $U\subseteq X$ be open. Then $(U,\dist)$ is locally compact, since it is open, and $n$-regular. The last property follows once we observe that $\mathcal{M}(U) = \mathcal{M}(X) \cap U$, using that open submanifolds of degree-regular metric $n$-manifolds are degree-regular, see Remark \ref{rmk: metric-mfld}.(4). Moreover, if $X$ is orientable then $U$ is orientable. This is because every open submanifold of an orientable manifold is orientable.
    \end{remark}

    The following is a criterion for orientability.

    \begin{corollary}
\label{cor:n-regular_full_support_current_implies_orientability}
    Let $(X,\dist)$ be a $n$-regular space. Then $X$ is orientable if and only if there exists $S\in \mathcal{I}_{n,\loc}(X)$ with $\spt(\partial S) \subseteq \mathcal{S}(X)$ and $\spt(S) = X$.
\end{corollary}
\begin{proof}
    If $X$ is orientable it is enough to take $S=T$, where $T$ is any orientation of $X$.

    Vice versa, assume to have $S$ as in the statement. Let $C$ be a connected component of $\mathcal{M}(X)$, which is open in $\mathcal{M}(X)$, hence in $X$. The current $S|_C \in \mathcal{I}_{n,\loc}(C)$ provided by Proposition \ref{prop:restriction_of_current_to_open_subsets} is not zero, since $\spt(S) = X$ and $C$ is open. Moreover, $\partial(S|_C) = 0$ by \eqref{eq:support_boundary_restriction}. Corollary \ref{cor: any-bdry-less-current-implies-orient} implies that $C$ is orientable. Since this is true for every connected component, we conclude that $\mathcal{M}(X)$ is orientable and hence $X$ is orientable.
\end{proof}

    The boundary  of an orientation $T$ may have no local finite mass. 
    \begin{example}[Orientation of a $n$-regular space whose boundary has locally infinite mass]
    \label{ex:extension_with_infinite_mass_boundary} 
    We construct a compact $2$-dimensional space $X$. For each integer $k \ge 3$, let $\ell_k := 1/k^2$ and $w_k := 1/k^2$. Let $\mathbb{S}^1_{\ell_k}$ be the circle of circumference $\ell_k$. 

For a fixed $k$, we take $k$ disjoint flat cylinders 
$C_{k,j} := \mathbb{S}^1_{\ell_k} \times [0, w_k]$, for $j = 1, \dots, k$. 
We form a quotient space $X_k$ by identifying all their bottom boundaries together, and all their top boundaries together. Explicitly, we define the equivalence relation $\sim_k$ on $\bigsqcup_{j=1}^k C_{k,j}$ by:
\begin{enumerate}
    \item $(t, 0)_{k,j} \sim_k (t, 0)_{k,1}$ for all $j \in \{1,\dots,k\}$ and $t \in \mathbb{S}^1_{\ell_k}$, forming a singular circle $A_k$.
    \item $(t, w_k)_{k,j} \sim_k (t, w_k)_{k,1}$ for all $j \in \{1,\dots,k\}$ and $t \in \mathbb{S}^1_{\ell_k}$, forming a singular circle $B_k$.
\end{enumerate}
The space $X_k$ is a "book" of $k$ cylindrical pages. Its metric diameter is bounded by $w_k + \ell_k/2 = \frac{3}{2k^2}$. 

To form a single connected space, we chain the spaces $X_k$ together. Select basepoints $p_k := (0, w_k) \in B_k$ and $q_k := (0,0) \in A_k$. We define the connected space $Y := \left(\bigsqcup_{k=1}^\infty X_k\right) / \approx$, where we identify $p_k \approx q_{k+1}$ for all $k \ge 1$. Let $Y$ be equipped with the induced length metric. Because $\sum_{k=1}^\infty \mathrm{diam}(X_k) \le \sum_{k=1}^\infty \frac{3}{2k^2} < \infty$, the sequence of gluing points is Cauchy. We define $X$ as the metric completion of $Y$, which adds exactly one limit point $x_\infty$. Notice that $X$ is compact.

The manifold part $\mathcal{M}(X) = X \setminus \mathcal{S}(X)$ is the disjoint union of the open cylinders $\bigsqcup_{k \ge 1} \bigsqcup_{j=1}^k \mathbb{S}^1_{\ell_k} \times (0, w_k)$, which is a degree-regular $2$-manifold.
Its total $2$-dimensional measure is finite:
\[
    \haus^2(\mathcal{M}(X)) = \sum_{k=1}^\infty k (\ell_k w_k) = \sum_{k=1}^\infty \frac{k}{k^4} = \sum_{k=1}^\infty \frac{1}{k^3} < \infty.
\]
The singular set is $\mathcal{S}(X) = \{x_\infty\} \cup \bigcup_{k=1}^\infty (A_k \cup B_k)$. Its $1$-dimensional measure is also finite:
\[
    \haus^1(\mathcal{S}(X)) = \sum_{k=1}^\infty (\haus^1(A_k) + \haus^1(B_k)) = \sum_{k=1}^\infty \frac{2}{k^2} < \infty.
\]

Let $T_{k,j}$ be the current given by the standard orientation on the cylinder $C_{k,j}$. Then $T := \sum_{k=1}^\infty \sum_{j=1}^k T_{k,j}$ is an orientation of $\mathcal{M}(X)$. Let $T^X$ be its extension to $X$.

The boundary of each oriented cylinder $C_{k,j}$ is exactly the top circle minus the bottom circle. Consequently, the boundary of the orientation on $X_k$ is exactly $k\br{B_k} - k\br{A_k}$. The boundary circles of two consectuives $X_k$'s meet only at points, so their
one-dimensional masses do not cancel.

Let $S^{(N)}$ be the sum of the page
currents for $3\leq k\leq N$, viewed as currents on $X$.
Then
\[
 \mathbf M(T^X-S^{(N)})
 \le \sum_{k>N}k^{-3}\longrightarrow0,
\]
and, for every Lipschitz test pair $(f,\pi)$,
\[
 \left|\partial(T^X-S^{(N)})(f,\pi)\right|
 \leq \Lip(f)\Lip(\pi)\,
       \mathbf M(T^X-S^{(N)})
 \leq \Lip(f)\Lip(\pi)\sum_{k>N}k^{-3}
 \longrightarrow0.
\]
This proves that
\[
    \partial T^X = \sum_{k=1}^\infty \left( k\br{B_k} - k\br{A_k} \right).
\]
The total mass of $\partial T^X$ is $\sum_{k=1}^\infty (k \ell_k + k \ell_k) = \sum_{k=1}^\infty \frac{2}{k} = \infty$. 
Furthermore, since any neighborhood of the limit point $x_\infty$ contains all $X_k$ for $k$ sufficiently large, the mass of $\partial T^X$ is locally infinite at $x_\infty$.
    \end{example} 

    The following result establishes a sufficient condition to ensure local finiteness of the mass of the boundary of an orientation.
    
    \begin{corollary}
    \label{cor:boundary_finite_mass_rectifiable}
        Let $(X,\dist)$ be a $n$-regular space. Assume that $\Theta^n(\cdot, \haus^n)$ is locally bounded on $\mathcal{S}(X)$, that $\haus^{n-1}\rstr \mathcal{S}(X)$ is locally finite and that $\mathcal{S}(X)$ is $(n-1)$-rectifiable. Then, every orientation $T\in \mathcal{I}_{n,\loc}(X)$ satisfies $\partial T \in \bM_{n-1,\loc}(X)$, hence $T\in \bI_{n,\loc}(X)$.
    \end{corollary}
    \begin{proof}
        Given an orientation $T\in \mathcal{I}_{n,\loc}(X)$, by Lemma \ref{lemma:tangential_function} we have that $\Theta^n(\cdot, \norm{T}) \le 2^n/\omega_n\cdot \Theta^n(\cdot, \haus^n)$, hence it is locally bounded on $\mathcal{S}(X)$. Moreover, $\haus^{n-1}\rstr \mathcal{S}(X)$ is locally finite by assumption and $\Theta_{n-1}(y,\haus^{n-1}) = 1$ for $\haus^{n-1}$-a.e. $y\in \mathcal{S}(X)$ since $\mathcal{S}(X)$ is $(n-1)$-rectifiable, see Lemma \ref{lemma:density_points_rectifiable_space}. Then the result follows by Theorem \ref{theo:boundary_finite_mass_hard}.
    \end{proof}

    Example \ref{ex:extension_with_infinite_mass_boundary} shows that the local boundedness of $\Theta^n(\cdot, \haus^n)$ cannot be removed from the assumptions of Corollary \ref{cor:boundary_finite_mass_rectifiable} and of Theorem \ref{theo:boundary_finite_mass_hard}. Even when the boundary of an orientation has locally finite mass, it may be non trivial, as the fundamental class of a smooth manifold with boundary shows. The next example shows an orientable $n$-regular space with no boundary for which every orientation has non-trivial boundary.

\begin{example}    \label{ex:n-regular_three_hemispheres}
    Let $(X,\dist)$ be the compact, $\CAT(1)$, geodesically complete space obtained by gluing three hemispheres by their common boundary. Then $\mathcal{M}(X)$ is the union of three open hemispheres, which is orientable. Let $T\in \mathcal{I}_{n,\loc}(X)$ be an orientation. By Corollary \ref{cor: uniq-orientation-by-current-good-mfld}, the restriction of $T$ on each hemisphere coincides, up to sign, with the unique orientation on the hemisphere, and its boundary is, up to a sign, the standard orientation on the circle. Since $\partial T$ is the sum of three of such currents, $T$ has non-zero boundary. The same construction with four hemispheres produces a space where some orientations have a boundary and some others do not.
\end{example}

    \subsection{Homological properties of $n$-regular spaces}

    In this section, we derive homological properties of $n$-regular spaces. They will be used in Sections \ref{sec:RCD} and \ref{sec:GCBA}. First of all, we generalize Corollaries  \ref{cor: any-bdry-less-current-implies-orient} and \ref{cor: uniq-orientation-by-current-good-mfld} to $n$-regular spaces.

    \begin{proposition}
    \label{prop:unique_orientation_extension}
        Let $(X,\dist)$ be a $n$-regular space with $\mathcal{M}(X)$ connected. Then the following statements hold.
        \begin{itemize}
            \item[(i)] If there exists $S\in \mathcal{I}_{n,\loc}(X)$, $S\neq 0$, with $\spt(\partial S) \subseteq \mathcal{S}(X)$ then there exist an orientation $T\in \mathcal{I}_{n,\loc}(X)$ and $k\in \mathbb{Z}$ such that $S = k\cdot T$.
            \item[(ii)] If $\mathcal{M}(X)$ is orientable, then there are exactly two orientations $\pm T \in \mathcal{I}_{n,\loc}(X)$.
            If $\partial T = 0$ then $H_n^{\textup{IC}}(X) \cong \mathbb{Z}$. If $\partial T \neq 0$ then $H_n^{\textup{IC}}(X) \cong \{0\}$. 
            \item[(iii)] If $X$ is not orientable then $H_n^{\textup{IC}}(X) \cong \{0\}$.
        \end{itemize}
    \end{proposition}
    \begin{proof}
        Let $S \in \mathcal{I}_{n,\loc}(X)$, $S\neq 0$, with $\spt(\partial S) \subseteq \mathcal{S}(X)$. Since $\haus^n(\mathcal{S}(X)) = 0$, then $S|_{\mathcal{M}(X)} \neq 0$. Moreover, $\partial(S|_{\mathcal{M}(X)}) = 0$ because of \eqref{eq:support_boundary_restriction}. By Corollary \ref{cor: any-bdry-less-current-implies-orient}, there exist an orientation $T\in \mathcal{I}_{n,\loc}(\mathcal{M}(X))$ and $k\in \mathbb{Z}$ such that $S|_{\mathcal{M}(X)} = k\cdot T$. Since $\haus^n(\mathcal{S}(X)) = 0$, then $S = (S|_{\mathcal{M}(X)})^X$, by Lemma \ref{lemma:uniqueness_extension}.
        Therefore, $S = k\cdot T^X$, where $T^X$ is the orientation of $X$ provided by Corollary \ref{cor:extension_of_T_X}. 

        Item (i) implies directly, as in the proof of Corollary \ref{cor: uniq-orientation-by-current-good-mfld}, that there are exactly two orientations $\pm T \in \mathcal{I}_{n,\loc}(X)$. The same proof shows that $H_n^\textup{IC}(X) \cong \mathbb{Z}$ if $\partial T = 0$. If $\partial T \neq 0$, there are no $S\in \bI_{n,\loc}(X)$ with $\partial S = 0$, by (i). Hence, $H_n^\textup{IC}(X) \cong \{0\}$. Again by (i), if $X$ is not orientable then there are no non-trivial $S\in\bI_{n,\loc}(X)$ with $\partial S= 0$, so $H_n^\textup{IC}(X) \cong \{0\}$.
    \end{proof}

    As a consequence, we characterize those $n$-regular spaces with exactly two orientations. Such spaces are called \emph{uniquely orientable}.
    
    \begin{corollary}
    \label{cor:n-regular_unique_orientability}
        Let $(X,\dist)$ be a $n$-regular space. Then $X$ is uniquely orientable if and only if $\mathcal{M}(X)$ is orientable and connected. Moreover, if it is the case, then $H_n^{\textup{IC}}(X) \cong \{0\}$ if and only if $\partial T \neq 0$ for every orientation $T\in \mathcal{I}_{n,\loc}(X)$, and $H_n^{\textup{IC}}(X) = \langle T \rangle \cong \mathbb{Z}$ if and only if $\partial T = 0$ for every orientation. If, additionally, $\haus^{n-1}(\mathcal{S}(X)) = 0$, then $H_n^\textup{IC}(X) \cong \mathbb{Z}$.
    \end{corollary}
    \begin{proof}
        If $\mathcal{M}(X)$ is non-orientable, then there are no orientations at all, by definition. If $\mathcal{M}(X)$ is orientable and disconnected, then it admits more than two orientations by Corollary \ref{cor: uniq-orientation-by-current-good-mfld}. Each of them extends to different orientations of $X$ by Corollary \ref{cor:extension_of_T_X}. Hence $X$ is not uniquely orientable. This proves that if $X$ is uniquely orientable, then $\mathcal{M}(X)$ is orientable and connected. The vice versa is  Proposition \ref{prop:unique_orientation_extension}.(ii). 

        Finally, assume that $X$ is uniquely orientable and let us fix an orientation $T$. Proposition \ref{prop:unique_orientation_extension} shows that $H_n^{\textup{IC}}(X) = \langle T \rangle \cong \mathbb{Z}$ if and only if $\partial T = 0$. Otherwise, $H_n^{\textup{IC}}(X) \cong \{0\}$. If $\haus^{n-1}(\mathcal{S}(X)) = 0$, then $\partial T = 0$ by Theorem \ref{theo:boundary_zero_mass}, hence $H_n^\textup{IC}(X) \cong \mathbb{Z}$.
\end{proof}

Both $H^\textup{IC}_n(X) \cong \mathbb{Z}$ and $H^\textup{IC}_n(X) \cong \{0\}$ may happen for orientable $n$-regular spaces. For the former, it is enough to take any closed orientable Riemannian manifold. For the latter we have the following example.
\begin{example}[Orientable locally $\CAT(1)$ space with trivial top-dimensional homology]
\label{ex:n-regular_orientable_but_trivial_homology}
    Let $Y$ be the compact, geodesically complete, $\CAT(1)$ space obtained by gluing three hemispheres by their boundary, i.e. the space of Example \ref{ex:n-regular_three_hemispheres}. For every choice of pairwise distinct $i,j,k \in \{1,2,3\}$ we select in the hemisphere $H_i$ two disjoint disks $D_i^j$ and $D_i^k$  of radius $1/10$. We remove them from $H^i$ and we glue $\partial D_i^j$ to $\partial D_j^i$ and  $\partial D_i^k$ to $\partial D_k^i$. The resulting space is called $X$. It admits a metric $\dist$ with respect to which it is a $\GCBA$ space. 
    By construction, $\mathcal{M}(X)$ is connected and orientable. However, $H_2^{\textup{IC}}(X)$ $\cong \{0\}$. To prove this, by Corollary \ref{cor:n-regular_unique_orientability} it is enough to notice that an orientation of $X$ has non-trivial boundary. This can be shown by the same argument of Example \ref{ex:n-regular_three_hemispheres}.
\end{example}

There are examples of non-orientable $n$-regular spaces with non trivial top-dimensional homology group, even when $\haus^{n-1}(\mathcal{S}(X)) = 0$.

\begin{example}[A non-orientable $\CAT(1)$ space with non-trivial top-dimensional homology]
\label{ex:n-regular_non_orientable_but_Z_homology}
    Equip the projective plane $\setR\mathbb{P}^2$ with twice the standard metric coming from the sphere: it is a geodesically complete $\CAT(1)$ space $(Y,\dist_Y)$. Let $X$ be the space obtained by gluing a sphere $\mathbb{S}^2$ to $Y$ by a point $p$: it is a compact, geodesically complete, $\CAT(1)$ space. Observe that $\mathcal{M}(X) = X \setminus \{p\} = (\mathbb{RP}^2 \setminus \{p\}) \sqcup (\mathbb{S}^2 \setminus \{p\})$ is disconnected. The component $\mathbb{RP}^2 \setminus \{p\}$ is not orientable, hence $X$ is not orientable. On the other hand, $H_2^\textup{sing}(X) \cong H_2^\textup{IC}(X) \cong \mathbb{Z}$, generated by the fundamental class of $\mathbb{S}^2$. In this example, $\mathcal{S}(X)$ is finite.
\end{example}

In the previous example, the generator $T \in H_2^\textup{IC}(X) \cong \mathbb{Z}$ is the fundamental class of $\mathbb{S}^2$. In particular its support is a strict subset of $X$. This is the only obstruction to the orientability of the space, as it follows by Corollary \ref{cor:n-regular_full_support_current_implies_orientability}.
Similarly, unique orientability can be characterized in homological terms, assuming that $\haus^{n-1}(\mathcal{S}(X)) = 0$.

\begin{corollary}
    Let $(X,\dist)$ be a $n$-regular space with $\haus^{n-1}(\mathcal{S}(X)) = 0$. Then $X$ is uniquely orientable if and only if $H_n^\textup{IC}(X) = \langle T \rangle$ and $\spt(T) = X$. In the latter case, $T$ is the unique, up to sign, orientation of $X$.
\end{corollary}
\begin{proof}
    By Theorem \ref{theo:boundary_zero_mass}, every current $S \in \bI_{n,\loc}(X)$ with $\spt(\partial S) \subseteq \mathcal{S}(X)$ satisfies $\partial S = 0$.
    Therefore, if $X$ is uniquely orientable and $T$ is an orientation, then $\spt(T) = X$ by definition and $H_n^\textup{IC}(X) = \langle T \rangle \cong \mathbb{Z}$ by Corollary \ref{cor:n-regular_unique_orientability}. On the other hand, assume that $H_n^\textup{IC}(X) = \langle T \rangle$ and $\spt(T) = X$. Since $\partial T = 0$, $X$ is orientable by Corollary \ref{cor:n-regular_full_support_current_implies_orientability}. If $\mathcal{M}(X)$ has more than one connected component, then it has more than two orientations by Corollary \ref{cor: uniq-orientation-by-current-good-mfld}. Each of them extends to orientations of $X$ by Corollary \ref{cor:extension_of_T_X}. Moreover, these extensions have no boundary because of Theorem \ref{theo:boundary_zero_mass}, since $\haus^{n-1}(\mathcal{S}(X)) = 0$. Therefore, $H_n^\textup{IC}(X)$ would not be cyclic. Hence, $\mathcal{M}(X)$ is connected, so $X$ is uniquely orientable by Corollary \ref{cor:n-regular_unique_orientability}.
\end{proof}

\subsection{Locally uniquely orientable $n$-regular spaces}
\label{subsec:n-regular_loc_uniq_orientable}

Example \ref{ex:n-regular_orientable_but_trivial_homology} shows a connected, orientable, compact, $2$-regular space with trivial top-dimensional homology group. On the other hand, Example \ref{ex:n-regular_non_orientable_but_Z_homology} gives a connected, non-orientable, compact, $2$-regular space with top-dimensional homology group isomorphic to $\mathbb{Z}$.

For a connected topological $n$-manifold, the top-dimensional locally finite singular homology group with integer coefficients is isomorphic to $\mathbb Z$ if the manifold is orientable, and is zero otherwise. We introduce a natural class of $n$-regular spaces for which we can recover this dichotomy for the homology of local integral currents, assuming that $\haus^{n-1}(S(X)) = 0$.

\begin{definition}
\label{defin:LUO_spaces}
    A $n$-regular space $(X,\dist)$ is said to be \emph{locally uniquely orientable} (LUO) if for every $x\in X$ there exists $r_x>0$ such that $B(x,r)$ is uniquely orientable for every $0<r<r_x$.
\end{definition}

The definition is well-posed because of Remark \ref{rmk:n-regular_open_sets}.
We can prove the desired dichotomy for locally uniquely orientable spaces.

\begin{theorem}
\label{theo:n-regular_loc_uniq_orient_dichotomy}
    Let $(X,\dist)$ be a connected, \textup{LUO} $n$-regular space. Then $\mathcal{M}(X)$ is connected. Moreover, one of the following two situations occur.
    \begin{itemize}
        \item[(a)] $X$ is not orientable and $H_n^\textup{IC}(X) = \{0\}$.
        \item[(b)] $X$ is uniquely orientable.  In this case, if in addition $\haus^{n-1}(\mathcal{S}(X)) = 0$, then $H_n^\textup{IC}(X) \cong \mathbb{Z}$. 
    \end{itemize}
\end{theorem}
\begin{proof}
    Suppose that $\mathcal{M}(X)$ is not connected and let $U,V$ be two disjoint open subsets such that $\mathcal{M}(X) = U \cup V$. Since $\mathcal{M}(X)$ is dense in $X$, $X = \overline{U} \cup \overline{V}$. Since $X$ is connected, then $\overline{U} \cap \overline{V} \neq \emptyset$. Let $x \in \overline{U} \cap \overline{V}$ and let $r_x>0$ as in Definition \ref{defin:LUO_spaces}. Let $0<r<r_x$ and let $U_r := U\cap B(x,r)$ and $V_r := V \cap B(x,r)$. The sets $U_r$ and $V_r$ are open, non-empty and disjoint in $\mathcal{M}(B(x,r))$. Moreover, $\mathcal{M}(B(x,r)) = U_r \cup V_r$.  Corollary \ref{cor:n-regular_unique_orientability} implies that $B(x,r)$ is not uniquely orientable, contradicting Definition \ref{defin:LUO_spaces}. Hence, $\mathcal{M}(X)$ is connected.
    Now, if $X$ is oriented then it is uniquely oriented by Corollary \ref{cor:n-regular_unique_orientability}. The same result implies that $H_n^\textup{IC}(X) \cong \mathbb{Z}$ if $\haus^{n-1}(S(X)) = 0$. If $X$ is non-orientable then the thesis follows by Proposition \ref{prop:unique_orientation_extension} since $\mathcal{M}(X)$ is connected.
\end{proof}

A converse of Theorem \ref{theo:n-regular_loc_uniq_orient_dichotomy} does not hold. In particular, the connectedness of $\mathcal{M}(X)$ does not imply that $X$ is LUO.
\begin{example}[A uniquely orientable $\CAT(1)$ space which is not LUO]
    Let $X$ be the compact, geodesically complete, $\CAT(1)$ space obtained by identifying the north and south pole of a sphere $\mathbb{S}^2$ with radius $2$. Call $p$ the corresponding point in $X$. By construction, $\mathcal{M}(X) = X \setminus \{p\}$ is connected and orientable. Hence, $X$ is uniquely orientable by Corollary \ref{cor:n-regular_unique_orientability}. However, $X$ fails the local unique orientability property at $p$. 
\end{example}

From the degree-$n$ homological point of view, LUO $n$-regular spaces $X$ with $\haus^{n-1}(S(X)) = 0$ resemble topological manifolds.

\begin{proposition}
\label{prop:n-regular_loc_uniq_conn_homology_of_conn_open_subsets}
    Let $(X,\dist)$ be a connected, orientable, \textup{LUO}, $n$-regular space with $\haus^{n-1}(\mathcal{S}(X)) = 0$ and let $T\in \mathcal{I}_{n,\loc}(X)$ be an orientation. Let $U\subseteq X$ be an open, connected subset. Then $T|_U$ generates $H_n^\textup{IC}(U)$. Assume, in addition, that $\haus^n(\overline{U}\setminus U) = 0$. If $T|_{\overline{U}} \in \bI_{n,\loc}(\overline{U})$, then $T|_ {\overline{U}}$ generates $H_n^\textup{IC}(\overline{U}, \overline{U}\setminus U)$, otherwise $H_n^\textup{IC}(\overline{U}, \overline{U}\setminus U) = \{0\}$.
\end{proposition}
\begin{proof}
    By Remark \ref{rmk:n-regular_open_sets}, $(U,\dist)$ is a LUO $n$-regular space. Moreover, $\haus^{n-1}(\mathcal{S}(U)) \le \haus^{n-1}(\mathcal{S}(X)) = 0$. By Proposition \ref{prop:extension_restriction_integral_currents}, Proposition \ref{prop:restriction_of_current_to_open_subsets}.(i) and \eqref{eq:support_boundary_restriction}, the current $T|_U \in \mathcal{I}_{n,\loc}(U)$ is an orientation of $U$. Theorem \ref{theo:n-regular_loc_uniq_orient_dichotomy} implies that $H_n^\textup{IC}(U) = \langle T|_U \rangle$. This proves the first claim. 
    
    Let $S\in \bI_{n,\loc}(\overline{U})$ be a representative of a class in $H_n^\textup{IC}(\overline{U}, \overline{U}\setminus U)$. Then $\spt(\partial S) \subseteq \overline{U}\setminus U$, hence $S|_U \in H_n^\textup{IC}(U)$, by Proposition \ref{prop:extension_restriction_integral_currents} and \eqref{eq:support_boundary_restriction}. By the first part of the statement, there exists $k\in \mathbb{Z}$ such that $S|_U = k\cdot (T|_U)$. Since we are assuming that $\haus^n(\overline{U} \setminus U)=0$, we can apply Lemma \ref{lemma:uniqueness_extension} to deduce that $S = (S|_U)^{\overline{U}}$ and $k\cdot T|_{\overline{U}} = (k\cdot T_U)^{\overline{U}}$. Therefore, $S = k\cdot T|_{\overline{U}}$. This shows that if $H_n^\textup{IC}(\overline{U},\overline{U} \setminus U) \ne \{0\}$ then $T|_{\overline{U}} \in \bI_{n,\loc}(\overline{U})$ and that it generates $H_n^\textup{IC}(\overline{U},\overline{U}\setminus U)$.
    This concludes the proof.
\end{proof}

\begin{remark}
    The results of this section remain true under more general assumptions, that can be expressed only in homological terms. More precisely, we say that a locally compact metric space $(X,\dist)$ is \emph{homologically locally uniquely orientable} (HLUO) if $H_n^\textup{IC}(X,X\setminus \{x\}) = \langle T_x\rangle \cong \mathbb{Z}$ for every $x\in X$ and $\spt(T_x) \supseteq B(x,r)$ for some $r>0$. Under this assumption, the analogous of Theorem \ref{theo:n-regular_loc_uniq_orient_dichotomy} holds. Example \ref{ex:n-regular_non_orientable_but_Z_homology} shows that requiring $H_n^\textup{IC}(X,X\setminus \{x\}) \cong \mathbb{Z}$ for every $x\in X$, without a condition on the support of the generator, is not enough. The relation between LUO $n$-regular spaces and HLUO spaces is analogous to the relation between topological manifolds and homology manifolds.
\end{remark}

\subsection{Branched double covers of $n$-regular spaces}

If a connected, degree-regular metric $n$-manifold is orientable, then it admits an orientation in the current sense by Theorem \ref{thme: characterization-orientability-metric-manifold}. If it is not orientable, we can consider the double cover, which is an orientable, connected manifold. If the manifold is a length space, we can lift the metric to the double cover and the resulting space is an orientable, connected, degree-regular metric $n$-manifold, which therefore supports a current as in Theorem \ref{thme: characterization-orientability-metric-manifold}. The next theorem extends this picture to $n$-regular spaces. We recall that a map $f\colon (X,\dist_X) \to (Y,\dist_Y)$ is a submetry if $f(B(x,r)) = B(f(x),r)$ for every $x\in X$ and every $r\ge 0$. A submetry is $1$-Lipschitz, $1$-open and surjective. The quotient map associated to the action of a group of isometries is always a submetry. Moreover, we recall that, as defined in \cite{CavallucciIncertiMedici2026}, a continuous, surjective map $\pi\colon X\to Y$ between topological spaces is a branched covering if: 
\begin{itemize}
    \item[(i)] $\pi|_{\pi^{-1}(\Sigma)}\colon \pi^{-1}(\Sigma) \to \Sigma$ is a homeomorphism, where the set $\Sigma$ is defined as  $\Sigma := \{y\in Y\,:\, \#\pi^{-1}(y) = 1\}$;
    \item[(ii)] $\pi|_{X\setminus\pi^{-1}(\Sigma)}\colon X \setminus\pi^{-1}(\Sigma) \to Y\setminus\Sigma$ is a covering.
\end{itemize}
The set $\Sigma$ is called the singular set. The branched covering is said to have pure-degree $2$ if $\pi|_{X\setminus\pi^{-1}(\Sigma)}$ is a $2$-covering map. 

\begin{theorem}
\label{theo:double_cover}
    Let $(X,\dist)$ be a $n$-regular space such that $(\mathcal{M}(X),\dist)$ is a length space and $\mathcal{M}(X)$ is not orientable. Then there exist a $n$-regular space $(DX, \dist_D)$ and an involutive isometry $\Gamma \colon DX \to DX$ satisfying the following properties.
    \begin{itemize}
        \item[(i)] $(X,\dist)$ is isometric to $\langle \Gamma \rangle \backslash DX$ equipped with the quotient distance. Given $\textup{Fix}(\Gamma) := \{\hat{x}\in DX\,:\, \Gamma \hat{x}=\hat{x}\}$, then the projection map $\pi\colon DX \to X$
        is a branched covering of pure degree $2$ with singular set $\pi(\textup{Fix}(\Gamma))$. More precisely, $\pi|_{DX \setminus \textup{Fix}(\Gamma)}$ is a $2$-fold covering map and a local isometry, while $\pi|_{\textup{Fix}(\Gamma)}$ is an isometry.
        \item[(ii)] The map $\pi$ is a submetry, satisfies $\dist(\pi(\hat{x}), \pi(\hat{x}')) = \min \left\{ \dist_D(\hat{x}, \hat{x}'), \, \dist_D(\hat{x}, \Gamma(\hat{x}')) \right\}$ for every $\hat{x},\hat{x}' \in DX$ and $\pi_\# (\haus^n \rstr DX) = 2 \haus^n \rstr X$.
        \item[(iii)] The manifold part $\mathcal{M}(DX)$ of $DX$ is orientable and connected. Moreover, $\pi_\#(\haus^{n-1}\rstr \mathcal{S}(DX)) \le 2\haus^{n-1}\rstr \mathcal{S}(X)$.
        \item[(iv)] There exist exactly two orientations $\pm T\in \mathcal{I}_{n,\loc}(DX)$, and $\Gamma_\#T = -T$. Moreover, every $S\in \mathcal{I}_{n,\loc}(DX)$ with $\spt(\partial S) \subseteq \mathcal{S}(DX)$ satisfies $\Gamma_\#S = - S$.
    \end{itemize}
\end{theorem}

\begin{proof}
    We consider the orientation double cover $\pi\colon Y \to \mathcal{M}(X)$ of $\mathcal{M}(X)$. Since $(\mathcal{M}(X),\dist)$ is a length metric space, we can define a metric $\dist_D$ on $Y$ that makes $\pi$ a local isometry, see \cite[Proposition I.3.25]{BridsonHaefliger}. Let $\Gamma \colon Y \to Y$ be the unique non-trivial deck transformation. Observe that $\Gamma$ is an involutive isometry of $Y$. Let $\overline{Y}$ be the completion of $Y$. We extend $\Gamma$ to an involutive isometry of $\overline{Y}.$ This gives a quotient map $\pi \colon \overline{Y} \to \overline{X}$ that extends $\pi$. Finally, we define $DX = \pi^{-1}(X)$.     
    
    By definition, $\Gamma$ acts as an involutive isometry on $DX$, $\pi |_{DX}$ is the associated quotient map, that we denote again by $\pi$, and $X = \langle \Gamma \rangle \backslash DX$ as sets. Since $\pi|_{Y}\colon Y \to \mathcal{M}(X)$ is a local isometry and $(\mathcal{M}(X),\dist)$ is a length space, then the quotient metric on $\mathcal{M}(X)$ coincides with $\dist$, hence the quotient metric on $X$ coincides with $\dist$ too, since $\mathcal{M}(X)$ is dense in $X$. Therefore, $\dist(\pi(\hat{x}), \pi(\hat{x}')) = \min \left\{ \dist_D(\hat{x}, \hat{x}'), \, \dist_D(\hat{x}, \Gamma(\hat{x}')) \right\}$ for every $\hat{x},\hat{x}' \in DX$.
    This implies that $\pi|_{\textup{Fix}(\Gamma)}$ is an isometry, hence a homeomorphism, and that $\pi|_{DX \setminus \textup{Fix}(\Gamma)}$ is a local isometry and a $2$-fold covering map, since $\textup{Fix}(\Gamma)$ is closed. Hence, $\pi\colon DX \to X$ is a branched covering of pure degree $2$ with singular set $\pi(\textup{Fix}(\Gamma))$. This shows (i).

    The fact that $\pi$ is a submetry can be proved exactly as in \cite[Corollary 4.7]{CavallucciIncertiMedici2026}. Since $\pi|_{Y}\colon Y \to \mathcal{M}(X)$ is a local isometry and a $2$-fold covering, $\pi_\#(\haus^n \rstr Y) = 2\haus^n \rstr X$. Similarly, $\haus^n(DX\setminus Y) \le 2\haus^n(X\setminus \mathcal{M}(X)) = 0$, hence $\pi_\#(\haus^n \rstr DX) = 2\haus^n \rstr X$. This proves (ii).

    We have not shown yet that $DX$ is a locally compact, $n$-regular space. 
    The quotient map $\pi$ associated to the finite group $\langle \Gamma \rangle$ is closed and has compact fibers, hence it is proper. Since it is also continuous and $X$ is locally compact, then $DX$ is locally compact. The metric $n$-manifold $(Y,\dist_D)$ is degree-regular, being locally isometric to the degree-regular metric $n$-manifold $\mathcal{M}(X)$, see Remark \ref{rmk: metric-mfld}.(4). By (ii), we have that $\haus^n(\mathcal{M}(DX)\setminus Y) \le 2\haus^n(\mathcal{S}(X)) = 0$. Hence, $(\mathcal{M}(DX), \dist_D)$ is a degree-regular metric $n$-manifold by Lemma \ref{lemma:degree-regularity_from_submanifold_to_manifold}. This shows that $(DX,\dist_D)$ is $n$-regular.

   The orientable manifold $(Y,\dist_D)$ is a length space. It is also open and dense in $\mathcal{M}(DX)$. Therefore, $\mathcal{M}(DX)$ is orientable by Proposition \ref{prop:orientability_from_submanifold_to_manifold_length} and connected. Moreover, $\mathcal{S}(DX) = DX \setminus \mathcal{M}(DX) \subseteq DX \setminus Y = \pi^{-1}(\mathcal{S}(X))$. Arguing as before, we have that $\haus^{n-1}\rstr \pi^{-1}(\mathcal{S}(X)) \le 2\haus^{n-1}\rstr \mathcal{S}(X)$, which proves (iii). 
    
    Since $\mathcal{M}(DX)$ is connected and orientable, $DX$ has exactly two orientations $\pm T \in \mathcal{I}_{n,\loc}(DX)$ by Corollary \ref{cor:n-regular_unique_orientability}. Since $\Gamma$ is an isometry, $\Gamma_\# T$ is an orientation, hence $\Gamma_\# T = \pm T$. Restricting to $Y$, we deduce that $\Gamma_\#T = - T$.
    Let $S\in \mathcal{I}_{n,\loc}(DX)$ with $\spt(\partial S) \subseteq \mathcal{S}(DX)$. If $S=0$ then $\Gamma_\#S = -S$. If $S\neq 0$, Proposition \ref{prop:unique_orientation_extension} implies that there exists $k\in \mathbb{Z}$ such that $S = k \cdot T$. Therefore $\Gamma_\#S = -S$.
\end{proof}

In general, the involution $\Gamma$ may have fixed points and $\pi$ is not a $2$-fold covering, see for instance \cite[Example 1.20]{BrenaBruePigati2024}. Similarly, the inclusion $\mathcal{M}(DX) \supseteq Y$, where $Y$ is the double cover of $\mathcal{M}(X)$ may be strict.

\begin{definition}
    Let $(X,\dist)$ be a non-orientable $n$-regular space such that $(\mathcal{M}(X),\dist)$ is a length space. The space $DX$ of Theorem \ref{theo:double_cover} is called the \emph{branched double covering} of $X$.  
\end{definition}

The name is justified by the following proposition, which shows that the couple $(DX,\Gamma)$ is unique up to equivariant isometries in the following sense.

\begin{proposition}
\label{prop:double_cover_unique}
   Let $(X,\dist)$ be a non-orientable $n$-regular space such that $(\mathcal{M}(X),\dist)$ is a length space. Let $(Y,\dist_Y)$ be a metric space with an involutive isometry $\Delta \colon Y \to Y$ such that $(X,\dist)$ is isometric to $ \langle\Delta\rangle \backslash Y$ equipped with the quotient metric, the quotient map $p\colon Y \to X$ satisfies that $p|_{p^{-1}(\mathcal{M}(X))}\colon p^{-1}(\mathcal{M}(X)) \to \mathcal{M}(X)$ is the orientation double cover and $(p^{-1}(\mathcal{M}(X)),\dist_Y)$ is a length space. Then there exists an isometry $F\colon DX \to Y$ such that $F\circ \Gamma = \Delta \circ F$. 
\end{proposition}

\begin{proof}
    The map $p|_{p^{-1}(\mathcal{M}(X))}$ is the orientation double cover, therefore a local isometry. This implies that there exists an isometry $F\colon \pi^{-1}(\mathcal{M}(X)) \to p^{-1}(\mathcal{M}(X))$ satisfying $F\circ \Gamma = \Delta \circ F$, where $\pi\colon DX \to X$ is the quotient map. By density, this map extends to an isometric embedding $F\colon DX \to \overline{Y}$ such that $F\circ \Gamma = \Delta \circ F$, where $\overline{Y}$ is the completion of $Y$. Since $X = \langle\Delta\rangle \backslash Y$, the image of $F$ is $Y$, which gives the thesis. 
\end{proof}

\section{Orientability of $\RCD$ spaces}
\label{sec:RCD}

Throughout this section, $(X,\dist,\haus^n)$ denotes a
non-collapsed $\RCD(\kappa,n)$ space, where $\kappa\in\setR$
and $n\in\setN\setminus\{0\}$.
Non-collapsed $\RCD(\kappa,n)$ spaces were introduced in
\cite{DPG} as a subclass of $\RCD$ spaces
\cite{AGS:14b,G:15}, motivated by the Cheeger--Colding
theory of Ricci limit spaces \cite{CC:97,CC:II,CC:III}.
A central feature of the $\RCD$ theory is the equivalence
between the Lagrangian formulation via optimal transport
\cite{LV:AnnMath,Sturm:I,Sturm:II} and the Eulerian
formulation via the Bakry--\'Emery condition
\cite{BakryEmery}, established in
\cite{AGS:15,EKS:15,AMS:JGA,AMS:19}.
Another fundamental feature is the local-to-global
property \cite{CaMi:21}.

For an introduction to the theory, we refer to the surveys
\cite{Amb:ICM,Vil:BS,Sturm:survey,G:Survey}.
For first- and second-order calculus, see
\cite{AGS:14a,Gigli:nonsmoothDG} and the lecture notes
\cite{GigliPasqualetto-Lectures}.
We assume familiarity with Gromov--Hausdorff distance
and convergence; see \cite{Gromov:Book,BBI-Book}
for the foundations and \cite{GMS15} for their
relation to the $\RCD$ condition.

Following \cite{MondinoNaber19,DPG}, we denote by
\[
    \mathcal R_n(X)
    :=\{x\in X:\operatorname{Tan}_xX=\{(\setR^n,\dist_E,0)\}\}
\]
the set of $n$-regular points, where $\operatorname{Tan}_xX$ denotes the set of Gromov-Hausdorff tangents at $x$.
For $0\leq k\leq n-1$, we denote by
\[
    \mathcal{S}^k(X)
    :=\{x\in X:\text{no tangent cone at $x$ splits off an $\setR^{k+1}$ factor}\}
\]
the $k$-th singular stratum. For every $k\in \setN$, \cite[Theorem 1.8]{DPG} gives
\begin{equation}
    \label{eq:RCD_dimension_singular_strata}
    \dim_H(\mathcal{S}^k(X)) \le k.
\end{equation}
We adopt the notion of $\RCD$ boundary introduced
in \cite{DPG}, namely
\begin{equation}
    \label{eq:RCD-boundary-definition}
    \partial X
    :=\overline{\mathcal{S}^{n-1}(X)\setminus \mathcal{S}^{n-2}(X)}.
\end{equation}
An alternative definition, formulated in terms of tangent cones, was introduced
in \cite{KapMonGT}; the two notions agree $\haus^{n-1}$-almost everywhere by
\cite[Theorem 6.6]{BrueNaberSemola20} and its proof.

When $n=1$, the analogous definitions are understood with the evident
modifications, and all the statements below follow directly from the
one-dimensional classification \cite{KL16}. We shall therefore restrict the proofs to the
case $n\geq 2$.

We also recall two quantitative notions from \cite{BrueNaberSemola20}. Given $\delta>0$, a ball
$B(x,r)$ is called a \emph{$\delta$-boundary ball} if
\[
 \dist_{\rm GH}(B(x,r),B_{\setR^n_+}(0,r))<\delta r,
 \qquad
 \setR^n_+:=\setR^{n-1}\times[0,+\infty).
\]
A map $u=(u^1,\ldots,u^k):B(x,r)\to\setR^k$ is a
\emph{$\delta$-splitting map} if its components are harmonic, have a dimensional
Lipschitz bound $\Lip(u^a)\leq C(n)$, and satisfy
\[
 r^2\dashint_{B(x,r)}|\Hess u^a|^2\,d\haus^n<\delta,
 \qquad
 \dashint_{B(x,r)}
 \big|\langle\nabla u^a,\nabla u^b\rangle-\delta_{ab}\big|\,d\haus^n<\delta,
\]
where $\delta_{ab}$ denotes the Kronecker delta.
The precise normalization of the error will not be important below. We will use
repeatedly the fact that, on a sufficiently good boundary ball, the
$\setR^{n-1}$-valued almost-splitting map supplied by
\cite[Theorems 8.4 and 9.1]{BrueNaberSemola20} remains, after a suitable linear
normalization, a good almost-splitting map at every smaller boundary-centered scale.
This is the boundary version of the transformation argument of
\cite[Proposition 3.13]{BrueNaberSemola20}.

\subsection{Regularity of the ambient space and of its boundary}

We revisit some basic properties of $\RCD$ spaces in our terminology.

\begin{theorem}
\label{theo:RCD-n-regular-connected-manifold-part}
Let $(X,\dist,\haus^n)$ be a non-collapsed $\operatorname{RCD}(\kappa,n)$ space.
Then $X$ is $n$-regular, $\mathcal{M}(X) \supseteq \mathcal{R}_n(X)$ and $(\mathcal M(X),\dist)$ is a length space.  In particular, $\mathcal{M}(B(x,r))$ is connected for every $x\in X$ and every $r>0$. Moreover, $X$ is $n$-rectifiable and for
$\haus^{n}$-almost every $x\in X$, the approximate tangent metric
space to $X$ at $x$ is isometric to Euclidean $\setR^{n}$; consequently the
tangential function of the rectifiable part of $X$ is equal to one
almost everywhere. 
\end{theorem}

\begin{proof}
Rectifiability of general $\RCD(K,N)$ spaces $(X,\dist,\meas)$,
without any non-collapsing assumption, was established in
\cite{MondinoNaber19}, together with the uniqueness and
Euclidean structure of the pointed Gromov--Hausdorff tangent
at $\meas$-almost every point. Subsequently, absolute
continuity of $\meas$ with respect to the Hausdorff measure
of the corresponding chart dimension was proved independently
in \cite{KellMondino,GigliPasqualetto:CAG}.
Together, these results imply that a non-collapsed
$\operatorname{RCD}(\kappa,n)$ space $(X,\dist,\haus^n)$
is $n$-rectifiable and has a unique pointed
Gromov--Hausdorff tangent isometric to $\setR^n$ at
$\haus^n$-almost every point. Consequently,
Lemma \ref{lemma:tangential_function} yields
$\lambda_X(x)=1$ for $\haus^n$-almost every $x\in X$.

We first prove $n$-regularity. By \eqref{eq:RCD_dimension_singular_strata}, $\haus^n(X\setminus\mathcal R_n(X))=0$.
Moreover, \cite[Proposition 3.8]{MP26}, together with the regular-coordinate
results recalled there, shows that every point of $\mathcal R_n(X)$ has an open
neighbourhood which is a topological $n$-manifold and admits an ETR chart in the
sense of Definition~\ref{defin:ETR}. Let $U\subseteq\mathcal M(X)$ be the union
of these neighbourhoods. Then $U$ is open and
\begin{equation}
    \label{eq:full-measure-ETR-RCD}
    \haus^n(X\setminus U)=0.
\end{equation}
The measure $\haus^n$ is locally doubling on $X$ by Bishop--Gromov. Hence every
porous subset of $U$ has zero $\haus^n$-measure. Lemma
\ref{lemma:ETR=>Control_of_the_degree} therefore implies that $U$ is a
degree-regular metric $n$-manifold. Since
$\haus^n(\mathcal M(X)\setminus U)=0$, Lemma
\ref{lemma:degree-regularity_from_submanifold_to_manifold} upgrades this to the
whole manifold part $\mathcal M(X)$. Equation
\eqref{eq:full-measure-ETR-RCD} also gives
$\haus^n(\mathcal S(X))=0$. Thus $(X,\dist)$ is $n$-regular.

Let us prove the length-space property. Fix $x,y\in\mathcal M(X)$
and $\varepsilon>0$. Since $\mathcal M(X)$ is open, we can choose
$0<r<\varepsilon$ such that $B(x,r)\cup B(y,r)\subseteq\mathcal M(X)$. The set $\mathcal R_n(X)$ has full $\haus^n$-measure in $X$. Moreover, by
\cite[Theorem~6.5]{Deng}, $\mathcal R_n(X)$ is $\haus^n$-a.e. convex: for
$\haus^n\times\haus^n$-almost every pair of points in
$\mathcal R_n(X)\times\mathcal R_n(X)$ there exists a minimizing geodesic
joining them and entirely contained in $\mathcal R_n(X)$. Since every nonempty
open ball has positive $\haus^n$-measure, we can therefore choose
\[
    x'\in B(x,r)\cap\mathcal R_n(X),
    \qquad
    y'\in B(y,r)\cap\mathcal R_n(X),
\]
such that a minimizing geodesic $\gamma$ from $x'$ to $y'$ is entirely
contained in $\mathcal R_n(X)$.

Let $\gamma_x$ and $\gamma_y$ be minimizing geodesics from $x$ to $x'$ and
from $y'$ to $y$, respectively. Since $\dist(x,x')<r$ and
$\dist(y,y')<r$, we have
\[
    \gamma_x\subseteq B(x,r)\subseteq\mathcal M(X),
    \qquad
    \gamma_y\subseteq B(y,r)\subseteq\mathcal M(X).
\]
Furthermore, $\gamma\subseteq\mathcal R_n(X)\subseteq\mathcal M(X)$. Hence the
concatenation $\gamma_x*\gamma*\gamma_y$ is a rectifiable curve in
$\mathcal M(X)$ joining $x$ to $y$, and
\begin{align*}
    \operatorname{Length}(\gamma_x*\gamma*\gamma_y)
    &=\dist(x,x')+\dist(x',y')+\dist(y',y)\\
    &\leq \dist(x,y)+2\dist(x,x')+2\dist(y,y')\\
    &<\dist(x,y)+4\varepsilon.
\end{align*}
Since $\varepsilon>0$ is arbitrary, the intrinsic length distance on
$\mathcal M(X)$ is bounded above by $\dist$. The reverse inequality is
immediate from the definition of length, and therefore the two distances
coincide. 

We now prove that $\mathcal M(B(x,r))$ is path connected. By Remark~\ref{rmk:n-regular_open_sets},
\[
    \mathcal M(B(x,r))=B(x,r)\cap\mathcal M(X).
\]
Moreover, $\mathcal M(X)$ is dense in $X$: indeed,
$\haus^n(X\setminus\mathcal M(X))=0$, while every nonempty open ball in an
$\RCD$ space has positive $\haus^n$-measure. Fix $p,q\in B(x,r)\cap\mathcal M(X)$. Set
\[
    a:=\min\{r-\dist(x,p),\,r-\dist(x,q)\}>0
\]
and choose $0<\delta<a/4$. By density, there exists
$z\in B_\delta(x)\cap\mathcal M(X)$. Since $\mathcal M(X)$ is a length
space, there are rectifiable curves $\alpha_p,\alpha_q\subseteq\mathcal M(X)$
joining $z$ to $p$ and $z$ to $q$, respectively, such that
\[
    \operatorname{Length}(\alpha_p)<\dist(z,p)+\delta,
    \qquad
    \operatorname{Length}(\alpha_q)<\dist(z,q)+\delta.
\]
For every $w\in\alpha_p$ we have
\[
    \dist(x,w)
    \leq \dist(x,z)+\operatorname{Length}(\alpha_p)
    < \dist(x,p)+3\delta<r,
\]
and the same argument applies to $\alpha_q$. Hence both curves are contained
in $B(x,r)\cap\mathcal M(X)$, and their concatenation joins $p$ to $q$ there.
Thus $\mathcal M(B(x,r))$ is path connected.
\end{proof}

We now turn to the boundary as a metric space in its own right. We recall some properties needed for the regularity of the boundary and for the study of the boundary of an orientation.

\begin{proposition}
\label{prop:RCD_basic_properties}
Let $(X,\dist,\haus^n)$ be a non-collapsed $\operatorname{RCD}(\kappa,n)$ space.
Then the following properties hold.
\begin{enumerate}
    \item[(i)] For every $x\in X$,
    \[
        \Theta^n(x,\haus^n)\leq1.
    \]
    Equivalently, if $v_{\kappa,n}$ denotes the volume function of the
    $n$-dimensional simply connected model space with constant sectional curvature $\kappa/(n-1)$, then
    \[
        \haus^n(B(x,r))\leq v_{\kappa,n}(r),\qquad \forall r>0.
    \]
    \item[(ii)] The set $\partial X$ is $(n-1)$-rectifiable and
    $\haus^{n-1}\rstr\partial X$ is locally finite. More quantitatively, for
    every compact set $K\subseteq X$ there exist $r_K>0$ and $C_K<\infty$ such
    that
    \begin{equation}
        \label{eq:local-boundary-upper-RCD}
        \haus^{n-1}(B(x,r)\cap\partial X)
        \leq C_K\cdot r^{n-1}
    \end{equation}
    for every $x\in K\cap\partial X$ and $0<r<r_K$.
    \item[(iii)] There is a closed set $C\subseteq X$ with
    $\dim_HC\leq n-2$ such that $X\setminus C$ is a topological
    $n$-manifold with boundary. Moreover, at every point of
    \[
       \mathcal R_{\partial X}
       := \mathcal{S}^{n-1}(X)
          \setminus \mathcal{S}^{n-2}(X)
    \]
    the tangent cone is unique and isometric to Euclidean $\setR^n_+$, and
    \begin{equation}
        \label{eq:boundary-minus-top-stratum}
        \partial X\setminus\mathcal R_{\partial X}
        \subseteq  \mathcal{S}^{n-2}(X).
    \end{equation}
\end{enumerate}
\end{proposition}

\begin{proof}
The existence of the density and the inequality
$\Theta^n(x,\haus^n)\leq1$ follow from Bishop--Gromov monotonicity and the
non-collapsed volume rigidity theory; see, for instance,
\cite[Section 2.4]{BrueNaberSemola20}. The model-volume estimate follows by
letting the smaller radius tend to zero in Bishop--Gromov.

The rectifiability of $\partial X$, the local upper estimate for its
$(n-1)$-dimensional measure, and the uniqueness of the half-space tangent on
$\mathcal R_{\partial X}$ are contained in
\cite[Theorems 1.4 and 6.6]{BrueNaberSemola20}. Those theorems are stated after
normalizing the curvature and assuming a lower volume bound on a unit ball.
To obtain \eqref{eq:local-boundary-upper-RCD}, fix a compact set $K$, rescale on
a fixed sufficiently small radius so that the normalized curvature lower bound
is at least $-(n-1)$, and use Bishop--Gromov together with compactness of $K$
to obtain a uniform positive lower volume bound on the finitely many balls
needed to cover $K$. The estimate in \cite[Theorem 1.4(i)]{BrueNaberSemola20}
then gives \eqref{eq:local-boundary-upper-RCD}.

The codimension-two topological regularity statement is
\cite[Theorem 9.2]{BrueNaberSemola20}. The inclusion
\eqref{eq:boundary-minus-top-stratum} follows from the lower semicontinuity of
the density and the definition of the boundary; see
\cite[(1.10) and (1.11)]{BrueNaberSemola20}.
\end{proof}

The next lemma
is the boundary analogue of the fixed-splitting-chart decomposition in
\cite[Proposition 3.2]{MP26}. 

\begin{lemma}[ETR for the boundary regular manifold]
\label{lem:ETR-boundary-RCD}
Let $(X,\dist,\haus^n)$ be a non-collapsed $\operatorname{RCD}(\kappa,n)$ space.
There exists a relatively open subset
$M_{\partial X}\subseteq\partial X$ such that
\begin{equation}
    \label{eq:boundary-regular-complement-codim2}
    \dim_{H}(\partial X\setminus M_{\partial X})\leq n-2,
\end{equation}
and $M_{\partial X}$ is an \textup{ETR} metric $(n-1)$-manifold. 
\end{lemma}

\begin{proof}
By \eqref{eq:boundary-minus-top-stratum} and \eqref{eq:RCD_dimension_singular_strata}, the complement of
$\mathcal R_{\partial X}$ in $\partial X$ has Hausdorff dimension at most
$n-2$. At every $p\in\mathcal R_{\partial X}$ the unique tangent is Euclidean
$\setR^n_+$, by Proposition \ref{prop:RCD_basic_properties}.
Fix the small parameters required below before choosing
the radii $r_p$. We choose each $r_p$ sufficiently small
that, after rescaling by $r_p^{-1}$, both the boundary-ball
error and the negative part of the curvature lower bound
satisfy all the smallness requirements of
\cite[Theorem 8.4 and Proposition 3.13]{BrueNaberSemola20}.
In particular, for the parameter $\delta>0$ used in
Proposition 3.13, we require
\[
 \max\{0,-\kappa\}\,r_p^2\leq\delta^2(n-1).
\]
Hence, after choosing $r_p>0$ sufficiently small and rescaling by
$r_p^{-1}$, $B(p,16r_p)$ is an arbitrarily good boundary ball. By
\cite[Theorem 8.4]{BrueNaberSemola20}, $B(p,8r_p)\cap\partial X$ is a
topological $(n-1)$-manifold. We may therefore take
\[
 M_{\partial X}:=
 \bigcup_{p\in\mathcal R_{\partial X}}
       (B(p,r_p/2)\cap\partial X).
\]
This is relatively open in $\partial X$, contains
$\mathcal R_{\partial X}$, is a topological $(n-1)$-manifold, and satisfies
\eqref{eq:boundary-regular-complement-codim2}.
    We prove that $M_{\partial X}$ satisfies the ETR property. Fix
    $x\in M_{\partial X}$. There exist
    $p\in\mathcal{R}_{\partial X}$ and $r>0$ such that
    \[
        x\in B(p,r/2)\cap\partial X
    \]
    and, after rescaling by $r^{-1}$, $B(p,16r)$ is an arbitrarily good
    boundary ball. After rescaling, we write $r=1$ and retain the small
curvature lower bound ensured by the original choice
of the radii. 
    We choose the constants below in the following order. Fix
    \[
        0<\varepsilon_0<\frac{\omega_{n-1}}{8}.
    \]
    We then choose an auxiliary splitting error $\eta>0$, depending only on
    $n$ and $\varepsilon_0$, sufficiently small for the argument in the proof
    of \cite[Theorem 8.4.(ii)]{BrueNaberSemola20} to produce a good set whose
    complement has $\haus^{n-1}$-measure at most
    $\varepsilon_0$ in the unit boundary ball. We finally choose
    $\delta_0$, and the splitting error used below, sufficiently small in
    terms of $n$ and $\eta$.

    By \cite[Theorem 8.4]{BrueNaberSemola20}, there exists a harmonic
    $\eta$-splitting map
    \[
        u:B(p,8) \to\setR^{n-1}
    \]
    such that
    \begin{equation}
        \label{eq:holder-boundary-chart-ETR}
        c\,\dist(z,z')^{1+\alpha}
        \leq
        |u(z)-u(z')|
        \leq
        C\,\dist(z,z')
    \end{equation}
    for every $z,z'\in B(p,1)\cap\partial X$, where
    $c>0$, $C<\infty$ and $\alpha>0$. In particular, the restriction of
    $u$ to $B(p,1)\cap\partial X$ is Lipschitz and injective.

    Choose an open neighbourhood $U$ of $x$ in $M_{\partial X}$ such that
    \[
        \overline{U}\subseteq B(p,3/4)\cap M_{\partial X}.
    \]
    Since $U$ is a topological $(n-1)$-manifold, the invariance of domain,
    together with \eqref{eq:holder-boundary-chart-ETR}, implies that
    \[
        u:U\to u(U)\subseteq\setR^{n-1}
    \]
    is a Lipschitz homeomorphism.

We claim that, after choosing $\delta_0>0$ sufficiently small, the following
holds. For every $y\in U$ and every sufficiently small $s>0$ there exists a Borel set
\[
    G_{y,s}\subseteq B(y,s)\cap\partial X
\]
such that
\begin{equation}
    \label{eq:good-set}
    \haus^{n-1}
    \big((B(y,s)\cap\partial X)\setminus G_{y,s}\big)
    \leq \varepsilon_0 s^{n-1},
\end{equation}
and $u|_{G_{y,s}}$ is bi-Lipschitz.

Let us verify the claim. We first observe that the hypotheses of the
transformation theorem \cite[Proposition 3.13]{BrueNaberSemola20} (after \cite{CJN:21}) are
satisfied after recentering at $y$. Indeed, by
\cite[Theorem 8.1]{BrueNaberSemola20}, after choosing the original
boundary error sufficiently small, every ball
\[
    B(y,t),\qquad 0<t<1,
\]
is a boundary ball with the error required in
\cite[Proposition 3.13]{BrueNaberSemola20}.

It remains to check the almost-splitting assumption on the fixed ball $B(y,2)$.
Since $y\in B(p,1)$, we have
\[
    B(y,2)\subseteq B(p,3)\subseteq B(p,8)
    \qquad\text{and}\qquad
    B(p,1)\subseteq B(y,2).
\]
Hence, by Bishop--Gromov,
\[
    \frac{\haus^n(B(p,8))}{\haus^n(B(y,2))}
    \leq C(n).
\]
Since $u:B(p,8)\to\mathbb{R}^{n-1}$ is an $\eta$-splitting map, it follows from
the definition that
\[
    \fint_{B(y,2)}|\Hess u_a|^2\,d\haus^n
    \leq C(n)\eta
\]
and
\[
    \fint_{B(y,2)}
    \left|
    \langle\nabla u_a,\nabla u_b\rangle-\delta_{ab}
    \right|\,d\haus^n
    \leq C(n)\eta
\]
for every $a,b=1,\ldots,n-1$. Thus, after reducing $\eta$ if necessary,
$u:B(y,2)\to\mathbb{R}^{n-1}$ satisfies the almost-splitting assumption of
\cite[Proposition 3.13]{BrueNaberSemola20}.

We can therefore apply the transformation theorem. For every prescribed
sufficiently small scale $s>0$ it gives a matrix
\[
    A_{y,s}\in {\rm GL}(n-1)
\]
such that $v:=A_{y,s}\circ u$ is an arbitrarily good almost-splitting map on $B(y, 16s)$. Notice that the matrix is allowed to depend on the scale $s$.

    We rescale now the metric by $s^{-1}$ and replace $v$ by
    \[
        \widetilde v:=s^{-1}\big(v-v(y)\big).
    \]
    In the rescaled space, $\widetilde v$ is a sufficiently good splitting
    map on $B(y,16)$ and the corresponding ball is a sufficiently good
    boundary ball. We can therefore apply verbatim the argument proving
    \cite[Theorem 8.4(ii)]{BrueNaberSemola20} (cf.\ \cite[Theorem~6.5]{MondinoNaber19}) to the prescribed splitting
    map $\widetilde v$. We briefly recall the point of that argument. One
    removes from $B(y,1)$ those points at which the
    maximal function of $|\Hess\widetilde v|^2$ is too large.
    Lemma 4.16 in \cite{BrueNaberSemola20} implies that at every point of
    the remaining good set, call it $\tilde{G}_{y,s}$, the map $\widetilde v$ is a sufficiently good
    almost-splitting map at all the relevant scales. Corollary 7.2 in
    \cite{BrueNaberSemola20} then gives, for any two points $z,z'\in \tilde{G}_{y,s}\cap \partial X$,
    \[
        (1-\tau)\dist(z,z')
        \leq
        |\widetilde v(z)-\widetilde v(z')|
        \leq
        (1+\tau)\dist(z,z')
    \]
    for a fixed $\tau<1$. The Vitali argument in
    \cite[Theorem 8.4(ii)]{BrueNaberSemola20}, using the
    $\haus^{n-1}$-upper bound for the boundary and the $L^2$-bound for the
    Hessian, gives
    \[
        \haus^{n-1}
        \big(
        (B(y,1)\cap\partial X)\setminus \widetilde G_{y,s}
        \big)
        \leq C(n)\eta^{1/2}.
    \]
    By our choice of $\eta$,
    \[
        C(n)\eta^{1/2}\leq\varepsilon_0.
    \]
    Scaling back, we obtain a Borel set
    $G_{y,s}\subseteq B(y,s)\cap\partial X$ satisfying
    \eqref{eq:good-set} and such that
    $v|_{G_{y,s}}$ is bi-Lipschitz.

    Since $A_{y,s}$ is invertible, also $u|_{G_{y,s}}$ is bi-Lipschitz.
    Indeed, if
    \[
        c_{y,s}\dist(z,z')
        \leq
        |A_{y,s}u(z)-A_{y,s}u(z')|
        \leq
        C_{y,s}\dist(z,z')
    \]
    for $z,z'\in G_{y,s}$, then
    \[
        \frac{c_{y,s}}{\|A_{y,s}\|}
        \dist(z,z')
        \leq
        |u(z)-u(z')|
        \leq
        \|A_{y,s}^{-1}\|C_{y,s}\dist(z,z').
    \]
    This proves \eqref{eq:good-set} and the claimed local
    property.

    Finally, we establish the full-measure decomposition for the fixed
    topological chart $u$. Let $E\subseteq U$ be a Borel set with
    \[
        \haus^{n-1}(E)>0.
    \]
    By the $(n-1)$-rectifiability and the local finiteness of
    $\haus^{n-1}\rstr\partial X$, proved in
    \cite[Theorem 1.4(i)]{BrueNaberSemola20}, and by
    Lemma \ref{lemma:density_points_rectifiable_space}, we can choose
    $y\in E$ such that
    \begin{equation}
        \label{eq:density-boundary-E}
        \lim_{s\to 0}
        \frac{\haus^{n-1}(B(y,s)\cap\partial X)}
        {\omega_{n-1}s^{n-1}}
        =1
    \end{equation}
    and
    \begin{equation}
        \label{eq:density-E-relative-boundary}
        \lim_{s\to 0}
        \frac{
        \haus^{n-1}
        \big((B(y,s)\cap\partial X)\setminus E\big)}
        {s^{n-1}}
        =0.
    \end{equation}
    Since $U$ is open in $M_{\partial X}$, for $s>0$ sufficiently small
    $B(y,s)\cap\partial X\subseteq U$. By
    \eqref{eq:density-boundary-E} and
    \eqref{eq:density-E-relative-boundary}, we can also assume that
    \[
        \haus^{n-1}(B(y,s)\cap\partial X)
        \geq \frac34\omega_{n-1}s^{n-1}
    \]
    and
    \[
        \haus^{n-1}
        \big((B(y,s)\cap\partial X)\setminus E\big)
        \leq\varepsilon_0s^{n-1}.
    \]
    Set
    \[
        A_s:=B(y,s)\cap\partial X.
    \]
    Since
    \[
        A_s\setminus(E\cap G_{y,s})
        =
        (A_s\setminus E)\cup(A_s\setminus G_{y,s}),
    \]
    we obtain
    \begin{align*}
        \haus^{n-1}(E\cap G_{y,s})
        &=
        \haus^{n-1}(A_s)
        -
        \haus^{n-1}
        \big(A_s\setminus(E\cap G_{y,s})\big)\\
        &\geq
        \haus^{n-1}(A_s)
        -
        \haus^{n-1}(A_s\setminus E)
        -
        \haus^{n-1}(A_s\setminus G_{y,s})\\
        &\geq
        \left(
        \frac34\omega_{n-1}-2\varepsilon_0
        \right)s^{n-1}
        >0.
    \end{align*}
    By the inner regularity of $\haus^{n-1}$, there exists a compact set
    \[
        K\subseteq E\cap G_{y,s}
    \]
    with $\haus^{n-1}(K)>0$. The restriction $u|_K$ is bi-Lipschitz.

    We have therefore proved that every Borel subset of $U$ with positive
    $\haus^{n-1}$-measure contains a compact subset of positive measure on
    which $u$ is bi-Lipschitz. Since $\overline U$ is compact and
    $\haus^{n-1}\rstr\partial X$ is locally finite, we have
    $\haus^{n-1}(U)<\infty$. Consider a maximal family of pairwise disjoint
    compact sets
    \[
        \{K_i\}_{i\in I},\qquad K_i\subseteq U,
    \]
    such that $\haus^{n-1}(K_i)>0$ and $u|_{K_i}$ is bi-Lipschitz.
    Since
    $\haus^{n-1}(U)<\infty$, the index set $I$ is countable, and we write
    $I\subseteq\setN$. If
    \[
        \haus^{n-1}
        \left(
        U\setminus\bigcup_{i\in I}K_i
        \right)>0,
    \]
    the property proved in the previous paragraph, applied to the Borel set
    $U\setminus\bigcup_{i\in I}K_i$, would give a compact set of positive
    measure, disjoint from every $K_i$, on which $u$ is bi-Lipschitz. This
    contradicts maximality. Therefore
    \[
        \haus^{n-1}
        \left(
        U\setminus\bigcup_{i\in I}K_i
        \right)=0.
    \]
    Thus the Lipschitz homeomorphism $u:U\to u(U)$ and the pairwise
    disjoint compact sets $K_i$ satisfy Definition \ref{defin:ETR}.
    Since $x\in M_{\partial X}$ was arbitrary, $M_{\partial X}$ is an ETR
    metric $(n-1)$-manifold.
\end{proof}

We can summarize the results as follows.

\begin{proposition}
\label{prop:RCD-boundary-nminus1-regular}
Let $(X,\dist,\haus^n)$ be a non-collapsed $\RCD(\kappa,n)$ space. Then the metric space $(\partial X,\dist)$ is $(n-1)$-regular. Moreover, for
$\haus^{n-1}$-almost every $x\in\partial X$, the approximate tangent metric
space to $\partial X$ at $x$ is Euclidean $\setR^{n-1}$; consequently the
tangential function of the rectifiable part of $\partial X$ is equal to one
almost everywhere.
\end{proposition}

\begin{proof}
By Lemma~\ref{lem:ETR-boundary-RCD}, $M_{\partial X}$ is an ETR metric
$(n-1)$-manifold and its complement in $\partial X$ is
$\haus^{n-1}$-negligible. On every relatively compact subset of
$M_{\partial X}$, the boundary volume $\varepsilon$-regularity theorem
\cite[Corollary 8.7]{BrueNaberSemola20}, after scaling, gives a local lower
Ahlfors $(n-1)$-bound. Lemma~\ref{lemma: lower-ahlfors-no-porous} therefore
implies that porous subsets of $M_{\partial X}$ have zero
$\haus^{n-1}$-measure. Lemma~\ref{lemma:ETR=>Control_of_the_degree} shows that
$M_{\partial X}$ is degree-regular.

Since $M_{\partial X}$ is open in the manifold part
$\mathcal M(\partial X)$ and has full $\haus^{n-1}$-measure there, Lemma
\ref{lemma:degree-regularity_from_submanifold_to_manifold} gives that
$\mathcal M(\partial X)$ is degree-regular. Furthermore,
\[
 \haus^{n-1}\bigl(\partial X\setminus\mathcal M(\partial X)\bigr)=0.
\]
Thus $(\partial X,\dist)$ is $(n-1)$-regular.

Finally, $\mathcal R_{\partial X}$ has full $\haus^{n-1}$-measure in
$\partial X$, and at every point of this set the ambient tangent is the
Euclidean half-space by Proposition \ref{prop:RCD_basic_properties}. The Hausdorff closeness of the boundary to the flat
hyperplane at all sufficiently small scales, given by
\cite[Theorem 8.1.(ii)]{BrueNaberSemola20}, implies that the metric tangent of
the boundary is $\setR^{n-1}$. The last assertion follows from Lemma
\ref{lemma:tangential_function}.
\end{proof}

\subsection{Orientability and homology of non-collapsed $\RCD$ spaces}

Orientable non-collapsed $\RCD$ spaces are LUO, hence uniquely orientable.

\begin{proposition}
\label{prop:RCD_orientable_implies_uniquely_orientable}
    Let $(X,\dist,\mathcal \haus^n)$ be a non-collapsed $\operatorname{RCD}(\kappa,n)$ space. If $X$ is orientable then it is \textup{LUO} and uniquely orientable. Every orientation $T\in \mathcal{I}_{n,\loc}(X)$ satisfies $\norm{T} = \haus^n$. 
\end{proposition}
\begin{proof}
    Every open ball $B(x,r)$ is orientable, because its manifold part is an open submanifold of the orientable manifold $\mathcal M(X)$. By Remark~\ref{rmk:n-regular_open_sets}, each $B(x,r)$ is $n$-regular, and by Theorem \ref{theo:RCD-n-regular-connected-manifold-part} its manifold part is connected. Corollary \ref{cor:n-regular_unique_orientability} therefore implies that every ball $B(x,r)$ is uniquely orientable. In particular, $X$ is LUO. The same proof applied to the whole $X$ shows that $X$ is uniquely orientable. Every orientation $T\in \mathcal{I}_{n,\loc}(X)$ satisfies $\norm{T} = \haus^n$ because $X$ is $n$-rectifiable and $\lambda_X = 1$ $\haus^n$-a.e., by Theorem \ref{theo:RCD-n-regular-connected-manifold-part}. 
\end{proof}

The next result shows that the orientation of a non-collapsed $\RCD$ space with boundary behaves as in the case of topological manifolds with boundaries. More precisely, the boundary of an orientation is the orientation of the boundary. It can be seen as a Stokes' Theorem for non-collapsed $\RCD$ spaces.

\begin{theorem}\label{theo:RCD-boundary-mass}
Let $(X,\dist,\haus^n)$ be a non-collapsed $\mathrm{RCD}(\kappa,n)$
space, and let $T\in\mathcal I_{n,\loc}(X)$ be an orientation. Then $\spt(\partial T) = \partial X$ and
\begin{equation}
    \label{eq:RCD_mass_of_boundary}
    \norm{\partial T}
    =
    \haus^{n-1}\rstr\partial X.
\end{equation}
In particular, if $X$ is orientable then $\partial X$ is orientable. In this case, for every orientation $T\in \mathcal{I}_{n,\loc}(X)$, $\partial T|_{\partial X} \in \mathcal{I}_{n-1,\loc}(\partial X)$ is an orientation of $\partial X$.
\end{theorem}

{\begin{proof}
    We assume $n\geq 2$, as the case $n=1$ follows directly from the one-dimensional classification of non-collapsed RCD spaces \cite{KL16}. By Remark \ref{rmk:n-regular_orientation}, we have that $\spt(\partial T) \subseteq \mathcal{S}(X)$. By Theorem \ref{theo:RCD-n-regular-connected-manifold-part}, the definition of $\partial X$ and \eqref{eq:RCD_dimension_singular_strata}, we have 
    $$\haus^{n-1}(\mathcal{S}(X) \setminus \partial X) \le \haus^{n-1}(\mathcal{S}^{n-2}(X)) = 0.$$
    Therefore, Proposition \ref{prop:RCD_basic_properties} implies that the assumptions of Corollary \ref{cor:boundary_finite_mass_rectifiable} are satisfied and hence $\partial T \in \bM_{n-1,\loc}(X)$. By Lemma \ref{lemma:mass-measure-abso-cont} we have that $\norm{\partial T}(X\setminus \partial X) = 0$. Since $\partial X$ is closed, we have $\spt(\partial T) \subseteq \partial X$, by \cite[Theorem 4.3.(2)]{Lang}.

    Let us fix $x\in \mathcal{S}^{n-1}(X) \setminus \mathcal{S}^{n-2}(X)$. By Proposition \ref{prop:RCD_basic_properties}, the tangent cone at $x$ is unique and isometric to Euclidean $\setR^n_+$. Thus, after rescaling at a sufficiently small radius, we may apply \cite[Theorem 9.1]{BrueNaberSemola20}. We obtain an open neighbourhood $V$ of $x$ and a homeomorphism
    \[
    F\colon V\longrightarrow F(V)\subseteq \setR^n_+
    \]
    that is Lipschitz and such that
    \[
    F(x)=0,
    \qquad
    F(V\cap\partial X)=F(V)\cap\partial\setR^n_+,
    \]
    and such that $F(V)$ contains a closed half-ball $\overline{B}_{\setR^n_+}(0,s)$ for some $s>0$. 
    
    Arguing as in the proof of Lemma \ref{lemma: good-subsets-existence}, we conclude that for almost all $r\in(0,s)$ we have $\haus^n(F^{-1}(S_r)) = 0$, where $S_r:=S_{\mathbb{R}^n_+}(0,r) = \{z\in \setR^n_+\,:\, \dist_E(z,0) = r\}$. By the slicing theorem \cite[Theorem 6.2]{Lang} applied to $T|_V$ and $u(y) = \dist_E(0,F(y))$ we find $r\in(0,s)$ such that $\haus^n(F^{-1}(S_r)) = 0$, the slice $\langle T|_V,u,r\rangle \in \bN_{n-1,\loc}(X)$ exists and satisfies $\spt(\langle T|_V,u,r\rangle) \subseteq F^{-1}(S_r)$.
    
    Set $B\coloneqq \overline{B}_{\setR^n_+}(0,r)$, $S\coloneqq S_r$ and $D\coloneqq F^{-1}(B)$. By \cite[Definition 6.1]{Lang}, we have
    \begin{equation}\label{eq: RCD-bdry-current-slice}
        \partial (T|_V\rstr D) =  (\partial (T|_V))\rstr D  + \langle T|_V,u,r\rangle.
    \end{equation}
    Notice that that $F\colon D\to B$ is a homeomorphism and 
    $$\haus^n(\partial D) = \haus^n(\partial X\cap D) + \haus^n(F^{-1}(S)) = 0$$
    because $\haus^n\rstr \partial X = 0$. This shows that $D$ is a regular disk, according to Definition \ref{def: good-ball}. Moreover, $F\colon D \to B$ is a homeomorphism that is Lipschitz and satisfies 
    \begin{equation}\label{eq: RCD-bdry-current-eq-homeom}
        F(D\cap \partial X) = B \cap \partial \setR^n_+ \quad \textup{and} \quad F(\partial D) = (B \cap \partial \setR^n_+) \cup S.
    \end{equation}
    Thus, $\textup{int}(D)$ is fully contained in $\mathcal{M}(V)$. 
    
    By Remark \ref{rmk:n-regular_orientation}, $T|_{\mathcal{M}(V)}$ is an orientation of $\mathcal{M}(V)$. 
    We equip $\textup{int}(D)$ with the orientation $\mathfrak{o}$ induced by the orientation of $\mathcal{M}(V)$ associated to $T|_{\mathcal{M}(V)}$, according to Theorem \ref{theo: characterization-orientability-metric-manifold-with-details}. Let $T_D^\mathfrak{o}\in \mathcal{I}_n(D)$ be the current associated to $D$ as constructed in Section \ref{sec: orient-regular-disks}. Since $\spt(\partial T) \subseteq \partial X$, \eqref{eq:support_boundary_restriction} gives that $\spt(\partial (T|_D)) \subseteq \partial D$.
    Corollary \ref{cor: uniqueness-T_U} implies that there exists $k \in \mathbb{Z}$ such that $k\cdot T_D^\mathfrak{o} = T|_D$. By Theorem \ref{theo: characterization-orientability-metric-manifold-with-details}, $T|_{D'} = T_{D}^\mathfrak{o}|_{D'}$ for every regular disk $D' \subseteq \textup{int}(D)$. Hence, $k=1$ and $T|_D = T_D^\mathfrak{o}$.
    
    Let us consider the doubles $\hat{D}$ and $\hat{B}$ of $D$ and $B$, respectively, as defined in Definition \ref{def: double-good-subset}. Moreover, let $\hat{\br{B}} \in \bI_n(\hat{B})$ and $\hat{T}_D^\mathfrak{o} \in \bI_n(\hat{D})$ be the associated currents as in Proposition \ref{prop: existence-metric-fund-class-double}, where $\br{B}$ is the fundamental class of $B$. Fix a bi-Lipschitz homeomorphism $\varphi \colon \hat{B} \to \mathbb{S}^n$, where $\mathbb{S}^n$ is a equipped with the standard round metric. Furthermore, equip $\textup{int}(B)$ and $\mathbb{S}^n$ with an orientation such that $F$ and $\varphi$ are sense-preserving. The map $F$ extends to a homeomorphism $\hat{F} \colon \hat{D} \to \hat{B}$ that is Lipschitz and has degree $1$. Thus, Proposition \ref{prop: existence-metric-fund-class-double} implies that $(\varphi \circ \hat{F})_\#\hat{T}_D^\mathfrak{o} = \br{\mathbb{S}^n}$. Since $\varphi$ is a bi-Lipschitz homeomorphism, it sends $\hat{\br{B}}$ to $\br{\mathbb{S}^n}$ and hence, $\hat{F}_\#\hat{T}_D^\mathfrak{o}=\hat{\br{B}}$. In particular,
    $$\iota_{1\#}{\br{B}} - \iota_{2\#}{\br{B}} =\hat{\br{B}} = \hat{F}_\#\hat{T}_D^\mathfrak{o} = \hat{F}_\#(\iota_{1\#}T_D^\mathfrak{o} - \iota_{2\#}T_D^\mathfrak{o}) = F_{1\#}(T|_D) - F_{2\#}(T|_D),$$
    where $F_i := \hat{F}\circ \iota_i$ and the $\iota_i$ are the isometric embeddings that identify $D$ and $B$ with their copies $D_i$ and $B_i$ in $\hat{D}$ and $\hat{B}$, respectively. Since $\iota_i \circ F|_D = \hat{F} \circ \iota_i$ for $i =1,2$ we have $F_{1\#}(T|_D) = \iota_{1\#}\br{B}$. By composing with the relative $\iota_1^{-1}$ maps, and using that $(T|_V)|_D=T|_D$ by Proposition \ref{prop:restriction_of_current_to_open_subsets}.(iv), we get that 
    $$(F_\#(T|_V))|_B \stackrel{\eqref{eq:restriction_vs_pushforward}}{=} (F|_D)_{\#}(T|_D) = \br{B}.$$
    Hence, by Corollary \ref{cor:extension_vs_restriction},
    $$F_\#(T|_V\rstr D) = (F_\#(T|_V)) \rstr B = ((F_\#(T|_V))|_B)^{F(V)} = \br{B}^{F(V)}.$$
    Therefore, 
    \begin{equation}
        \begin{aligned}
            F_\#(\partial(T|_V\rstr D)) \stackrel{\eqref{eq:boundary_pushforward_commute}}{=} \partial(F_\#(T|_V \rstr D)) = \partial (\br{B}^{F(V)}) &=  \partial( \br{\overline{B}_{\setR^n_+}(0,s)}^{F(V)}) \\
            &= \br{\overline{B}_{\setR^{n-1}}(0,r))}^{F(V)} - \br{S}^{F(V)}.
        \end{aligned}
    \end{equation}
    Moreover, by \eqref{eq: RCD-bdry-current-slice} we have
    $$F_\#(\partial (T|_V \rstr D)) = F_\#((\partial (T|_V))\rstr D ) + F_\#(\langle T|_V,u,r\rangle).$$
    Since $\spt(\langle T|_V,u,r\rangle) \subseteq F^{-1}(S)$, then $\spt(F_\#(\langle T|_V,u,r\rangle)) \subseteq S$. Moreover,
    \begin{equation}
    \begin{aligned}
        \spt(F_\#(\partial (T|_V) \rstr D)) \subseteq F(\spt(\partial (T|_V)) \cap D) &\stackrel{\eqref{eq:support_boundary_restriction}}{\subseteq} F((\spt(\partial T)\cup \partial V) \cap D) \\
        &\subseteq F(D\cap \partial X) =      
        \overline{B}_{\setR^{n-1}(0,r)}.
    \end{aligned}
    \end{equation}    
    We conclude that 
    \begin{equation}
        \label{eq:F_push_boundary_disk}
        F_\#((\partial (T|_V))\rstr D) = \br{\overline{B}_{\setR^{n-1}}(0,r))}^{F(V)},
    \end{equation}
    so $(\partial (T|_V))\rstr D \neq 0$  and $x\in \spt((\partial (T|_V))\rstr D)$ by \eqref{eq:support_vs_pushforward} and the injectivity of $F$. Recall that $V$ is open, hence $\partial V \cap V = \emptyset$.  Therefore, 
    $$x\in \spt((\partial (T|_V))\rstr D) \subseteq \spt(\partial (T|_V)) \cap D \stackrel{\eqref{eq:support_boundary_restriction}}{\subseteq} \spt(\partial T) \cap D.$$
    As $x\in \mathcal{S}^{n-1}(X) \setminus \mathcal{S}^{n-2}(X)$ was arbitrary and $\partial X = \overline{\mathcal{S}^{n-1}(X) \setminus \mathcal{S}^{n-2}(X)}$ it follows that $\partial X \subseteq \spt(\partial T)$. We conclude that $\spt(\partial T) = \partial X$. 

    By Proposition \ref{prop:RCD-boundary-nminus1-regular}, $(\partial X, \dist)$ is $(n-1)$-regular. The current $S:=\partial T|_{\partial X} \in \mathcal{I}_{n-1,\loc}(\partial X)$ by Proposition \ref{prop:extension_restriction_integral_currents}. Moreover, $\spt(S) = \partial X$ by Lemma \ref{lemma:restriction_to_support}. Furthermore, using Corollary \ref{cor:extension_vs_restriction} multiple times, Lemma \ref{lemma:extension_vs_boundary} and that $S^X = \partial T\rstr \partial X = \partial T$, which follows again by Corollary \ref{cor:extension_vs_restriction}, we have that 
    \begin{equation}
        \begin{aligned}
            \partial S = (\partial S)^X|_{\partial X} = (\partial (S^X))|_{\partial X} = (\partial \partial T)|_{\partial X} = 0|_{\partial X} = 0.
        \end{aligned}
    \end{equation}
    Therefore, by Corollary \ref{cor:n-regular_full_support_current_implies_orientability}, we conclude that $\partial X$ is orientable. 
    
    Let again $x\in \mathcal{S}^{n-1}(X) \setminus \mathcal{S}^{n-2}(X)$. From the first part of the proof, we can find an open neighbourhood $V$ of $x$, a homeomorphism $F\colon V \to F(V) \subseteq \setR^n_+$ which is Lipschitz, and a regular disk $x\in D\subseteq V$ such that $F_\#((\partial(T|_V)) \rstr D) = \br{\overline{B}_{\setR^{n-1}}(0,r)}^{F(V)}$. 
    Let $A:= F^{-1}(B_{\setR^{n-1}}(0,r) \times \{0\}) \subseteq \partial X$, which is open in $\partial X$, hence orientable because of Remark \ref{rmk:n-regular_open_sets}. Let $W \in \mathcal{I}_{n-1,\loc}(A)$ be an orientation of $A$. By Proposition \ref{prop:RCD-boundary-nminus1-regular},
    $$\norm{W} = \haus^{n-1}\rstr A.$$
    By Lemma \ref{cor:cycles_open_sets}, $\partial(S|_A) = 0$. Since $\mathcal{M}(A) = A$ is connected, Proposition \ref{prop:unique_orientation_extension} gives $k\in \mathbb{Z}$ such that
    $$S|_A = k\cdot W.$$
    Let $U := B_{\setR^{n-1}}(0,r) \times \{0\}$. Notice that $\partial(T|_V) = (\partial T)|_V$ by Lemma \ref{cor:cycles_open_sets}, since $V$ is open. Therefore,
    \begin{equation}
        \begin{aligned}
            (F|_A)_\#(S|_A) = 
            (F|_{F^{-1}(U)})_\#(((\partial T)|_V)|_{F^{-1}(U)}) &= (F|_{F^{-1}(U)})_\#((\partial (T|_V) \rstr D)|_{F^{-1}(U)}) \\
            &\stackrel{\eqref{eq:restriction_vs_pushforward}}{=} \br{\overline{B}_{\setR^{n-1}}(0,r)}|_{U} = \br{U}.
        \end{aligned}
    \end{equation}
    Hence,
    $$k\cdot (F|_A)_\# W = \br{U}.$$
    Since $(F|_A)_\# W$ is integer rectifiable, we deduce that $k = \pm 1$. Consequently, by Proposition \ref{prop:restriction_of_current_to_open_subsets}, 
    $$\norm{S} \rstr A = \norm{S|_A} = \norm{W} = \haus^{n-1}\rstr A.$$
    Since this holds on a set of full $\haus^{n-1}$-measure, hence on a set of full $\norm{\partial T}$-measure of
    $\partial X$, then 
    $$\norm{\partial T} = \norm{\partial T \rstr \partial X} = \norm{\partial T|_{\partial X}} = \haus^{n-1}\rstr \partial X.$$
    Since $\partial(\partial T|_{\partial X}) = 0$, $\partial T|_{\partial X}$ is an orientation of $\partial X$.
\end{proof} 

As a consequence, the top-degree homology of non-collapsed $\RCD$ spaces behaves exactly as in the case of manifolds with boundary.

\begin{theorem}
\label{theo:RCD-homological-properties}
    Let $(X,\dist,\mathcal \haus^n)$ be a non-collapsed $\operatorname{RCD}(\kappa,n)$ space. 
    \begin{itemize}
        \item[(i)] If $X$ is orientable then $H_n^\textup{IC}(X) \cong \mathbb{Z}$ if and only if $\partial X = \emptyset$, otherwise $H_n^\textup{IC}(X) = \{0\}$ and $H_n^\textup{IC}(X,\partial X) \cong \mathbb{Z}$. 
        \item[(ii)] If $X$ is not orientable, then $H_n^\textup{IC}(X) = H_n^\textup{IC}(X,\partial X) = \{0\}$.
    \end{itemize}
\end{theorem}

\begin{proof}

If $X$ is orientable then it is uniquely orientable by Proposition \ref{prop:RCD_orientable_implies_uniquely_orientable}, so it has exactly two orientations $\pm T$. Proposition \ref{prop:unique_orientation_extension} implies that $H_n^{\mathrm{IC}}(X)\cong\mathbb Z$ if and only if $\partial T=0$. By
Theorem~\ref{theo:RCD-boundary-mass}, $    \|\partial T\|=\haus^{n-1}\rstr\partial X$. Since $\haus^{n-1}\rstr \partial X$ is locally positive, by \cite[Theorem 1.4.(iii) and Theorem 1.2]{BrueNaberSemola20}, we deduce that $\partial X=\emptyset$ if and only if $\partial T=0$. In case $\partial X \ne \emptyset$ then $T\in H_n^\textup{IC}(X,\partial X)$. Moreover, for every $S\in \bI_{n,\loc}(X)$ with $\spt(\partial S) \subseteq \partial X$, there exists $k\in \mathbb{Z}$ such that $S = k\cdot T$, by Proposition \ref{prop:unique_orientation_extension} and Theorem \ref{theo:RCD-n-regular-connected-manifold-part}. This proves
(i).

Finally, suppose that $X$ is not orientable. By Theorem \ref{theo:RCD-n-regular-connected-manifold-part}, $\mathcal M(X)$ is
connected. Hence, Proposition~\ref{prop:unique_orientation_extension}.(iii) applies and gives $H_n^{\mathrm{IC}}(X)=\{0\}$. If $S \in \bI_{n,\loc}(X)$ with $\spt(\partial S) \subseteq \partial X$, consider $S|_{\mathcal{M}(X)}$. Since $\mathcal{M}(X)$ is open, $\partial (S|_{\mathcal{M}(X)}) = 0$, by Lemma \ref{cor:cycles_open_sets}. Since $\mathcal{M}(X)$ is non-orientable, Corollary \ref{cor: any-bdry-less-current-implies-orient} implies that $S|_{\mathcal{M}(X)} = 0$. Since $\haus^n(X \setminus \mathcal{M}(X)) = 0$, $S = 0$ by Corollary \ref{cor:extension_vs_restriction} and Lemma \ref{lemma:mass-measure-abso-cont}. This proves (ii).
\end{proof}

We end this section with the proof of the Brouwer-like no-retraction theorem for non-collapsed $\RCD$ spaces.

\begin{corollary}\label{coro: Brouwer}
    Let $(X,\dist,\mathcal \haus^n)$ be an orientable, non-collapsed $\operatorname{RCD}(\kappa,n)$ space. Then, there exists no proper Lipschitz retraction $r\colon X \to \partial X$.
\end{corollary}
\begin{proof}
    Let $T\in \bI_{n,\loc}(X)$ be an orientation.
    Suppose that there exists $r$ as in the statement. Then $r|_{\partial X} = \textup{id}_{\partial X}$, so 
    $$0 = \partial(r_\#T) \stackrel{\eqref{eq:boundary_pushforward_commute}}{=} r_\#(\partial T) = (r_\#(\partial T))|_{\partial X} \stackrel{\eqref{eq:restriction_vs_pushforward}}{=} (r|_{\partial X})_\#(\partial T|_{\partial X}) = \partial T|_{\partial X} \ne 0,$$
    by Theorem \ref{theo:RCD-boundary-mass}. The first equality is true, since $\haus^n(\partial X) = 0$, so there are no non-trivial $n$-currents with locally finite mass on $\partial X$ because of Lemma \ref{lemma:mass-measure-abso-cont}. This contradiction concludes the proof.
\end{proof}

\section{Orientability of $\GCBA$ spaces}
\label{sec:GCBA}

We recall the basic properties of $\GCBA$ spaces. For more details we refer to \cite{BridsonHaefliger, LytchakNagano2019_GCBA, LN:JEMS22}. After that, we will study their orientations via currents.

\subsection{Basics on $\GCBA$ spaces}

Given $\kappa \in \setR$, we define $D_\kappa := \pi/\sqrt{\kappa}$ when $\kappa >0$ and $D_\kappa = \infty$ otherwise. A complete metric space $(X,\dist)$ is $\CAT(\kappa)$ if it is $D_\kappa$-geodesic (i.e. for every $x,y\in X$ with $\dist(x,y)<D_\kappa$ there exists a geodesic joining $x$ to $y$) and every geodesic triangle of perimeter at most $2D_\kappa$ is smaller than its comparison triangle in the simply connected, $2$-dimensional Riemannian manifold with sectional curvature $\kappa$. We refer to \cite{BridsonHaefliger} for basic properties of such spaces. 

A metric space $(X,\dist)$ is geodesically complete if every local geodesic $\gamma \colon [a,b] \to X$ admits an extension $\tilde{\gamma}\colon [a-\varepsilon,b+\varepsilon]\to X$, $\tilde{\gamma}|_{[a,b]} = \gamma$, which is a local geodesic.

A metric space $(X,\dist)$ is a $\GCBA$ space if it is separable, geodesically complete, locally compact and if for every point $x\in X$ there exist $\kappa_x\in \setR$ and $0< r_x \le D_{\kappa_x}$ such that $\overline{B}(x,r_x)$ is $\CAT(\kappa_x)$. The last property is generally called CBA: curvature bounded above. The G stands for "geodesically complete".
For every point $x$ in a $\GCBA$ space $(X,\dist)$, there exists an integer $n_x\in \setN$ such that the Hausdorff dimension of every small ball around $x$ is exactly $n_x$ (\cite[Theorem 1.2]{LytchakNagano2019_GCBA}). The quantity $n_x$ is called the \emph{dimension} of $x$. 

A $\GCBA$ space $(X,\dist)$ is said to be \emph{purely $n$-dimensional}, $n\in \setN \setminus\{0\}$, if the dimension of every point is $n$.
The Hausdorff $n$-measure $\haus^n$ on a purely $n$-dimensional $\GCBA$ space is  positive on open sets and locally finite (\cite[Theorem 1.4]{LytchakNagano2019_GCBA}). The {\em space of directions} at $x \in X$ is 
$$\Sigma_xX := \{\gamma\colon [0,\varepsilon] \to X\,:\, \gamma \text{ geodesic with } \gamma(0)=x\}/\sim,$$
where $\gamma \sim \gamma'$ if and only if $\angle_x(\gamma,\gamma') = 0$; for the definition of the {\em Alexandrov angle} between two geodesics, we refer to \cite[Definition I.1.12]{BridsonHaefliger}. 
The space $\Sigma_x X$ equipped with the distance $\angle_x(\cdot,\cdot)$ is a compact, geodesically complete, $\CAT(1)$ space of diameter $\pi$ (\cite[Corollary 5.9]{LytchakNagano2019_GCBA}), or empty if $x$ is isolated. 
The Euclidean cone over the space of directions, namely $T_x  X := \textup{Cone}(\Sigma_x X) = [0,+\infty)\times \Sigma_x X /_\sim$, $(0,v)\sim(0,w)$ for every $v,w\in \Sigma_x X$, with the corresponding metric $\dist_T$ (see \cite[Definiton I.5.6]{BridsonHaefliger}), is a proper, geodesically complete, $\CAT(0)$ space which is called the {\em tangent space} at $x$. Let $x\in X$, $k\in \setN$ and $\delta > 0$. A $k$-tuple $\{v_1,\ldots,v_k\} \subseteq \Sigma_x X$ is called \emph{$\delta$-spherical} if there exists another $k$-tuple $\{\bar{v}_1,\ldots,\bar{v}_k\} \subseteq \Sigma_x X$ such that:
\begin{itemize}
    \item for every $i=1,\ldots,k$, $\angle_x(v_i,w) + \angle_x(w,\bar{v}_i) < \pi + \delta$ for every $w\in \Sigma_x X$;
    \item for every $1\le i\ne j \le k$, $\angle_x(v_i,v_j), \angle_x(v_i,\bar{v}_j), \angle_x(\bar{v}_i,\bar{v}_j) < \pi/2 + \delta$.
\end{itemize}
A point $x\in X$ is said to be \emph{$(k,\delta)$-strained} if $\Sigma_x X$ has a $\delta$-spherical $k$-tuple. We define the set of $(k,\delta)$-\emph{strained points} as
$$X_{k,\delta} := \{x\in X\,:\, \text{$x$ is $(k,\delta)$-strained}\}.$$
In the next result, we express the structure theory developed in \cite{LytchakNagano2019_GCBA} with the terminology introduced in this paper. We will focus only on purely dimensional spaces.

\begin{theorem}
\label{theo:GCBA_basic_properties}
    Let $(X,\dist)$ be a purely $n$-dimensional $\GCBA$ space. Then $(X,\dist)$ is $n$-regular and $n$-rectifiable. Moreover, $\lambda_X \equiv 1$ and $\haus^n\rstr X$ is locally Ahlfors $n$-regular. Finally, $\mathcal{S}(X)$ is $(n-1)$-rectifiable and $\haus^{n-1}\rstr \mathcal{S}(X)$ is locally finite.
\end{theorem}
\begin{proof}
    Let $\delta \le 1/50n^2$. By \cite[Corollary 11.8]{LytchakNagano2019_GCBA}, see also the proof of \cite[Theorem 1.2]{LytchakNagano2019_GCBA},  the set $X_{n,\delta}$ is open and dense in $X$ and it is locally bi-Lipschitz homeomorphic to open subsets of $\setR^n$. Therefore, $(X_{n,\delta},\dist)$ is a metric $n$-manifold that satisfies the ETR property of Definition \ref{defin:ETR}. By \cite[Section 12.5]{LytchakNagano2019_GCBA}, $\haus^n\rstr X_{n,\delta}$ is locally Ahlfors $n$-regular, hence porous sets in $X_{n,\delta}$ have $\haus^n$-measure zero, by Lemma \ref{lemma: lower-ahlfors-no-porous}. By Lemma \ref{lemma:ETR=>Control_of_the_degree}, $X_{n,\delta}$ is a degree-regular metric $n$-manifold. By \cite[Corollary 11.8]{LytchakNagano2019_GCBA}, $\haus^{n}\rstr (X\setminus X_{n,\delta}) = 0$. Then $(\mathcal{M}(X), \dist)$ is a degree-regular metric $n$-manifold by Lemma \ref{lemma:degree-regularity_from_submanifold_to_manifold}. So $(X,\dist)$ is $n$-regular. The same properties of $X_{n,\delta}$ imply that $(X,\dist)$ is $n$-rectifiable.
    A combination of \cite[Theorem 1.2, Theorem 1.4 and Theorem 12.1]{LytchakNagano2019_GCBA} gives that the measured Gromov-Hausdorff tangent space at $\haus^n$-a.e. point of $X$ is isometric to $\setR^n$. Therefore, $\lambda_X(x) = 1$ for $\haus^n$-a.e. $x\in X$ by Lemma \ref{lemma:tangential_function}. The measure $\haus^n\rstr X$ is locally upper Ahlfors $n$-regular by \cite[Corollary 11.8]{LytchakNagano2019_GCBA}, while it is locally lower Ahlfors $n$-regular by \cite[Theorem 3.1]{CavallucciSambusetti2022}.

    Since $\mathcal{S}(X) = X\setminus \mathcal{M}(X) \subseteq X\setminus X_{n,\delta}$ and since the latter is $(n-1)$-rectifiable by \cite[Proposition 10.6]{LytchakNagano2019_GCBA}, then $\mathcal{S}(X)$ is $(n-1)$-rectifiable. The same reference shows that $\haus^{n-1}\rstr \mathcal{S}(X) \le \haus^{n-1}\rstr (X\setminus X_{n,\delta})$ is locally finite.
\end{proof}

\subsection{Wall points of $\GCBA$ spaces}

Let $(X,\dist)$ be a purely $n$-dimensional $\GCBA$ space. A point $x\in X$ is called a \emph{wall point} if $T_xX$ splits a $\setR^{n-1}$-factor but it is not equal to $\setR^n$. The set of wall points of $X$ will be denoted by $\mathcal{W}(X)$. Notice that $\mathcal{W}(X)$ has the same definition of the set $\mathcal{S}^{n-1}(X) \setminus \mathcal{S}^{n-2}(X)$ defined in the $\RCD$ setting.

\begin{remark}
    \label{rmk:GCBA_open_subsets}
    Let $(X,\dist)$ be a purely $n$-dimensional $\GCBA$ space and let $U\subseteq X$ be an open subset. Then $(U,\dist)$ is a purely $n$-dimensional $\GCBA$ space. Moreover, $\mathcal{M}(U) = U\cap \mathcal{M}(X)$ and $\mathcal{W}(U) = U\cap \mathcal{W}(X)$. Indeed, all these sets are defined via local or infinitesimal properties.
\end{remark}

Wall points should be thought of as the set of codimension $1$ singular points. This point of view can be formalized as follows. Given $\delta,\delta^* > 0$, we define
$$\mathcal{S}^{n-1}_{\delta,\delta^*}(X) := X_{n-1,\delta} \setminus X_{n,\delta^*},$$
that we call the \emph{codimension one $(\delta,\delta^*)$-singular set}. 
A fine geometric analysis of the sets $\mathcal{W}(X)$ and $\mathcal{S}^{n-1}_{\delta,\delta^*}(X)$ for suitable small $\delta,\delta^*$ has been done in \cite{Nagano2026}. We summarize here the properties we need, providing some improvements on the original statements. 

\begin{theorem}
\label{theo:GCBA_characterization_no_wall_points}
    Let $n\in \setN$. Then, there exists $\delta^* = \delta^*(n) > 0$ satisfying the following property.
    Let $(X,\dist)$ be a purely $n$-dimensional $\GCBA$ space. Then $\mathcal{W}(X)= \bigcap_{\delta > 0} \mathcal{S}^{n-1}_{\delta,\delta^*}(X)$ and $\haus^{n-1}(X \setminus (\mathcal{M}(X)\cup \mathcal{W}(X))) = 0$. 
    Moreover, the following properties are equivalent.
    \begin{itemize}
        \item[(i)] $\mathcal{W}(X) = \emptyset$.
        \item[(ii)] $X\setminus\mathcal{M}(X)$ has Hausdorff dimension at most $n-2$.
        \item[(iii)] $X \setminus X_{n,\delta^*}$ has Hausdorff dimension at most $n-2$.
    \end{itemize}
    Finally, $\mathcal{M}(X) \cap \mathcal{W}(X) = \emptyset$.
\end{theorem}

\begin{proof}
    Let $\delta^* = \delta^*(n) >0$ be such that \cite[Theorem 1.5 and Lemma 5.9]{Nagano2026} hold true. Applying \cite[Theorem 1.5]{Nagano2026} we get the equivalence between (i) and (iii). Moreover, (i) is equivalent to (ii) because of \cite[Theorem 1.2]{Nagano2026}. Applying \cite[Lemma 5.9]{Nagano2026} we conclude that $\mathcal{W}(X) \subseteq \mathcal{S}_{\delta,\delta^*}^{n-1}(X)$ for every $0<\delta <\delta^*$, hence $\mathcal{W}(X) \subseteq \bigcap_{\delta > 0} \mathcal{S}^{n-1}_{\delta,\delta^*}(X)$. If $x\in \bigcap_{\delta > 0} \mathcal{S}^{n-1}_{\delta,\delta^*}(X)$, then $\Sigma_x X$ admits a $\delta$-spherical $(n-1)$-tuple for every $\delta > 0$. By \cite[Proposition 6.6]{LytchakNagano2019_GCBA}, \(T_xX\) splits an \(\mathbb R^{n-1}\)-factor. Since \(x\notin X_{n,\delta_*}\), its tangent cannot be isometric to \(\mathbb R^n\). Hence \(x\in \mathcal{W}(X)\). This shows that $\mathcal{W}(X) = \bigcap_{\delta > 0} \mathcal{S}^{n-1}_{\delta,\delta^*}(X)$. Finally, for every $0<\delta \le \delta^*$, we have 
    $$X \setminus (\mathcal{M}(X) \cup \mathcal{S}^{n-1}_{\delta,\delta^*}(X)) \subseteq X \setminus (X_{n,\delta^*} \cup \mathcal{S}^{n-1}_{\delta,\delta^*}(X)) \subseteq X\setminus X_{n-1,\delta}.$$
    By \cite[Proposition 10.6]{LytchakNagano2019_GCBA} we have that
    $$\haus^{n-1}(X \setminus (\mathcal{M}(X) \cup \mathcal{S}^{n-1}_{\delta,\delta^*}(X))) \le \haus^{n-1}(X\setminus X_{n-1,\delta}) = 0.$$
    Therefore, 
    $$\haus^{n-1}(X \setminus (\mathcal{M}(X) \cup \mathcal{W}(X))) = \lim_{k\to \infty} \haus^{n-1}(X \setminus (\mathcal{M}(X) \cup \mathcal{S}^{n-1}_{1/k,\delta^*}(X))) = 0.$$
    The space $\mathcal{M}(X)$ is open in $X$, hence it is a purely $n$-dimensional $\GCBA$ space by Remark \ref{rmk:GCBA_open_subsets}. Applying the equivalence between (i) and (ii) to the space $(\mathcal{M}(X),\dist)$ we deduce that $\mathcal{M}(X) \cap \mathcal{W}(X) = \mathcal{W}(\mathcal{M}(X)) = \emptyset$.
\end{proof}

\subsection{Orientability}

We specialize the orientability theory via currents to purely $n$-dimensional $\GCBA$ spaces.

\begin{theorem}
\label{theo:characterization_orientability_GCBA}
    Let $(X,\dist)$ be a purely $n$-dimensional $\GCBA$ space. Then every orientation $T\in \mathcal{I}_{n,\loc}(X)$ satisfies $\norm{T}=\haus^n$,  $\partial T \in \bM_{n-1,\loc}(X)$, $T\in \bI_{n,\loc}(X)$ and $\norm{\partial T} \ll \haus^{n-1} \rstr \mathcal{W}(X)$. In particular, if $\mathcal{W}(X) = \emptyset$ then $\partial T = 0$.
\end{theorem}
\begin{proof}
    The space  $(X,\dist)$ is $n$-regular, $n$-rectifiable and $\lambda_X\equiv 1$ by Theorem \ref{theo:GCBA_basic_properties}. Therefore, by definition, $\norm{T}= \haus^n$. Theorem \ref{theo:GCBA_basic_properties} says that the hypotheses of Corollary \ref{cor:boundary_finite_mass_rectifiable} are satisfies, so $\partial T \in \bM_{n,\loc}(X)$ and $T\in \bI_{n,\loc}(X)$ for every $T\in \mathcal{I}_{n,\loc}(X)$ as in (i). Since $\partial T\in \bM_{n-1,\loc}(X)$ and $\spt(\partial T) \subseteq \mathcal{S}(X)$, then $\norm{\partial T} \ll \haus^{n-1}\rstr \mathcal{S}(X)$ by Lemma \ref{lemma:mass-measure-abso-cont}. Theorem \ref{theo:GCBA_characterization_no_wall_points} gives that
    $$0=\haus^{n-1}((X \setminus \mathcal{M}(X)) \cap (X\setminus \mathcal{W}(X))) = \haus^{n-1}(\mathcal{S}(X)\setminus \mathcal{W}(X)),$$
    hence $\norm{\partial T} \ll \haus^{n-1}\rstr \mathcal{W}(X)$. If $\mathcal{W}(X) = \emptyset$, then $\norm{\partial T} = 0$, hence $\partial T = 0$.    
\end{proof}

Example \ref{ex:n-regular_three_hemispheres} shows a purely $2$-dimensional, orientable, $\GCBA$ space such that the boundary of every orientation is not zero, and also a purely $2$-dimensional, orientable, $\GCBA$ space with an orientation with trivial boundary. In particular, the equivalent of Theorem \ref{theo:RCD-boundary-mass} for $\GCBA$ spaces, with $\mathcal{W}(X)$ replacing $\partial X$, cannot be true. Indeed, the space $\mathcal{W}(X)$ should not be thought as a boundary, since the space is locally geodesically complete.

\subsection{Local Poincaré inequality on locally uniquely orientable $\GCBA$ spaces}

All the pathological examples of Section \ref{sec:n-regular} are $\GCBA$ spaces. As explained in Section \ref{subsec:n-regular_loc_uniq_orientable}, the expected homological properties can be recovered for locally uniquely orientable $n$-regular spaces.

\begin{example}[Homology manifolds are locally uniquely orientable]
\label{ex:CBA_hom_mfds_are_loc_uniq_orientable}
Every separable $\CBA$ homology $n$-manifold is locally uniquely orientable. First of all, every homology $n$-manifold is locally compact by definition. If it has a $\CBA$ metric, then it is automatically geodesically complete, hence $\GCBA$ spaces, and purely $n$-dimensional (see \cite[Lemma 2.15]{CavallucciSambusetti2026}). Every homology $n$-manifold is locally orientable in the sense of homology manifolds by \cite[Theorem 16.15]{bredon-book}. Every submanifold of an orientable homology manifold is orientable, so $X$ is locally orientable as $n$-regular space. By \cite[Theorem 1.2]{LN:JEMS22}, $\mathcal{M}(X)$ is locally connected. Therefore, Corollary \ref{cor:n-regular_unique_orientability} implies that $X$ is locally uniquely orientable.
\end{example}

The following class contains examples that are not homology manifolds.

\begin{example}[A locally uniquely orientable $\CAT(1)$ space which is not a homology manifold]
\label{ex:GCBA_loc_uniq_orient_not_mfd}
    Let $\mathbb{T}^2 = \mathbb{S}^1_{2\pi} \times \mathbb{S}^1_{2\pi}$, where $\mathbb{S}^1_{2\pi}$ has length $2\pi$. Let $X = \mathbb{T}^2 \ast \mathbb{S}^0$ be the spherical suspension over it. Since $\mathbb{T}^2$ is $\CAT(1)$ then $X$ is $\CAT(1)$ as well by \cite[Corollary II.3.15]{BridsonHaefliger}. It is also compact and purely $3$-dimensional. It is not a homology manifold by \cite[Lemma 3.1]{LN:JEMS22}. Indeed, $H_1^\textup{sing}(\Sigma_{p_\pm} X) \cong H_1^\textup{sing}(\mathbb{T}^2) \cong \mathbb{Z}^2$, where $p_\pm$ are the two suspension points. Observe that $\mathcal{M}(X) = X \setminus \{p_\pm\}$. Corollary \ref{cor:n-regular_unique_orientability} shows that $X$ is locally uniquely orientable. 

    For every $n\ge 4$, by taking $X' = X \ast \mathbb{S}^{n-4}$ we construct a compact, locally uniquely orientable, purely $n$-dimensional $\CAT(1)$ space such that $\dim_H\mathcal{S}(X') = n-3$. These spaces are genuinely different from homology manifolds, since for them $\dim_H(\mathcal{S}(X)) = 0$ by \cite[Theorem 1.2]{LN:JEMS22}.
\end{example}

The top-degree homology of locally uniquely orientable $\GCBA$ spaces behaves like in the case of manifolds.

\begin{theorem}
\label{theo:GCBA_loc_uniq_orient_dichotomy}
    Let $(X,\dist)$ be a connected, purely $n$-dimensional $\GCBA$ \textup{LUO} space. Then $\mathcal{W}(X) = \emptyset$ and $\mathcal{M}(X)$ is connected. Moreover, one of the following two situations occurs.
    \begin{itemize}
        \item[(a)] $X$ is non-orientable and $H_n^\textup{IC}(X) = \{0\}$.
        \item[(b)] $X$ is uniquely orientable and $H_n^\textup{IC}(X) \cong \mathbb{Z}$. 
    \end{itemize}
\end{theorem}
\begin{proof}
    Suppose that $\mathcal{W}(X) \neq \emptyset$. By \cite[Theorem 1.1]{Nagano2026}, there exists an open subset $U\subseteq X$ homeomorphic to $\setR^{n-1}\times T_\ell$, where $T_\ell$ is the tree with one vertex $o$ and $\ell \ge 3$ infinite edges. By choosing $x = (p,o) \in U$, where $p\in \setR^{n-1}$, we have that $\mathcal{M}(X)\cap B(x,r)$ is not connected for every sufficiently small $r>0$. This contradicts the local unique orientability of $X$ at $x$, by Corollary \ref{cor:n-regular_unique_orientability}. This proves that $\mathcal{W}(X) = \emptyset$. Then $\haus^{n-1}(\mathcal{S}(X)) = 0$ by the computation in the proof of Theorem \ref{theo:characterization_orientability_GCBA}. 
    Then, the remaining statements follow by Theorem \ref{theo:n-regular_loc_uniq_orient_dichotomy}. 
\end{proof}

The following is an immediate consequence.
\begin{corollary}
    Let $(X,\dist)$ be a separable, $\CBA$, connected homology $n$-manifold. Then $\mathcal{W}(X) = \emptyset$. Moreover,
    \begin{itemize}
        \item[(a)] if $X$ is orientable as homology manifold then $H_n^{\textup{IC}}(X) \cong \mathbb{Z}$;
        \item[(b)] if $X$ is non-orientable as homology manifold then $H_n^{\textup{IC}}(X) \cong \{0\}$.
    \end{itemize}
\end{corollary}
\begin{proof}
    By Example \ref{ex:CBA_hom_mfds_are_loc_uniq_orientable}, $(X,\dist)$ is a locally unique orientable, purely $n$-dimensional $\GCBA$ space. Theorem \ref{theo:GCBA_loc_uniq_orient_dichotomy} implies that $\mathcal{W}(X) = \emptyset$. Notice that the fact that $\mathcal{W}(X) = \emptyset$ can also be deduced directly from the definition of wall points. By \cite[Theorem 1.2]{LN:JEMS22}, $\mathcal{M}(X)$ is locally connected and $(\mathcal{M}(X),\dist)$ is a length space. Arguing as in the proof of Proposition \ref{prop:orientability_from_submanifold_to_manifold_length}, we deduce that $X$ is orientable as homology manifold if and only if $\mathcal{M}(X)$ is orientable as topological manifold. The conclusion follows by Theorem \ref{theo:GCBA_loc_uniq_orient_dichotomy}. 
\end{proof}

By using the homological properties of locally uniquely orientable $\GCBA$ spaces, we prove that they are locally $\textup{PI}$-spaces. The proof
is conceptually related to the one of \cite{BassoMartiWenger25}. It applies also to spaces which are far from being homology manifolds, see Example \ref{ex:GCBA_loc_uniq_orient_not_mfd}.

\begin{theorem}\label{thm:local-PI-new}
    Let $(X,\dist)$ be a purely $n$-dimensional $\GCBA$ $\textup{LUO}$ space. Then $(X,\dist,\haus^n)$ is locally doubling and satisfies a local $1$-Poincaré inequality. More precisely, for every $x \in X$ there exist $r_x>0$ and $C_x>0$ such that for every Lipschitz function $f\colon X \to \setR$ we have
    $$
    \dashint_B \left\vert f - \dashint_B f \right\vert \, d\haus^n \le C_x r \dashint_{3 B} \lip_a f\,d\haus^n
    $$ 
    for every ball $B = B(y,r)$ with $y \in B(x,r_x)$ and $0<r<r_x$. 
\end{theorem}

Before giving the proof, we need some preparation. We use the following well-known relation between relative isoperimetric inequalities and Poincar\'e inequalities; see e.g. \cite{KorteLahti2014}. Recall that for a ball $B = B(x,r)$ in a metric space $(X,\dist)$ and $\lambda>0$ we write $\lambda B = B(x,\lambda r)$.

\begin{lemma}\label{lemma: rel-isoper-implies-poincare}
    Let $(X,\dist)$ be a locally compact metric space with a locally finite and locally doubling Borel measure $\meas$. Suppose that $\spt( \meas) = X$ and for every $x\in X$ there exist $r_x>0$ and $\lambda,C\geq 1$ such that for all $y\in B(x,r_x)$, each $0<R<r_x$ and every Borel set $E \subseteq B(x,r_x)$ we have
    \begin{equation}\label{eq: rel-isoper-inequality-statement-lemma}
        \frac{\min\bigl\{\meas(E\cap B), \meas(B\setminus E)\bigr\} }{\meas(B)}\leq  \frac{CR}{\meas(\lambda B)}  \liminf_{s \to 0}\frac{\meas((B(E,s)\setminus E)\cap \lambda B)}{s},
    \end{equation}
    where $B= B(y,R)$. Then, for every $x\in X$ there exist $r'_x>0$ and $C'\geq 1$ such that for all $y\in B(x,r'_x)$, each $0<R<r'_x$ and for every Lipschitz function $f\colon X \to \setR$ we have
    \begin{equation}\label{eq: poinc-inequality-statement-lemma}
        \dashint_B \left\vert f - \dashint_B f \right\vert \, d\meas \le C'R \dashint_{\lambda B} \lip_a f\,d\meas,
    \end{equation}
    where $B= B(y,R)$. 
\end{lemma}

\begin{proof}
   Fix $x\in X$ and let $r_x,\lambda,C$ be as in the
hypothesis. Choose $r'_x>0$ sufficiently small that
$4\lambda r'_x<r_x$ and
$\meas(B(x,(1+\lambda)r'_x))<\infty$.
Let $y\in B(x,r'_x)$, $0<R<r'_x$, and
$B=B(y,R)$.
Let $f\colon X\to\setR$ be Lipschitz.
    
    We first prove \eqref{eq: poinc-inequality-statement-lemma} under the additional assumption that $f$ is non-negative and $\spt(f)\subseteq B(x,r'_x)$.
 Set $\nu:=\meas\rstr\lambda B$ and for $s,t>0$ define
    \[
     f_s(z):=\sup_{w\in B(z,s)}f(w)\qquad  \text{and} \qquad E_t:=\{f>t\}.
    \]
    Since $B(E_t,s)=\{f_s>t\}$, the layer-cake formula gives
    \[
     \int_0^\infty
     \frac{\nu(B(E_t,s)\setminus E_t)}{s}\,dt
     =
     \int_X\frac{f_s-f}{s}\,d\nu.
    \]
    Moreover,
    \[
     0\leq\frac{f_s-f}{s}\leq\Lip(f),
     \qquad
     \limsup_{s\downarrow0}\frac{f_s(z)-f(z)}{s}
     \leq\lip_a f(z).
    \]
    Fatou's lemma on the left, and the upper-limit estimate on
    the right using finiteness of $\nu$, imply
    \[
     \int_0^\infty
     \liminf_{s\downarrow0}
     \frac{\nu(B(E_t,s)\setminus E_t)}{s}\,dt
     \leq\int_X\lip_a f\,d\nu.
    \]

Therefore, the relative isoperimetric inequality \eqref{eq: rel-isoper-inequality-statement-lemma} implies
    \begin{align*}
        \frac{R}{\meas(\lambda B)}\int_{\lambda B} \lip_a(f) \, d\meas &\geq \frac{R}{\meas(\lambda B)}\int_0^\infty
     \liminf_{s\downarrow0}
     \frac{\nu(B(E_t,s)\setminus E_t)}{s}\,dt
        \\
        &=\frac{R}{\meas(\lambda B)} \int_0^\infty \liminf_{s\downarrow 0}\frac{\meas((B(E_t,s)\setminus E_t)\cap \lambda B)}{s} \, dt
        \\
        &\geq \frac{1}{C} \int_0^\infty \frac{\min\bigl\{\meas(B \cap E_t), \meas(B\setminus E_t)\bigr\} }{\meas(B)}\, dt.
    \end{align*}
    A direct computation shows 
    $$2 \, \frac{\min\bigl\{\meas(B \cap E_t), \meas(B\setminus E_t)\bigr\} }{\meas(B)} \geq 2 \, \frac{\meas(B \cap E_t) \cdot \meas(B\setminus E_t)}{\meas(B)^2}=   \dashint_B \left| \mathbb{1}_{E_t} -  \dashint_B \mathbb{1}_{E_t}\right| \, d\meas$$
    
    for all $t>0$. Therefore, 
$$ \frac{2CR}{\meas(\lambda B)}
 \int_{\lambda B}\lip_a f\,d\meas
 \geq
 \int_0^\infty
 \dashint_B
 \left|
 \mathbb{1}_{E_t}
 -\dashint_B\mathbb{1}_{E_t}\,d\meas
 \right|
 \,d\meas\,dt.
$$
    Finally, by the layer-cake formula we have
\[
 f(z)-\dashint_B f \, d\meas
 =\int_0^\infty
   \bigl(\mathbb{1}_{E_t}(z)-\dashint_B\mathbb{1}_{E_t}\,d\meas\bigr)\,dt,
\]
and hence
\[
 \int_0^\infty
 \dashint_B\left|\mathbb{1}_{E_t}-\dashint_B\mathbb{1}_{E_t}\right|\,d\meas\,dt \geq \dashint_B\left|f-\dashint_B f\right|\,d\meas.
\]

Combining this with the previous inequality establishes \eqref{eq: poinc-inequality-statement-lemma} under the additional assumption that $f\geq 0$ and $\spt(f)\subseteq B(x,r'_x)$. Now, let $f\colon X \to \setR$ be any Lipschitz function with $\spt(f) \subseteq B(x,r'_x)$. Set $f^+=\max\{f,0\}$ and $f^-=\max\{-f,0\}$. Both functions are Lipschitz, non-negative and have support contained in $\spt(f) \subseteq B(x,r_x')$. Moreover,
\[
 |f-f_B|
 \leq |f^+-(f^+)_B|+|f^--(f^-)_B|
 \qquad \text{and} \qquad
 \lip_a f^\pm\leq \lip_a f.
\]
Applying the first part of the proof to $f^+$ and $f^-$ separately gives
$$\dashint_B \left\vert f - \dashint_B f \right\vert \, d\meas \le C'R \dashint_{\lambda B} \lip_a f\,d\meas,$$
where $C' = 4C$. Hence, we proved \eqref{eq: poinc-inequality-statement-lemma} for Lipschitz functions $f$ with support contained in $B(x,r'_x)$.

\medskip

Finally, let $f\colon X \to \setR$ be an arbitrary Lipschitz function. Choose a Lipschitz cutoff function $\chi$ with support contained in $B(x,r_x')$ and with $\chi\equiv1$ on $B(x,r_x'/2)$. Set $g \coloneqq \chi f$. Then $g$ is Lipschitz and its support is contained in $B(x,r'_x)$. Put
\[
 \rho_x:=\frac{r'_x}{3(1+\lambda)}.
\]
For $y\in B(x,\rho_x)$ and $0<R<\rho_x$,
\[
 B(y,\lambda R)\subseteq B(x,r'_x/3),
\]
so $\chi\equiv 1$ on a neighbourhood of $ B(y,\lambda R)$.
Consequently,
\[
 g=f\ \text{on }B(y,R),
 \qquad
 \lip_a g=\lip_a f\ \text{on }B(y,\lambda R).
\]
Applying the inequality already proved to $g$ therefore
gives \eqref{eq: poinc-inequality-statement-lemma} for $f$
whenever $y\in B(x,\rho_x)$ and $0<R<\rho_x$.
Renaming $\rho_x$ as $r'_x$ completes the proof,
with $C'=4C$.
\end{proof}

The main tools to establish a relative isoperimetric inequality in a purely $n$-dimensional $\GCBA$ $\textup{LUO}$ space are isoperimetric inequalities for integral currents.

\medskip

Let \(T\in \bI_{k}(X)\) be a cycle. We call \(S\in \bI_{k+1}(X)\) a \textit{minimal filling} of \(T\) if \(\partial S=T\) and \(\bM(S)\leq \bM(S')\) for every \(S'\in \bI_{k+1}(X)\) with \(\partial S'=T\). The next result guarantees the existence of minimal fillings in compact spaces with a $2$-convex bicombing. The definition of a convex bicombing, see \cite[Definition 2.1]{Wenger2005}, will not be used directly.

\begin{proposition}\label{thm:filling-cat(k)-spaces}
    Let $k\geq 0$ be an integer. Then there exists $D_k\geq 1$ such that the following holds. Let $(Y,\dist)$ be a compact metric space admitting a $2$-convex bicombing, and let $T\in\bI_k(Y)$ be a cycle. If $k=0$, assume additionally that $T(1)=0$. Then there exists a minimal filling $S\in\bI_{k+1}(Y)$ of $T$. Any such filling satisfies
    \begin{equation}\label{eq:filling-estis}
        \bM(S) \leq D_k \bM(T)^{\frac{k+1}{k}} \quad \text{ and } \quad \norm{S}(B(y, r))\geq D_k^{-k}r^{k+1}
    \end{equation}
    for all \(y\in \spt (S)\) and all $0\le r \le \dist(y, \spt (T))$. In particular, \(\spt (S) \subseteq \overline{B}(\spt(T), R)\), for $R=D_k\bM(S)^{\frac{1}{k+1}}$.
\end{proposition}

When $k=0$, the first inequality in \eqref{eq:filling-estis}
is omitted; the remaining conclusions still apply. 

\begin{proof}
    Since $Y$ is compact, the spaces $\bI_{k}(Y)$ coincide with the spaces of integral $k$-currents in the sense of Ambrosio-Kirchheim; see Section \ref{subsec:AK_currents}. We first show that there exists a filling $V \in \bI_{k+1}(Y)$ of $T$, that is, $\partial V = T$. By \cite[Corollary 1.4]{Wenger2005} the space $Y$ satisfies an isoperimetric inequality of Euclidean type for $\bI_m(Y)$, for all $m \ge 1$, with constant $C_m$ depending only on $m$. In particular, if $k\geq 1$, then there exists a filling $V \in \bI_{k+1}(Y)$ of $T$. On the other hand, if $k = 0$, then the integral current $T$ is a finite sum of integer-weighted atoms. Since $T(1)=0$, its positive and negative multiplicities can be paired. As $Y$ admits a $2$-convex bicombing we can join each pair by a Lipschitz curve. This gives an integral filling of $T$ of finite mass. The existence of a minimal filling $S\in \bI_{k+1}(Y)$ follows now from \cite[Theorem 5.2 and 8.5]{AmbrosioKirchheim00}; c.f. \cite[Section 10]{AmbrosioKirchheim00}. It remains to show that $S$ satisfies \eqref{eq:filling-estis}.
    
    We first treat the case $k=0$. By \cite[Theorem 5.3]{bonicatto-del-nin-2022}, there exists an at most countable family of Lipschitz
curves, each of which is injective or an injective loop, $\gamma^i \colon [0,1] \to Y$ such that
    \[
    S=\sum_{i \in \setN} S_i \quad \text{ and } \quad \bN(S)=\sum_{i\in \setN} \bN(S_i)
    \]
    where \(S_i=\gamma_{\#}^i\big(\br{\mathbb{1}_{[0,1]}}\big)\). The additivity of the normal mass implies
\[
 \|S\|=\sum_{i\in \setN}\|S_i\|,
 \qquad
 \|\partial S\|=\sum_{i\in \setN}\|\partial S_i\|.
\]
    Thus, there is no cancellation either in the currents or in their boundaries. Initially, the decomposition allows for injective curves or loops; however, the minimality of $S$ excludes the possibility of loops. Indeed, removing a loop would strictly decrease the mass without changing the boundary, and the remaining curves have endpoints in $\spt(T)$. Hence, each curve $\gamma^i$ connects two points of $\spt(T)$ with opposite signs. Additivity of the normal mass implies that there are only finitely many curves, since each $S_i$ contributes boundary mass \( \bM(\partial S_i) = 2\). Thus, every point of $\spt(S)$ lies on one of these curves. It follows that for each $y \in \spt (S)$ and every $0\le r\le\dist(y,\spt(T))$, the ball $B(y,r)$ contains portions of a curve $\gamma^i$ of length at least \(r\) on both sides of \(y\). As the curve $\gamma^i$ is injective we get 
    $$\|S\|(B(y,r)) \geq \haus^1(\textup{im}(\gamma^i) \cap B(y,r))\geq2r.$$

Now let $k\geq1$. Minimality and the isoperimetric
inequality give
\[
 \bM(S)\leq C_k\bM(T)^{\frac{k+1}{k}}.
\]
The standard slicing argument for mass-minimizing
fillings (see \cite[Theorem~10.6]{AmbrosioKirchheim00}
or the proof of \cite[Lemma~3.4]{Wenger2005}) gives
a constant $c_k>0$, depending only on $k$, such that
\[
 \norm{S}(B(y,r))\geq c_k r^{k+1}
\]
for every $y\in\spt(S)$ and
$0\leq r\leq\dist(y,\spt(T))$.
Choosing
\[
 D_k\geq\max\{1,C_k,c_k^{-1/k}\}
\]
gives both estimates in \eqref{eq:filling-estis}.
For $k=0$, the preceding argument gives the required
estimate with $D_0=1$. If $T=0$, minimality implies $S=0$.
Otherwise, applying the density estimate with
$r=\dist(y,\spt(T))$ gives
\[
 \dist(y,\spt(T))
 \leq D_k^{\frac{k}{k+1}}\bM(S)^{\frac1{k+1}}
 \leq D_k\bM(S)^{\frac1{k+1}},
\]
which proves the asserted support inclusion.

\end{proof}

We can now prove Theorem \ref{thm:local-PI-new}.

\begin{proof}[Proof of Theorem \ref{thm:local-PI-new}] 
 If $n=0$ then $X$ is discrete and there is nothing to prove. If $n=1$ then $X$ is locally isometric to a graph with locally finite valency. Theorem \ref{theo:GCBA_loc_uniq_orient_dichotomy} implies that $\mathcal{W}(X) = \emptyset$, hence $X$ has no branching points. This implies that $X$ is locally isometric to an interval, where the usual $1$-Poincaré inequality holds true.   

From now on, we assume $n\ge 2$. Fix $x\in X$ and let $r_x>0$ be sufficiently small such that $Y = \overline{B}(x,r_x)$ is compact and $\CAT(\kappa)$ for some $\kappa \in \setR$, $r_x < D_\kappa/2$, the ball $B(x,r_x)$ is orientable with orientation $T\in \bI_n(B(x,r_x))$ and there exists $A>0$ such that
    \begin{equation}\label{eq: isop-gcba-ahlfors-reg}
       A^{-1} r^n \leq \norm{T}(B(y,r)) = \haus^n(B(y,r)) \leq Ar^n
    \end{equation}
    for all $y\in B(x,r_x/4)$ and every $r \in (0,\frac{3}{4}r_x)$. This last property follows by Theorem \ref{theo:GCBA_basic_properties}. Combining Theorem \ref{theo:GCBA_characterization_no_wall_points} with Theorem  \ref{theo:GCBA_loc_uniq_orient_dichotomy} implies that $\haus^{n-1}(\mathcal{S}(X))=0$. Hence, by the slicing theory and replacing $T$ with the restriction of an orientation of a slightly larger ball to $Y$, we may suppose that $T$ has an extension $T^Y \in \bI_{n}(Y)$ with $\spt(\partial (T^Y)) \subseteq \{y \in X \colon \dist(y,x) = r_x\}$. We also choose $r_x$ so that
$$\haus^n(\{z\in X:\dist(x,z)=r_x\})=0.$$
This is possible for almost every radius, by local
finiteness of $\haus^n$. Consequently,
$\|T^Y\|=\haus^n\rstr Y$.
    From now on, we denote $T^Y$ again by $T$; its
restriction to $B(x,r_x)$ is the original orientation. Furthermore, it follows from \cite[Section 2.2]{CavallucciSambusetti2022} that $Y$ admits a $2$-convex bicombing. Hence, Theorem \ref{thm:filling-cat(k)-spaces} applies to $Y$ for all $k \geq 0$. We set \(D:=\max\{ D_{n-1}, D_{n-2}\}\), where  \(D_{n-1}, D_{n-2}\) denote the constants from the Theorem. 
    
    Fix a Borel set \(E\subseteq B(x,r_x/4)\) and let \(B=B(y, R)\) be an open ball of radius \(0< R <  r_x/4\) and center $y \in B(x,r_x/4)$. Let $E_s:= \{z \in Y\,:\, \dist(z,E) < s\}$ denote the open $s$-neighbourhood of $E$ in $Y$ and $\lambda B = B(y,\lambda R)$. We claim that
    \begin{equation}\label{eq:main-goal-section-10}
    \min\bigl\{\norm{T}\big(E\cap B\big), \norm{T}\big(B \setminus E\big) \bigr\} \leq C  \cdot \left(\liminf_{s \to 0}\frac{\norm{T}((E_s\setminus E)\cap 3B)}{s}\right)^{\frac{n}{n-1}}
    \end{equation}
    for some constant \(C\) depending only on the data of $n$ and $A$. Let \(\varepsilon>0\) be such that $D \cdot (2\varepsilon)^\frac{1}{n-1}<1$.
    We may  assume that \(\norm{T}(E\cap B)\), \(\norm{T}(B\setminus E)\neq 0\) as well as 
    \begin{equation}\label{eq:assumption-on-P}
    P\coloneqq \liminf_{s \to 0}\frac{\norm{T}((E_s\setminus E)\cap 3B)}{s}\leq  \varepsilon R^{n-1}.
    \end{equation}
    Otherwise, the Ahlfors regularity property \eqref{eq: isop-gcba-ahlfors-reg} implies \eqref{eq:main-goal-section-10} directly with $C = A \varepsilon^{-(n/(n-1))}$. Since $E_s\setminus E$ contains $\partial E\setminus E$ we necessarily have $\norm{T}\big((\partial E \setminus E)\cap 3B\big) = 0$. Hence, after replacing \(E\) with its closure, we may assume that \(E\) is compact. In particular, 
    \begin{equation}\label{eq: isoper-E_r-limit-measure}
        \norm{T}(B\setminus E_s) \xrightarrow{s\searrow 0} \norm{T}(B\setminus E) \qquad \text{and}\qquad \norm{T}(B\cap E_s) \xrightarrow{s\searrow 0} \norm{T}(B\cap E) .
    \end{equation}
    Let \(0 < \delta <\min\big\{\varepsilon R^{n-1}, 1\big\}\). The parameter \(\delta\) is introduced mainly to account for the possibility that \(P=0\). By definition, there exists \(s>0\) arbitrarily small such that
    \begin{equation*}
        \frac{\norm{T}((E_s\setminus E)\cap 3B)}{s}\leq P + \delta. 
    \end{equation*}
    Moreover, for almost every sufficiently small \(t>0\), we know from the slicing theory, Section \ref{sec: slicing}, that \(\partial(T\rstr E_t)\) is an integral \((n-1)\)-current in $Y$ whose support is contained in \(\{ z\in Y : \dist(z, E)=t\}\), and
    \[
    \int_0^s \norm{\partial(T\rstr E_t)}(3B)\, dt \leq \norm{T}\big((E_s \setminus E)\cap 3B\big).
    \]
    The inequality follows from \cite[Theorem 6.2.(3)]{Lang}.
    Therefore, there exists $t>0$ arbitrarily small such that 
    \(T\rstr E_t\in \bI_n(Y)\) with \(\spt ( \partial(T\rstr E_t))\subseteq \{ z\in Y : \dist(z, E)=t\}\), and
    \begin{equation}\label{eq:mass-partial-V}
    \norm{\partial(T\rstr E_t)}(3B) \leq  P+\delta.
    \end{equation}
    Using slicing again we find \(R'\in (2R, 3R)\) such that \(U \coloneqq \partial(T\rstr E_t)\rstr B(y, R')\in \bI_{n-1}(Y)\) and the support of the current \(\partial U\) is contained in \(S(y, R')=\{ z\in Y : \dist(z, y)=R' \}\).
    When $n=2$, the additional condition
$(\partial U)(1)=0$ holds automatically.
    By Theorem~\ref{thm:filling-cat(k)-spaces}, there exists a minimal filling \(V\in \bI_{n-1}(Y)\) of \(\partial U\) satisfying
    \[
 \spt(V)\subseteq
 \overline B\bigl(\spt(\partial V),
 D\bM(V)^{\frac1{n-1}}\bigr).
\]
    Since \(V\) is a minimal filling, \eqref{eq:mass-partial-V} implies
    \begin{equation}\label{eq: 1st-minimal-filling-supp}
        \bM(V) \leq \bM(U) \leq \norm{\partial(T\rstr E_t)}(3B)\leq P+\delta.
    \end{equation}
    This, together with \eqref{eq:assumption-on-P} and the definition of \(\delta\), yields 
    \begin{equation}\label{eq:support-of-U-is-close}
    \spt(V) \subseteq B(\spt (\partial V ), \varepsilon'R),
    \end{equation}
    where \(\varepsilon' = D\cdot(2 \varepsilon)^{\frac{1}{n-1}}\). As \(\spt (\partial V) = \spt(\partial U ) \subseteq S(y, R')\) and \(\varepsilon'<1\), we have
    \begin{equation}\label{eq:spt-of-U-does-not-meet-R-ball}
    \spt(V) \cap B(y, R)=\emptyset.
    \end{equation}
    Let $W=U-V \in \bI_{n-1}(Y)$. Notice that $\partial W = 0$ and $\bM(W)\leq 2( P+\delta)$. Again, by Theorem~\ref{thm:filling-cat(k)-spaces}, there exists a minimal filling \(S\in \bI_n(Y)\) of \(W\) satisfying \(\bM(S)\leq D \bM(W)^{\frac{n}{n-1}}\). Let $Q \coloneqq T\rstr E_t - S \in \bI_n(Y)$. Then,
    \begin{align*}
        \partial Q
        &= \partial(T\rstr E_t)-W\\
        &= \partial(T\rstr E_t)-U+V\\
        &= \partial(T\rstr E_t)
           \rstr\bigl(Y\setminus B(y,R')\bigr)+V.
    \end{align*}
    By \eqref{eq:spt-of-U-does-not-meet-R-ball} and our choice of $R'$ we have $\spt( \partial Q) \cap B = \emptyset.$ Recall that $\haus^{n-1}(\mathcal{S}(X)) = 0$. Thus, Proposition \ref{prop:n-regular_loc_uniq_conn_homology_of_conn_open_subsets} implies that $T|_{B}$ is an orientation of $B$ and there exists $k \in \mathbb{Z}$ such that $k \cdot T|_B = Q|_B$. By Corollary \ref{cor:extension_vs_restriction} we conclude that $k\cdot(T\rstr B) = Q\rstr B$.    
    Therefore,
    $$Q\rstr B =  k\cdot T\rstr B = k \cdot (T\rstr (E_t \cap B) + T\rstr (B\setminus E_t)).$$
    Thus,
    $$ (1-k) \cdot T\rstr (E_t \cap B) - k \cdot T\rstr (B\setminus E_t)= S\rstr B.$$
    In particular,
    $$ |1-k|\,\|T\|(E_t\cap B)  +|k|\,\|T\|(B\setminus E_t) \leq \mathbf M(S).$$
    Consequently,
    $$\min\bigl\{\norm{T}\big(E_t\cap B\big), \norm{T}\big(B \setminus E_t\big) \bigr\}  \leq \bM(S) \leq D \bM(W)^\frac{n}{n-1} \leq D[2(P+\delta)]^{\frac{n}{n-1}}.$$ 
    Since $t>0$ can be chosen to be arbitrarily small, \eqref{eq: isoper-E_r-limit-measure} implies
    $$\min\bigl\{\norm{T}\big(E\cap B\big), \norm{T}\big(B \setminus E\big) \bigr\}  \leq D [2( P+\delta)]^\frac{n}{n-1}.$$
    As $\delta>0$ was arbitrary, we conclude from the definition of $P$ that \eqref{eq:main-goal-section-10} holds with $C =  \max\{2^{n/(n-1)} D, A \varepsilon^{-(n/(n-1))}\}$.
    As $\haus^n\rstr Y = \norm{T}$, we also have
    \begin{equation}\label{eq: isoper-gcba-hn}
        \min\bigl\{\haus^n\big(E\cap B\big), \haus^n\big(B \setminus E\big) \bigr\} \leq C  \cdot \left(\liminf_{s \to 0}\frac{\haus^n((E_s\setminus E)\cap 3B)}{s}\right)^{\frac{n}{n-1}}
    \end{equation}
    for every Borel set \(E\subseteq B(x,r_x/4)\) and every open ball \(B=B(y, R)\) of radius \(0< R <  r_x/4\) and center $y \in B(x,r_x/4)$. Fix such $E$ and $B=B(y,R)$. Write $m=\min\{\mathscr H^n(E\cap B),\mathscr H^n(B\setminus E)\}$ and
    \[
      P=\liminf_{s\to 0}\frac{\mathscr H^n((E_s\setminus E)\cap 3B)}{s}.
    \]
    We may assume $P<\infty$.  We use \eqref{eq: isop-gcba-ahlfors-reg}, enlarging $A$ to be at least $1$ if necessary.
    Raising the inequality \eqref{eq: isoper-gcba-hn} to the power $\frac{n-1}{n}$ gives
    $m^{\frac{n-1}{n}}\leq C^{\frac{n-1}{n}}P$. Since
    $m\leq \mathscr H^n(B)\leq A R^n$, we have $m^{1/n}\leq A^{1/n}R$, and hence
    \[
      m=m^{\frac1n}\,m^{\frac{n-1}{n}}\leq A^{\frac1n}C^{\frac{n-1}{n}}\,R\,P.
    \]
    Moreover,
    $\mathscr H^n(3B)\leq A3^nR^n\leq A^2 3^n\mathscr H^n(B)$. Dividing by
    $\mathscr H^n(B)$, we obtain
    \[
      \frac{\min\{\mathscr H^n(E\cap B),\mathscr H^n(B\setminus E)\}}{\mathscr H^n(B)}
      \leq \frac{C'R}{\mathscr H^n(3B)}
      \liminf_{s\to 0}\frac{\mathscr H^n((E_s\setminus E)\cap 3B)}{s},
    \]
    where $C'=3^n A^{2+1/n}C^{(n-1)/n}$. It follows now from Lemma \ref{lemma: rel-isoper-implies-poincare} that $X$ satisfies a local $1$-Poincaré inequality as claimed.
\end{proof}

In particular, every $\CBA$ homology manifold satisfies a local $1$-Poincaré inequality. This result is new. The corresponding statement for topological manifold with a metric with curvature bounded above is known to hold true. The classical proof follows by Semmes (\cite{Semmes1996}). A more recent proof, at least for $n\neq 4$, follows by \cite[Corollary 1.6]{BassoMartiWenger25} and the comment below.

\begin{corollary}
     Let $(X,\dist)$ be a $\CBA$ homology $n$-manifold. Then $(X,\dist,\haus^n)$ is locally doubling and satisfies a local $1$-Poincaré inequality.
\end{corollary}

}

\bibliographystyle{plain}
\bibliography{OrientabilityBiblio}

\end{document}